\documentclass{amsart}

\numberwithin{equation}{section}

\usepackage{amssymb}
\usepackage{enumerate, xspace}
\usepackage{tikz}
\usepackage{marvosym} %monetary symbols, \EUR or \EURcr, \EyesDollar
\usepackage[sans]{dsfont}        %allows \mathds{E 1}, without sans is less heavy
\usepackage[colorlinks]{hyperref}

\newtheorem{thm}{Theorem}[section]
\newtheorem{lem}[thm]{Lemma}
\newtheorem{cor}[thm]{Corollary}

\newtheorem{prop}[thm]{Proposition}
\newtheorem*{conj}{Conjecture} 
\newtheorem{defin}{Definition}

\newtheorem{rem}[thm]{Remark}
\newtheorem*{notation}{Notation}
\newcommand\cA{{\mathcal A}}
\newcommand\cB{{\mathcal B}}

\newcommand\cF{{\mathcal F}}
\newcommand\cG{{\mathcal G}}
\newcommand\cH{{\mathcal H}}

\newcommand\cL{{\mathcal L}}
\newcommand\cO{{\mathcal O}}
\newcommand\cP{{\mathcal P}}

\newcommand\cM{{\mathcal M}}
\newcommand\cT{{\mathcal T}}
\newcommand\cV{{\mathcal V}}
\newcommand\cW{{\mathcal W}}
\newcommand\cZ{{\mathcal Z}}

\newcommand\bB{{\mathbb B}}
\newcommand\bC{{\mathbb C}}

\newcommand\bE{{\mathbb E}}
\newcommand\bF{{\mathbb F}}
\newcommand\bN{{\mathbb N}}
\newcommand\bM{{\mathbb M}}
\newcommand\bP{{\mathbb P}}
\newcommand\bQ{{\mathbb Q}}
\newcommand\bR{{\mathbb R}}

\newcommand\bV{{\mathbb V}}
\newcommand\bZ{{\mathbb Z}}

\newcommand\vf{\varphi}

\newcommand{\Const}{C_\#}
\newcommand\Id{{\mathds{1}}}

\newcommand{\Hl}{\widehat{\mathcal{L}}}
\newcommand\bomega{{\bar\omega}}
\newcommand{\bw}{\bar w}

\newcommand{\betavar}{\varpi}
\newcommand{\exalpha}{\varrho}
\newcommand{\essup}{{\rm ess\, sup}}
\newcommand{\ds}{{\delta_\star}}

\begin{document}

\title{Deterministic walks in random environment}
\author{Romain Aimino}
\address{Romain Aimino\\
Departamento de Matem\'atica\\
Faculdade de Ci\^encias da Universidade do Porto\\
Rua do Campo Alegre, 687, 4169-007 Porto, Portugal.}
\email{{\tt romain.aimino@fc.up.pt}}
\author{Carlangelo Liverani}
\address{Carlangelo Liverani\\
Dipartimento di Matematica\\
II Universit\`{a} di Roma (Tor Vergata)\\
Via della Ricerca Scientifica, 00133 Roma, Italy.}
\email{{\tt liverani@mat.uniroma2.it}}
%\date{\today. {\bf File: {\jobname}.tex.}}% eliminate in final version
\begin{abstract}
Motivated by the random Lorentz gas, we study deterministic walks in random environment and show that (in simple, yet relevant, cases) they can be reduced to a class of random walks in random environment where the jump probability depends (weakly) on the past. In addition, we prove few basic results (hopefully the germ of a general theory) on the latter, purely probabilistic, model.
\end{abstract}
\thanks{L.C. thanks Dmitry Dolgopyat and Marco Lenci for several discussions on random walks  through the years. Also we would like to thank Serge Troubetzkoy for pointing out the possibility that, in our setting, the ergodicity of the process as seen from the particle might hold under very weak extra assumptions. We also thank the anonymous referee for several helpful suggestions. This work was partially supported by the European Advanced Grant Macroscopic Laws and Dynamical Systems (MALADY) (ERC AdG 246953). L.C. acknowledges the MIUR Excellence Department Project awarded to the Department of Mathematics, University of Rome Tor Vergata, CUP E83C18000100006. This work started while R.A. was affiliated to Universit\`a di Roma (Tor Vergata). R.A. was partially supported by FCT fellowship SFRH/BPD/123630/2016, by FCT projects PTDC/MAT-CAL/3884/2014 and PTDC/MAT-PUR/28177/2017, with national funds, and by CMUP (UID/MAT/00144/2019), which is funded by FCT with national (MCTES) and European structural funds through the programs FEDER, under the partnership agreement PT2020.}
\maketitle
%%%PAPER
\section{Introduction}
The motion of a point particle among periodically distributed elastically reflecting convex bodies has been intensively studied for many years, both in the case of diluted obstacles and in the opposite case of high density. Diluted obstacles (the Boltzmann-Grad limit) can be treated by kinetic theory and lattice dynamics ideas (see \cite{MS11} and related work). The high density case has required more specific dynamical systems tools, starting with the seminal work of Bunimovich, Sinai and Chernov \cite{BS91} until the recent and much more precise results obtained in \cite{DSV08}. In particular, the latter result uses the new standard pairs and martingale problem techniques introduced in the field by Dolgopyat (see \cite{DeL15} for an elementary introduction to such ideas and references to the original works).

All the above deal with the periodic case, yet any material is expected to have defects. Hence, the study of obstacles distributed according to a random, translation invariant process is of paramount  importance. 

Unfortunately very few results are available in the non periodic case with the notable exception of the low density regime (Boltzmann-Grad limit), see \cite{marklof-14} and reference therein. For high obstacle density the only results are \cite {Lenci03, Lenci06}, where recurrence is proven  for special examples, and \cite{DSV09}, where the Central Limit Theorem is proven when the periodicity is broken only in a finite region (hence translation invariance is violated). 

However some basic results  hold in full generality: \cite{Conze99, Schmidt98} establish a criteria for recurrence and \cite{Lenci03} shows that recurrent systems are ergodic. For example, by the criteria in \cite{Conze99, Schmidt98}, the problems of recurrence and ergodicity are reduced, in the two dimensional case, to establishing a CLT, see \cite{Lenci06} for details. Also, establishing recurrence in the case of a one dimensional array of obstacles (tubes) is substantially simpler and has been obtained in \cite{CLS10, CLS10b, SLEC11, Lenci-Troubetzkoy}. Nevertheless, even in the simpler one dimensional situation, the study of rate of mixing and CLT is wide open.

Part of the difficulty in studying the above problems stems from the fact that, on the one hand, one needs non trivial results for the dynamical system and, on the other hand, one has to overcome the same  obstacles that exist in analysing the problems of random walks in random environment (see \cite{Zeitouni04} for a review on the subject). In particular, on the dynamical system side it seems necessary to establish some type of {\em memory loss}. In other words, one must show that the deterministic dynamics is akin to a random process.

In this article we aim at separating the above two difficulties, so that they can be (hopefully) solved independently. To this end we investigate more general models: a) deterministic random walks in random environment (see section \ref{sec:walk-det}); b) random walks (with memory) in random environment (see section \ref{sec:walk-rand}). The former include the random Lorentz gases; the latter include {\em persistent random walks} but allows for infinite memory.

These two classes of models are connected by the following conjecture: relevant classes of deterministic walks in random environment are equivalent to the above purely probabilistic models. Hence, e.g.,  establishing the CLT for the probabilistic model implies the CLT for the deterministic walk. 

%%%%%%%%%%%%%%%%%%%%%%%%%%%%%%%%%%
\subsection{Results and structure of the paper}
The paper is organised as follows: in the next section we discuss briefly the Lorentz gases establishing, in the process, several notations needed in the following. In Section \ref{sec:walk-det} we describe a general class of deterministic walks in random environment which encompass both the Lorentz gas and the example with a simpler dynamics that we will consider later. 

In Section \ref{sec:walk-rand} we describe the class of random systems. We establish properties under which one can prove the ergodicity of the random process as seen from the particle. In particular, if one would succeed in reducing the Lorentz gas to such a probabilistic model, then one would automatically recover all the known results. Of course, we believe that much more  would follow form such a reduction.  As a first step, in Section \ref{sec:sup-simp}, we prove the reversibility of the associated probabilistic model. In the same section we show that if one restricts the dynamics to a Markov one dimensional expanding map, then the dynamical part of the problem can be completely obliterated and one is led to known purely probabilistic models.

Next, in Section \ref{sec:model}, we explore more realistic (but still one dimensional) models. We explain under which conditions the deterministic dynamics can be obliterated yielding a probabilistic model as in Section \ref{sec:walk-rand}. 
In Section \ref{sec:ex-bis} we verify  such conditions for a class of one dimensional non-Markov expanding maps. In the following sections we prove  the statements of Section \ref{sec:model}.

\begin{notation} In the following we will use $\Const$ to designate a generic constant that depends only on the parameters or the assumptions of the considered model. The actual value of $\Const$ is immaterial. In particular, the value of $\Const$ can change from one occurrence to the next.
\end{notation}
%%%%%%%%%%%%%%%%%%%%%%%%%%%%%%%%%%%%%%
\section{The Random Lorentz gas}\label{sec:lorentz}
The random Lorentz gas consists of a distribution of convex, non overlapping, obstacles in $\bR^d$, $d\geq 2$ and independent point particles that move of free motion and collides elastically with the obstacles. If one describes the particle density by a distribution, then the problem is reduced to studying the motion of one particle with an initial distribution given by a measure. If the obstacle distribution is described by some probability measure, then the goal is to study the dynamics of the particle for almost all obstacle distributions.

Of course, the problem depends heavily on the obstacle distribution. Two reasonable assumptions are that the distribution enjoys some type of stationarity and ergodicity with respect to some subgroup of the space translations.
Also a key issue is the existence or not of trajectories that can spend an unboundedly long time without experiencing any collision. Given the many possibilities, let us restrict to small perturbations of a periodic array of discs on $\bZ^2$. Note however that similar examples can be considered for $d=1$ (Lorentz tubes \cite{CLS10, CLS10b, SLEC11}) and $d>2$ . 

Following \cite{Lenci06} we start with an exactly periodic distribution of discs of radius $R$ centred on the square  lattice $\bZ^2$, such a periodic array divides naturally $\bR^2$ in cells. We assume that $R<1/2$, which implies that the obstacles do not overlap. We require that trajectories that are non horizontal nor vertical will eventually collide with one of the obstacle. 

\begin{lem} In a periodic Lorentz gas with obstacles made by discs of radius $R$ centred on the square  lattice $\bZ^2$
 if $R> \frac 1{2\sqrt 2}$, then the only trajectories that never experience a collision are either horizontal or vertical.
 \end{lem}
 
\begin{proof}
By periodicity, we can reduce the motion to a motion on the torus $\bR^2/\bZ^2$ with just one obstacle of radius $R$ at its center. 
Also we can limit ourselves to trajectories of the type $\xi+(1,\omega)t\mod 1$ with $\omega\in (0,1]$ since the other cases can be obtained by symmetry. Note that a flow on the torus in the direction $v_\omega:=(1,\omega)$ induces the rotation $f(s)=s+\omega\mod 1$ on the Poincar\'e section $\{(0, s)\}_{s\in [0,1]}$.  In addition the shadow, in the direction $v_\omega$, of the obstacle on the Poincar\'e section is a segment of length $\frac{2 R\sqrt{1+\omega^2}}{\omega}$. Accordingly, the trajectory can avoid collisions only if the rotation on the section always avoids some interval of such a length. It follows $\omega\not\in\bQ$. On the other hand if $\omega=\frac pq$, $p,q\in\bN$ relatively prime, then $f$ has all periodic orbits. Let us consider first the case $p<q$, then $q=ap+b$ for some $a,b\in \bN$, $b<p$. Accordingly,
\[
1-f^a(0)=1-a\frac pq=\frac{\frac bp}{a+\frac bp}\leq \frac 1{a+1}\leq \frac 12.
\]
Thus the maximal gap in the trajectory of $f$ is of length $1/2$ and hence if $\frac{2 R\sqrt{1+\omega^2}}{\omega}>\frac 12$ all trajectories will collide with the center obstacle. This is implied by $R> \frac 1{4\sqrt 2}$. We are left with the case $\omega=1$. In this case, if $R> \frac1 {2\sqrt 2}$, then the shadow of the obstacle is larger than 1, thus it covers all the Poincar\'e section, hence the claim. 
\end{proof}

In the following we assume $R>\frac{1}{2\sqrt 2}$.

Next, in each cell we assume that there is another disc of radius $r$ with center in a $\delta$-neighborhood of the center of the cell. We assume that $r+\delta+R<1/\sqrt 2$, which imply that the obstacles do not overlap. Moreover,  we assume $R> \frac 1{2\sqrt 2}$ which implies that, in absence of the center obstacle, the only trajectories that never collide are either horizontal or vertical. Note that the above implies $r+\delta< \frac {1}{2\sqrt 2}$. It then suffices to ask $r-\delta+R>\frac 12$ to ensure that the horizontal and vertical trajectories collide with the center obstacle, hence establishing that the array has the finite horizon property. 
The locations of the central obstacle in different cells is an i.i.d. random variable.
\begin{figure}[ht]\
\begin{minipage}{.48 \linewidth} 
\hspace{.5cm}	
\begin{tikzpicture}[scale=0.25]
\fill[gray!20!white] (0,0) circle (1.5);% 1
\draw (0,0) circle (1.5);
\fill[gray!20!white] (4,0) circle (1.5);% 
\draw (4,0) circle (1.5);
\fill[gray!20!white] (8,0) circle (1.5);% 
\draw (8,0) circle (1.5);
%%%
\fill[gray!20!white] (0,-4) circle (1.5);% 1
\draw (0,-4) circle (1.5);
\fill[gray!20!white] (4,-4) circle (1.5);% 
\draw (4,-4) circle (1.5);
\fill[gray!20!white] (8,-4) circle (1.5);% 
\draw (8,-4) circle (1.5);
%%%%
\fill[gray!20!white] (0,-8) circle (1.5);% 1
\draw (0,-8) circle (1.5);
\fill[gray!20!white] (4,-8) circle (1.5);% 
\draw (4,-8) circle (1.5);
\fill[gray!20!white] (8,-8) circle (1.5);% 
\draw (8,-8) circle (1.5);
%%%%%%%%%central
\fill[gray!20!white] (2.2,-2) circle (1);% central disc
\draw (2.2,-2) circle (1);
\fill[gray!20!white] (6,-1.8) circle (1);% central disc
\draw (6,-1.8) circle (1);
\fill[gray!20!white] (2,-6.3) circle (1);% central disc
\draw (2,-6.3) circle (1);
\fill[gray!20!white] (6,-5.7) circle (1);% central disc
\draw (6,-5.7) circle (1);
\end{tikzpicture}
\end{minipage}
%%%%%%%%%%%%%%%%%%%%%%%%%%%%%%
\begin{minipage}{.48 \linewidth}
\hspace{1cm}
\begin{tikzpicture}[scale=0.33]
\draw[dashed] (0,0)--(8,0);
\draw [dashed](0,0)--(0,-8);
\draw[dashed](0,-8)--(8,-8);
\draw[dashed](8,-8)--(8,0);
\draw[very thick, dashed] (3,0)--(5,0);
\draw[very thick, dashed] (3,-8)--(5,-8);
\draw[very thick, dashed] (0,-3)--(0,-5);
\draw[very thick, dashed] (8,-5)--(8,-3);
\fill[gray!20!white] (0,0)--(0,-3) arc (-90:0:3)--cycle;%first disc
\draw[very thick] (0,-3) arc (-90:0:3);
\fill[gray!20!white] (0,-8)--(3,-8) arc (0:90:3)--cycle;%second disc
\draw[very thick] (3,-8) arc (0:90:3);
\fill[gray!20!white] (8,-8)--(5,-8) arc (180:90:3)--cycle;%third disc
\draw[very thick] (5,-8) arc (180:90:3);
\fill[gray!20!white] (8,0)--(5,0) arc (180:270:3)--cycle;%third disc
\draw[very thick] (5,0) arc (180:270:3);
\fill[gray!20!white] (4.3,-4.4) circle (2);% central disc
\draw (4.3,-4.4) circle (2);
\node at (1.7,-1.2) {$C_2$};
\node at (6.3,-1.2) {$C_1$};
\node at (1.7,-6.2) {$C_3$};
\node at (6.4,-6.2) {$C_4$};
\node at (4,-4) {$C_5$};
\node at (8.7,-4) {$B_1$};
\node at (-.8,-4) {$B_2$};
\node at (4,-.5) {$B_3$};
\node at (4,-7.5) {$B_4$};
\end{tikzpicture}
\end{minipage}
\\ \vskip.1cm
{ Fig 1.a {\it Random obstacle configuration}\hskip1cm Fig 1.b {\it Poincar\'e section (in bold)}\hspace{1cm}}
\end{figure}
See Figure 1.a for an illustration.

The above is a reasonable model for a material with a periodic structure and random impurities. 
Let $\|\omega_z\|\leq\delta$ be the displacement of the center obstacle from the center of the cell $z \in \bZ^2$. The obstacles are distributed according to a product measure $\bP$ on $\Omega=\{\omega\in\bR^2\;:\;\|\omega\|\leq\delta\}^{\bZ^2}$. On $\Omega$ are defined the translations $\tau_z:\Omega\to\Omega$ by $z\in\bZ^2$: for all $\bomega\in\Omega$ and $w\in\bZ^2$
\[
\tau_z(\bomega)_w=(\bar \omega)_{w+z}.
\]
Next, consider a particle with position $q\in\bR^2$ and velocity $p$, $\|p\|=1$, moving among the obstacles with elastic collisions. 

The dynamics is deterministic but we consider stochastic initial conditions: Assume w.l.o.g. that the particle starts from the zero cell with positions and velocity described by a smooth distribution $h_0$.

Let $(q(t),p(t))$ be the position and velocity at time $t$. 
Natural questions are: does it exist an asymptotic velocity (Law of large numbers)
\begin{equation}\label{eq:lln}
V=\lim_{t\to\infty} t^{-1} [q(t)-q(0)].
\end{equation}
If so, does the Central Limit Theorem holds? That is, does
\begin{equation}\label{eq:clt}
\frac 1{\sqrt t}[q(t)-q(0)-Vt]
\end{equation}
converge in law to a Gaussian Random Variable $\bP$ a.s. (quenched CLT).
\subsection{Poincar\'e section} Consider a single cell (see Figure 1.b) and let $C=\cup_{i=1}^4 C_i$. $C_i$ being one of the bold arcs centered at the corners of the box, and  let $G_{(0,0)}=C\times [-\pi/2,\pi/2]$;  $[-\pi/2,\pi/2]$ being the angle that the post-collisional velocity forms with the external normal. 
Next, let $\{B_i\}_{i=1}^4$ be the segments constituting the boundary of the box not contained in the obstacles  (see Figure 1.b) and set $G_{(1,0)}=B_1\times [-\pi/2,\pi/2]$,  $G_{(-1,0)}=B_2\times [-\pi/2,\pi/2]$, $G_{(0,1)}=B_3\times [-\pi/2,\pi/2]$ and $G_{(0,-1)}=B_4\times [-\pi/2,\pi/2]$.
Next, let  $\cW= \{(0,0), (\pm 1, 0), (0, \pm 1) \}$, $\cB=\cup_{w\in\cW}G_w$ and consider the phase space $\bB=\cB\times \bZ^2$. This is a Poincar\'e section. 

For each $\bomega$, given a particle at $(x,z)\in\bB$, we follow its motion till it hits $\cB$ again, this defines a map $f_{\bomega_z}(x)$. The  Poincar\'e map $\bF_{\bomega}:\bB\to\bB$ associated to the billiard flow is then defined by
\begin{equation}\label{eq:poinc-map}
\begin{split}
&\bF_{\bomega}(x,z)=(f_{\bomega_{z+e(\bomega_z, x)}}(x),z+e(\bomega_z,x))\\
&e(\bomega_z,x)=\sum_{w \in \cW} \Id_{G_w}(x) w,
\end{split}
\end{equation}
where $\Id_B$ is the characteristic function of the set $B$. 

\subsection{The random process}
For each initial density $h_0$ and $\bomega\in\Omega$ we have thus defined the random variables\footnote{Using the probabilistic usage we will often suppress the $\bomega$ dependency, when this does not create confusion.}
\[
( x(\bomega,n),z(\bomega, n))=\bF^{n}_{\bomega}(x,z).
\]

Let $\bM=\{(z(n))\in (\bZ^2)^{\bN}\;: \; z(0) = 0, \; z(n+1)-z(n)\in\cW\}$ be the path space and consider the random process on $\Omega_\star=\Omega\times\bM$ defined as follows. For each $\bomega\in\Omega$ define the conditional probability
\begin{equation}\label{eq:proba}
\bP_\star(\{z(1),\dots, z(n)\}\;|\;\bomega)=\int_\cB  \prod_{k=0}^{n-1} \Id_{G_{w(k)}}(x(\bomega, k)) h_0(x)dx,
\end{equation}
where $w(k)=z(k+1)-z(k)$. While $\bomega$ is distributed according to $\bP$.

Next, we introduce the transfer operators\footnote{In the present case the set on which we take the sum consists of only one element and the determinant of the Jacobian of the $f_\omega$ is always one.}
\[
\cL_{\bomega, z,w}\vf(x)=\sum_{\{y\;:\; f_{\bomega_{z+w}}(y)=x\}}|\det \partial_y f_{\bomega_{z+w}} (y)|^{-1}\Id_{G_{w}}(y)\vf (y).
\]
Changing variable repeatedly, yields
\begin{equation}\label{eq:transfer-op-prob}
\bP_\star(\{z(1),\dots, z(n)\}\;|\;\bomega)=\int_\cB \cL_{\bomega,z(n-1),w(n-1)}\cdots\cL_{\bomega, z(0),w(0)} h_0.
\end{equation}
Thus the measure $\bP_\star$ can be expressed as products of transfer operators. 
The process $\bP_\star$ is, in general, not Markov, however we conjecture:
\begin{conj}
Under some {\em appropriate technical conditions} on $\bP$ and $h_0$, there exist $\nu \in (0,1)$ and $\Const >0$ such that for all $\bomega\in\Omega$
\[
\left|\mathbb{P}_\star( z(n) \; | \; z^n, \bomega) - \mathbb{P}_\star( \widehat{z}_m(n-m) \; | \; \widehat{z}^{n-m}, \tau_{z(m)} \bomega) \right| \leq \Const \nu^{n-m},
\]
where $z^k=z(1), \ldots, z(k-1)$, $\widehat{z}^k=\widehat{z}_m(1), \ldots, \widehat{z}_m(k-1)$ with $\widehat{z}_m(k) = z(m+k) - z(m)$.
\end{conj} 
If the above were true, then one could reduce the study of the dynamical system to a purely probabilistic problem, albeit not an easy one. See however Section \ref{sec:walk-rand} for the beginning of a theory in some related cases. 

Next, we introduce a large class of dynamical systems, generalising the Lorentz models. Then, we will prove the above conjecture for simple, and not so simple, dynamical systems, giving an idea of what {\em  some appropriate technical conditions} might mean (see Theorem \ref{thm:main}).

%%%%%%%%%%%%%%%%%%%%%%%%%%%%%%%%%%%%%%%%%%%%%%%%%
\section{Deterministic walks in random environment}\label{sec:walk-det}
As already mentioned there is an extreme scarcity of results pertaining the high density random Lorentz Gas. It is then sensible to consider simpler models in which one can start to solve some of the outstanding difficulties. To this end, following Lenci \cite{Lenci06}, it is convenient to see the Lorentz gas as a special case of a more general class of models: {\em deterministic walks in random environment}.  Even for such models very few results exist. Exceptions are \cite{Little16} in which a zero-one law for systems with local dynamics which are markovian, but deterministic, is established and \cite{SS09} which considers statistical properties for a related (simplified) model with local dynamics consisting of expanding linear maps of the circle or hyperbolic toral automorphisms.

The model can be stated in rather general terms, for simplicity let us restrict to the case of the $\bZ^d$ lattice with bounded jumps $\cW \subset \bZ^d$, $\#\cW<\infty$,\footnote{ That is, only jumps $w\in\cW$ are allowed.} and local dynamics which live all on the same phase space $\cM$.

Consider the set $\cA=\{(f_\alpha, \cM, \cG_\alpha)\}_{\alpha\in A}$, where $(A, \mathcal{S})$ is a measurable space\footnote{For most of our concrete applications, $(A, \mathcal{S})$ will simply be a finite set with the discrete $\sigma$-algebra.}, $f_\alpha:\cM\to\cM$ are maps such that $(\alpha, x) \in A \times \cM \mapsto f_\alpha(x) \in \cM$ is measurable, and $\cG_\alpha=\{G_{\alpha,w}\}_{w\in\cW}$ are partitions of $\cM$ such that, for each $w \in \cW$, the map $(\alpha, x) \mapsto \Id_{G_{\alpha, w}}(x)$ is measurable. The environment is described by the probability space $\Omega=A^{\bZ^d}$ equipped with a translation invariant probability $\bP$. Also we assume that all the maps $f_\alpha : \cM \to \cM$ are nonsingular with respect to some reference measure $m$ on $\cM$. Then, for each realisation $\bomega\in\Omega$ we can define the dynamics $\bF_{\bomega}(\cdot,\cdot):\cM\times\bZ^d\to \cM\times\bZ^d$:
\[
\begin{split} \label{eq:def_dwrd}
&\bF_{\bomega}(x,z)=(f_{\bomega_{z+ e(\bomega_z, x)}}(x), z+e(\bomega_z,x))\\
&e(\alpha,x)=\sum_{w\in\cW}\Id_{G_{\alpha,w}}(x) w ,
\end{split}
\]
and $(x(n),z(n))=\bF_{\bomega}^n(x(0),z(0))$.
The randomness at fixed environment rests in the initial condition: $z(0)=z_0$ while $x(0)$ is distributed according to some probability measure $\mu$ absolutely continuous with respect to $m$. We will assume, w.l.o.g., $z_0 = 0$. Then the path $(z(n))_{n\in\bN}$ belongs to $\bM=\{(z_n)\in(\bZ^d)^\bN\;:\; z_0 = 0, z_{n+1}-z_n\in\cW\}$, which we call the space of {\em admissible paths}, and $\bP_\star$ is the law of the resulting process on $\Omega_\star=\Omega\times \bM$.

Each $f_\alpha$ admits a transfer operator $\cL_{f_\alpha} : L^1(m) \to L^1(m)$ defined by
\begin{equation}\label{eq:transfer}
\int_{\cM} (\cL_{f_\alpha} \phi) \psi dm = \int_{\cM} \phi \cdot \psi \circ f_\alpha dm
\end{equation}
for all $\phi \in L^1(m)$ and $\psi \in L^{\infty}(m)$.

Let $h_0$ be the density of the initial condition ($d \mu = h_0 dm$). Then, setting $w(n)=z(n+1)-z(n)$, a repeated use of \eqref{eq:transfer} in \eqref{eq:proba} yields

\begin{equation}\label{eq:iterate_transfer}
\mathbb{P}_\star(z(1), \ldots, z(n) \; | \; \bomega) = \int_{\cM} \cL_{\bomega, z(n-1), w(n-1)} \cdots \cL_{\bomega, z(0), w(0)} h_0 dm
\end{equation}
for each $\bomega \in \Omega$ and each admissible path $(z(n)) \in \bM$, and with 
\begin{equation}\label{eq:TOdef}
\cL_{\bomega, z, w}(\phi) = \cL_{f_{\bomega_{z + w}}}(\Id_{G_{\bomega_{z}, w}} \phi).
\end{equation}

\subsection{The point of view of the particle}\label{sec:environ}
An important process associated to our deterministic walk is the process of the environment as seen from the particle. This is the dynamical system defined on $\Omega\times \cM$ by
\[
\cF(\bomega,x)=(\tau_{e(\bomega_0, x)}\bomega, f_{\bomega_{e(\bomega_0, x)}} (x)).
\]
This is known to be a fruitful point of view, in particular if one knows the invariant measure. As noted by Lenci \cite{Lenci03, Lenci06}, in an important subclass of deterministic random walks the invariant measure can be trivially computed.
\begin{lem}\label{lem:povop} 
If all the maps $f_\alpha$ have the same invariant measure $\lambda$ and the set $\cG_\alpha$ is deterministic (i.e., it does not depend on $\alpha$), then the probability measure $\bP_0=\bP\times \lambda$ is invariant  for the map $\cF$.
\end{lem}
\begin{proof}
Let $\bE_0$ be the expectation with respect to $\bP_0$. Since the set $\cG_\alpha$ is deterministic, we can write $e(\alpha, x) = e(x)$ for all $\alpha \in A$ and $x \in \cM$. For each bounded measurable function $\vf$ we have
\[
\begin{split}
\bE_0(\vf\circ \cF)&= \int \varphi(\tau_{e(x)} \bomega, f_{\bomega_{e(x)}}(x)) \bP(d \bomega) \lambda(dx) \\ 
&= \int \varphi(\tau_{e(x)} \bomega, f_{(\tau_{e(x)} \bomega)_0}(x)) \bP(d \bomega) \lambda(dx) \\ 
&= \int \varphi(\bomega, f_{\bomega_0}(x)) \bP(d \bomega) \lambda(dx) \\
&= \int \varphi(\bomega, x) \bP(d \bomega) \lambda(dx) 
=\bE_0(\vf),
\end{split}
\]
where we have used first the invariance of $\bP$ with respect to the translations and then the invariance of $\lambda$ with respect to the maps $f_\alpha$.
\end{proof}
We have obtained a dynamical system with a finite measure.
\begin{lem}\label{lem:zero-velocity} In the hypotheses of Lemma \ref{lem:povop} the limit 
\[
V=\lim_{n\to\infty}\frac 1n z(n) 
\]
exists $\bP_0$ a.s.. Moreover, if $(\Omega\times\cM,\cF,\bP_0)$ is ergodic, then $V=\lambda(e)$.
\end{lem}
\begin{proof}
Setting $(\bomega(n),x(n))=\cF^n(\bomega,x)$ yields $\frac 1nz(n)=\frac 1n\sum_{k=0}^{n-1}e(x(k))$.
Hence the existence of the limit follows from Birkhoff ergodic theorem for the dynamical system $(\Omega\times\cM,\cF,\bP_0)$ applied to $\vf(\bomega,x) = e(\bomega_0, x)=e(x)$.  If the system is ergodic, then the limit equals the average of $\vf$ with respected to $\bP_0$ which is equal to $\bE_0(\varphi) = \lambda(e)$. 
\end{proof}
As an application we have the following Lemma (which is implicit in \cite{CLS10}).
\begin{lem} For the Lorentz gas described in section \ref{sec:lorentz},  with $r+\frac{R}{\sqrt 2}<\frac 12$  and $\delta$ small enough, we have $V=0$.\footnote{ The conditions are certainly not optimal, however to obtain the result for a larger set of parameters entails a more refined analysis of the geometry of the trajectories and such an analysis exceeds our present goals.}
\end{lem}
\begin{proof}
We apply \cite[Theorem 5.4]{Lenci06} to prove that $(\Omega \times \cM, \cF,\bP_0)$ is ergodic. Then, Lemma \ref{lem:zero-velocity} implies that $V=\lambda\left(\sum_{w \in \cW} w G_w\right)$ which is zero due to the fact that $\lambda$ is invariant for the change $p\to -p$.

To verify the hypotheses of  \cite[Theorem 5.4]{Lenci06} it suffices to prove that there are non singular trajectories that go from the interior of any $B_i$ to the interior of any other $B_j$ (here and in the following we use the notation of Figure 1.b) for $\delta=0$, the result will then follow by the continuity of the non singular trajectories with respect to the positions of the obstacles. Let us start with a trajectory connecting $B_2$ to $B_1$. 

Consider the trajectory $\ell_h$ starting from $(0,h)$, $h\geq r$, with  direction $(1,0)$.\footnote{ Here we have chosen coordinates in which the center of $C_5$ is $(0,0)$.} 
If $h=1/2-R/\sqrt 2=:h_0$ then the trajectory reflects on the obstacle $C_1$ at the point $(1/2-R/\sqrt 2, 1/2-R/\sqrt 2)$ and, after reflection, has the direction $(0,-1)$. Note that if $r+\frac{R}{\sqrt 2}<\frac 12$, then such trajectory misses $C_5$ and the next collision is with $C_4$ while the previous is with $C_2$. Accordingly, if we increase $h$ the velocity, after the collision with $C_1$, rotates clockwise till, for some $h_->h_0$, the trajectory is tangent to $C_5$, let $v_-:=v(\theta_-)$, with $v(\theta)=(\cos\theta,\sin\theta)$ and  $\theta_-\in(-\pi, -\pi/2)$, be the velocity at tangency. On the other hand for  $h=1/2$ the trajectory is periodic between $C_2$ and $C_1$ and if we decrease $h$ the velocity, after reflection with $C_1$, rotates counterclockwise. Thus there exists $h_+$ such that the trajectory, after reflection with $C_1$ is tangent to the upper side of  $C_5$, let $v_+:=v(\theta_+)$, $\theta_+\in(\pi,  3\pi/2)$, be the velocity at tangency. It follows that the trajectories we are considering, for $h\in (h_-,h_+)$, will first collide with $C_1$ and then with $C_5$. In addition, the velocity after collision will go continuously from $v_+$ to $v_-$ while the collision point will belong to an interval $I$ on the boundary of $C_5$ delimited by the points $r(\sin\theta_+,-\cos\theta_+)$ and $r(\sin\theta_-,-\cos\theta_-)$. Moreover, in this open interval the velocity after colliding with $C_5$ cannot belong to $\{v_+$, $v_-\}$ since $v_+$ and $v_-$, at the points in $I$, point toward the interior of $C_5$. Also, near $v_+$, the post collisional velocity with $C_5$ rotates clockwise. Hence there exists an interval $(\bar h_-,\bar h_+)\subset (h_-,h_+)$ such that the angle $\theta$ of the velocity $v(\theta)$, after colliding with $C_5$, varies continuously from $\pi/2$ to  $- \pi/2$. Thus,
for $h\in (\bar h_-,\bar h_+)$, the intersection of the trajectory with the line $\{(\frac 12, t)\}_{t\in\bR}$ goes continuously from $+\infty$ to  $- \infty$. Accordingly, there exists a value for which the trajectory crosses the center of $B_1$ without any further collision. By symmetry such a trajectory connects the interiors of $B_2$ to $B_1$, and the same holds for $\delta$ small enough.

Analogously there are trajectories connecting $B_3$ and $B_4$. 

Let us now construct trajectories that connect $B_2$ to $B_3$. 
The diagonal trajectory tangent to $C_5$ is given by 
$\frac{1}{\sqrt 2}(-r,r)+t(1,1)$. Such a line intersects the boundary $\{(-1/2, s)\}_{s\in[-1/2,1/2]}$ of the box at the point $(-1/2, -1/2+r \sqrt 2)$ hence either it enters in $B_2$, if $R <r \sqrt 2$, or it collides with $C_3$, if $1/2>R > r \sqrt 2$. 

In the latter case we can consider all the trajectories that collide at $(-r,r)$ symmetrically with respect to the diagonal. When the angle of the trajectory with the normal to $C_5$ goes from $\pi/2$ to $0$ the trajectory goes from colliding with $C_3$ to colliding with $C_2$, hence for intermediate angles there must be trajectories that enter $B_2$. By symmetry such trajectories connect $B_2$ and $B_3$ and will survive for small, but positive, $\delta$. The same argument provides trajectories that connect $B_3, B_1$; $B_1, B_4$ and $B_4, B_2$. We have thus verified the hypothesis of  \cite[Theorem 5.4]{Lenci06}. 
\end{proof}
%%%%%%%%%%%%%%%%%%%%%%%%%%%%%%%%%%%%%%%%%%%%%%%%%%%%%%%
\section{Gibbs random walks in random environment}\label{sec:walk-rand}
We do not expect the process described by $\bP_\star$ to be Markov, yet we expect that the jump rates have a weak dependence of the past, provided the maps are strongly chaotic. To be more precise, we conjecture that the process is a {\em random walk in random environment with weak memory}. In the probabilistic literature random walks with a finite memory are called {\em persistent}, here we expect the memory to be infinite although depending weakly on the far past, like a potential of a Gibbs measure. Let us specify exactly what we mean by this.

Let $(A, \mathcal{S})$ be a measurable space and $\cW \subset \bZ^d$, $2 \le \sharp \cW < \infty$. Consider the measurable space $\Omega=A^{\bZ^d}$ and a translation invariant, ergodic, probability distribution $\bP$ that describes the distribution of the environments $\bomega\in\Omega$. For each $n\in \bN$ and $(w_0, \ldots, w_{n-1}) \in \cW^n$, assume that are given compatible probabilities $p(\bomega, n, \cdot )$ on $\cW^n$,  i.e. $$\sum_{(w_0, \ldots, w_{n-1}) \in  \cW^n} p(\bomega, n, w_0 \ldots w_{n-1})=1,$$ and $$p(\bomega, n, w_0 \ldots w_{n-1}) = \sum_{w \in \cW} p(\bomega, n+1, w_0 \ldots w_{n-1} w)$$ for all $\bomega \in \Omega$, $n\ge0$ and $(w_0, \ldots, w_{n-1}) \in \cW^n$. Assume also that all the maps $\bomega \mapsto p(\bomega, n, w_0 \ldots w_{n-1})$ are measurable.
We have then for each $\bomega \in \Omega$ a probability measure ${\bP}^{\bomega}$ on the space $\cW^{\bN}$ by Kolmogorov extension theorem. By the monotone class theorem, the map $G \mapsto {\bP}^{\bomega}(G)$ is measurable for any measurable set $G \subset \cW^\bN$, and we can thus define a probability measure ${\bP}_\star$ on $\Omega \times \cW^{\bN}$ by 
\[
\bP_\star (d \bomega, d \bw) = \mathbb{P}(d \omega) \mathbb{P}^{\bomega}(d \bw).
\]

\begin{rem}\label{rem:measures} The measure $\bP_\star$ can be naturally identified with a measure on the space $\Omega_\star = \Omega \times \bM$, where $\bM = \{ (z_n) \in (\bZ^d)^{\bN} \, : \, z_0 = 0, \, z_{n+1} - z_n \in \cW, \, \forall n \}$ is the space of admissible paths starting at 0, since there is a 1-to-1 correspondence between elements of $\cW^n$ and admissible paths of length $n$, via the relations $w_k = z_{k+1} - z_k$.
\end{rem}

\subsection{ The weak memory requirement}
We find convenient, although not strictly necessary, to require the following assumption that ensures that all admissible paths have positive probability:
\begin{description} 
\item[{\bf (Pos)}] for $\bP$-a.e. $\bomega \in \Omega$, for all $n \ge 0$ and all $\bw \in \cW^\bN$, $$p(\bomega, n, \bw_0 \ldots \bw_{n-1})>0.$$
\end{description}

Let $\bP_{\bomega}( \bw_n \; | \; \bw_0 \ldots \bw_{n-1})$ be the conditional probability $\frac{p(\bomega, n+1, \bw_0 \ldots \bw_{n-1} \bw_n)}{p(\bomega, n, \bw_0 \ldots \bw{n-1})}$.

The weak memory requirement  is made precise by the following:
\begin{description}
\item[{\bf (Exp)}] there exist $C_*>0$ and $\nu \in (0,1)$ such that for $\bP$-a.e. $\bomega \in \Omega$, all $n > m \ge 0$ and all $\bw \in \cW^\bN$
\begin{equation}\label{eq:gibbs}
\left|  \bP_{\bomega}(\bw_n \; | \; \bw_0 \ldots \bw_{n-1}) - \bP_{\tau_{z_{m}} \bomega} ( \bw_n \; | \; \bw_m \ldots \bw_{n-1}) \right| \leq C_* \nu^{n-m},
\end{equation}
where $z_m = \sum_{k=0}^{m-1} \bw_k$.
\end{description}

\begin{rem} This is the property we have conjectured to be true for the random Lorentz gas at the end of Section \ref{sec:lorentz}.
\end{rem}

Note that the above condition implies loss of memory:
\begin{lem}\label{lem:old-forget}
If $\bP_\star$ satisfies \eqref{eq:gibbs}, then (using the notation of Remark \ref{rem:measures}),
\[
\left|\bP_\star( z_{n+1} \;|\; z_{n-m+1},\dots, z_n, \bomega)-\bP_\star( z_{n+1} \;|\; z_1,\dots,  z_n, \bomega)\right|\leq \Const \nu^{m}.
\]
\end{lem}
\begin{proof}
Let $\cP(z_{n-m+1})$ be the set of admissible paths $(w_0, \ldots, w_{n-m})$ that arrive in $z_{n-m+1}$, i.e. $\sum_{i=0}^{n-m} w_i = z_{n-m+1}$, and set $w_i = z_{i+1} - z_i$ for $i=n-m+1, \ldots, n$.
Then we have, using {\bf (Exp)}:
\[
\begin{aligned}
&\bP_\star( z_{n+1} \;|\; z_{n-m+1},\dots, z_n, \bomega)  = \sum_{w \in \cP(z_{n-m+1})} \frac{p(\bomega, n+1, w_0 \ldots w_n)}{\bP_\star(z_{n-m+1}, \ldots, z_{n} \; | \; \bomega)} \\
& = \sum_{w \in \cP(z_{n-m+1})} \bP_{\bomega}( w_n \; | \; w_0 \ldots w_{n-1}) \frac{p(\bomega, n, w_0 \ldots w_{n-1})}{\bP_\star(z_{n-m+1}, \ldots, z_{n} \; | \; \bomega)} \\
& = \bP_{\tau_{z_{n-m+1}} \bomega}(w_n \; | \; w_{n-m+1} \ldots w_{n-1}) + \cO(\nu^m),
\end{aligned}
\]
and 
\[
\begin{split}
\bP_\star( z_{n+1} \;|\; z_1,\dots,  z_n, \bomega) &= \bP_{\bomega}(w_n \; | \; w_0 \ldots w_{n-1})\\
 &= \bP_{\tau_{z_{n-m+1}} \bomega}(w_n \; | \; w_{n-m+1} \ldots w_{n-1}) + \cO(\nu^m).
\end{split}
\]

\end{proof}

\subsection{The point of view of the particle and three further assumptions}\label{sec:environ2}
One can define the process of the environment as seen from the particle. It is given by the dynamical system on the space $\Omega_\star = \Omega \times \cW^{\bN}$ defined by
\[
\cF_\star(\bomega, \bw) = (\tau_{\bw_0} \bomega, \tau_\star \bw),
\]
where $\tau_\star : \cW^\bN \to \cW^\bN$ is the unilateral shift.
Note that, in general, $\bP_\star$ is not invariant for $\cF_\star$.
Next, let us show that some easy properties of the Markov case persist in the present context, under reasonable extra conditions.

\begin{rem}\label{rem:povotp}
Note that in the probabilistic literature, see for instance \cite{Zeitouni04}, it is more usual to consider the random process $\bomega(n) = \tau_{z(n)} \bomega$ on $\overline{\Omega} = \Omega^{\bN}$, and its law $\overline{\bP}$ when $\bomega(0)$ is distributed according to $\mathbb{P}$ (as we did in Section \ref{sec:environ}). When the set of periodic environments has probability 0, these two points of view are equivalent, since the map which associates to each $(\bomega, \bw) \in \Omega_\star$ the corresponding sequence $(\bomega(n)) \in \overline{\Omega}$ is invertible almost everywhere and realizes a conjugacy between the two dynamical systems $(\Omega_\star, \cF_\star, \bP_\star)$ and $(\overline{\Omega}, \bar{\tau}, \overline{\bP})$, where $\bar{\tau}$ is the shift on $\overline{\Omega}$. The same comments hold also for the definition given in Section \ref{sec:environ}. See also Remark \ref{rem:environ2} for further comments.
\end{rem}

We are interested in the asymptotic properties for $z_n$. Note that if we define $\varphi(\bomega, \bw) = \bw_0$, then we have $z_n = \sum_{k=0}^{n-1} \varphi \circ \cF_\star^k$.

The next assumption will imply the existence of a $\cF_\star$-invariant measure:

\begin{description}
\item[{\bf (Abs)}] There exists $C_0>0$ such that for $\bP$-a.e. $\bomega \in \Omega$, all $n \ge 0$, all $k \ge 1$ and all $\bw \in \cW^\bN$,
\[
C_0^{-1} \le \sum_{(w_1, \ldots, w_k) \in \cW^k} \frac{p(\tau_{-(w_1 + \ldots + w_k)} \bomega, n+k, w_1 \ldots w_k \bw_0 \ldots \bw_{n-1})}{p(\bomega, n, \bw_0 \ldots \bw_{n-1})} \le C_0.
\]
\end{description}
To legitimate this assumption, we prove two relevant facts.

\begin{lem} \label{lem:abs_cond} If there exists a probability measure $\bQ_\star$ equivalent to $\bP_\star$, invariant for $\cF_\star$ and such that $c^{-1} \le \frac{d\bQ_\star}{d\bP_\star} \le  c$ for some $c>0$, then {\bf (Abs)} holds with $C_0=c^2$. In particular, {\bf (Abs)} holds if $\bP_\star$ is $\cF_\star$-invariant.
\end{lem}

\begin{proof} For all measurable set $B \subset \Omega$ and all cylinder $[\bw_0 \ldots \bw_{n-1}] \subset \cW^\bN$, the preimage $\cF_\star^{-k}(B \times [\bw_0, \ldots, \bw_{n-1}])$ is equal to the disjoint union $$\bigcup_{(w^1, \ldots, w^k) \in \cW^k} \tau_{-(w^1 + \ldots + w^k)}(B) \times [w^1 \ldots w^k \bw_0 \ldots \bw_{n-1}].$$
Let $B' = B \times [\bw_0 \ldots \bw_{n-1}]$. By definition of $\bP_\star$, we have
\[
\bP_\star(B') = \int_B p(\bomega, n, \bw_0 \ldots \bw_{n-1}) \bP(d \bomega)
\]
and
\[
\begin{aligned}
\bP_\star(\cF_\star^{-k}B') & = \sum_{(w^1, \ldots, w^k) \in \cW^k} \int_{\tau_{-(w^1 + \ldots + w^k)}B} p (\bomega, n+k, w^1 \ldots w^k \bw_0 \ldots \bw_{n-1}) \bP(d \bomega) \\
& = \int_B \sum_{(w^1, \ldots, w^k) \in \cW^k} p(\tau_{w^1 + \ldots + w^k} \bomega, n+k, w^1 \ldots w^k \bw_0 \ldots \bw_{n-1}) \bP(d \bomega)
\end{aligned}
\]
thanks to the translation invariance of $\bP$.

Since $\bP_\star(B') \le c \bQ_\star(B') = c \bQ_\star(\cF_\star^{-k} B') \le c^2 \bP_\star( \cF_\star^{-k} B')$, we have
\[
\begin{aligned}
\int_B p(\bomega, n, &\bw_0 \ldots \bw_{n-1}) \bP(d \bomega) \\
& \le c^2 \int_B \sum_{(w^1, \ldots, w^k) \in \cW^k} p(\tau_{w^1 + \ldots + w^k} \bomega, n+k, w^1 \ldots w^k \bw_0 \ldots \bw_{n-1}) \bP(d \bomega),
\end{aligned}
\]
and similarly,
\[
\begin{aligned}
\int_B \sum_{(w^1, \ldots, w^k) \in \cW^k} p(\tau_{w^1 + \ldots + w^k} \bomega, n+k, &w^1 \ldots w^k \bw_0 \ldots \bw_{n-1}) \bP(d \bomega)\\
& \le c^2 \int_B p(\bomega, n, \bw_0 \ldots \bw_{n-1}) \bP(d \bomega).
\end{aligned}
\]
Since the set of cylinders is countable, this proves the lemma.
\end{proof}
We say that the process is {\em reversible} if $\cW$ is symmetric (i.e. $-w \in \cW$ for all $w \in \cW$) and, for $\bP$-a.e. $\bomega \in \Omega$, all $n \ge 0$ and all $\bw \in \cW^\bN$:
\[
p(\bomega, n, w_0 \ldots w_{n-1})=p(\tau_{(w_0+\dots+w_{n-1})}\bomega, n, -w_{n-1}, \ldots, - w_{0}).
\]
 Note that this definition of reversibility is more akin to the one used in Dynamical Systems than the one used in Markov processes (the self-adjointness of the generator). However it does have relevant implications.
\begin{lem}\label{lem:rev} If the process is reversible, then $\bP_\star$ is $\cF_\star$-invariant. In particular, {\bf (Abs)} is verified, by Lemma \ref{lem:abs_cond}.
\end{lem}
\begin{proof}
It suffices to consider a measurable set $B'$ of the form $B' = B \times [\bw_0, \ldots, \bw_{n-1}]$. We have:
\[
\begin{split}
\bP_\star(\cF_\star^{-1} B') & = \bP_\star \left( \bigcup_{w \in \cW} \tau_{-w}( B) \times [w, \bw_0, \ldots, \bw_{n-1}] \right) \\
&=\sum_{w \in \cW} \int_{\tau_{-w}(B)} p(\bomega, n+1, w, \bw_0, \ldots, \bw_{n-1}) \bP(d \bomega) \\
&=\sum_{w \in \cW} \int_B p(\tau_{-w} \bomega, n+1, w, \bw_0, \ldots, \bw_{n-1}) \bP(d \bomega) \\
&=\sum_{w \in \cW} \int_B p( \tau_{(\bw_0 + \ldots + \bw_{n-1}) }\bomega, n+1, -\bw_{n-1}, \ldots,- \bw_0,- w) \bP(d \bomega) \\
&=\int_B p(\tau_{(\bw_0 + \ldots + \bw_{n-1})} \bomega, n, -\bw_{n-1}, \ldots,- \bw_0) \bP(d \bomega)
\end{split}
\]
\[
\begin{split}
&\phantom{\bP_\star(\cF_\star^{-1} B') }= \int_B p(\bomega, n, \bw_0 \ldots \bw_{n-1}) \bP(d \bomega) = \bP_\star(B'). 
\end{split}
\]
\end{proof}

\begin{rem}\label{rem:toostrong}
Lemma \ref{lem:abs_cond} suggests that ({\bf Abs}) may be too strong. Yet, there are simple models (e.g. Sinai walk, see Example 1 in Section \ref{sec:mark-ex}) for which there does not exist an invariant probability measure absolutely continuous with respect to $\bP_\star$. Hence some condition is necessary.
\end{rem}

As common for random walks, we require an ellipticity assumption:

\begin{description}
\item[{\bf (Ell)}] There exist $\gamma_0 >0$ and $n_\star \ge 0$ such that for $\bP_\star$-a.e. $\bomega \in \Omega$, all $n \ge n_\star$ and all $\bw \in \cW^\bN$,
\[
\bP_{\bomega}(\bw_n \; | \; \bw_0 \ldots \bw_{n-1}) \ge \gamma_0.
\]
\end{description}

Next, we state an ergodicity assumption on the probability measure $\bP$:

\begin{description}
\item[{\bf (Pro)}]  Let $\cV(\cW)=\{z\in\bZ^d\;:\; z=w_0+\cdots+w_{n-1}, n\in\bN, w_i\in\cW\}$. We assume that $G=\cV(\cW)$ is an additive group and that $\bP$ is ergodic with respect to the action of $G$.
\end{description}
\begin{rem}
It might not be necessary to assume that $\cV(\cW)$ is a group, but we will not pursue this direction, as the examples we have in mind satisfies the above assumption: e.g., $\cV(\cW)$ is an additive group whenever $\cW$ is symmetric. We leave to the interested reader possible weakening of property ({\bf Pro}). We nevertheless mention that a closer look at the proof of Theorem \ref{thm:gibbs_walk_erg} reveals that it remains valid  if $\bP$ is ergodic for each translation $\tau_z$, $z \in \bZ^d$, $z\neq 0$ (for instance if $\bP$ is mixing when $d=1$, or if $\bP$ is i.i.d. when $d>1$), without any extra assumption on $\cW$.
\end{rem}

%%%%%%%%%%%%%%%%%%%%%%%%%%%%%%%%%%%%%%%%%%

\subsection{A few basic results}\label{sec:ergodicity}\ \\
The above assumptions are justified by the following Theorem. The rest of the section is devoted to its proof.
\begin{thm} \label{thm:gibbs_walk_erg} Suppose that the conditions {\bf (Pos)}, {\bf (Exp)}, {\bf (Abs)}, {\bf (Ell)} and {\bf (Pro)} hold. Then there exists a unique $\cF_\star$-invariant probability measure $\bQ_\star$ equivalent to $\bP_\star$ and the dynamical system $(\Omega_\star, \cF_\star, \bQ_\star)$ is  ergodic. In particular, we have 
\begin{equation} \label{eqn:drift}
\lim_{n \to \infty} \frac{z_n}{n} = V, 
 \end{equation}
$\bP_\star$-a.s., with $V = \int_{\Omega_\star} \varphi \, d \bQ_\star \in \bR^d$.\footnote{ Recall that $\varphi(\bomega, \bw) = \bw_0$.}
\end{thm}

\begin{rem} \label{rem:zero-velocity} If $\bP_\star$ is invariant, then 
\[
V = \int_{\Omega_\star} \varphi \, d\bP_\star = \sum_{w \in \cW} w \int_\Omega p(\bomega, 1, w) \, \bP(d \bomega).
\]
\end{rem}

As a simple but important consequence, we deduce the recurrence in 1-d:

\begin{cor} \label{cor:recurrence} Under the conditions of the above theorem, if $d =1$ and $V=0$, then the process $(z_n)$ is recurrent: $z_n= 0$ infinitely often, $\bP_\star$-a.s..
\end{cor}

\begin{proof}[{\bf\emph{Proof of Corollary \ref{cor:recurrence}}}] We refer to \cite{schmidt2006recurrence} for a nice survey on the recurrence of cocycles. Since $(\Omega_\star, \cF_\star, \bQ_\star)$ is ergodic by Theorem \ref{thm:gibbs_walk_erg} when $d=1$, the walk is recurrent if $V= 0$, see \cite{Atkinson76} or \cite[Theorem 3]{schmidt2006recurrence}. Note that this result for recurrence of cocycles is stated for invertible dynamical systems, but it can be extended to non-invertible systems using the natural extension, see for instance \cite[Appendix B]{Lenci08}.
\end{proof}

\begin{rem}
When $d=2$ and $V=0$, if $(z_n)$ satisfies an annealed central limit theorem, i.e. if $\frac{z_n}{\sqrt{n}}$ converges in law to a Gaussian distribution under the probability measure $\bQ_\star$, then the process $(z_n)$ is recurrent by the results of Conze \cite{Conze99} and Schmidt \cite{Schmidt98}.
\end{rem}

From now on and till the end of the section we will assume conditions {\bf (Pos)}, {\bf (Exp)}, {\bf (Abs)}, {\bf (Ell)} and {\bf (Pro)} if not explicitly stated otherwise.

To prove Theorem \ref{thm:gibbs_walk_erg}, we will analyze the properties of the transfer operator $\cL_\star$ associated to $\cF_\star$ with respect to $\bP_\star$. More precisely, we will show that the operator $\cL_\star$ enjoys some regularization properties on a space of H\"older functions. We first define the usual separation time on $\cW^\bN$ by 
\[
s(\bw, \bw') = \inf \{ n \ge 0 \, : \, \bw_n \neq \bw'_n \},
\]
and for $0 < \nu < 1$, the metric $d_\nu(\bw, \bw') = \nu^{s(\bw, \bw')}$ on $\cW^\bN$. 

For a measurable function $f : \Omega \times \cW^\bN \to \bC$, we set:
\[
\|f \|_\infty = \underset{\bomega \in \Omega} \essup \sup_{\bw \in \cW^\bN} |f(\bomega, \bw)|,
\]
\[
| f |_\nu = \underset{\bomega \in \Omega}\essup \sup_{\bw \neq \bw'} \frac{|f(\bomega, \bw) - f(\bomega, \bw')|}{d_\nu(\bw, \bw')},
\]
and define 
\[
\cH_\infty = \{f : \Omega \times \cW^\bN \to \bC \, : \, \|f \|_\infty < \infty \},
\]
\[
\cH_\nu = \{ f : \Omega \times \cW^\bN \to \bC \, : \, \|f \|_\nu := \|f\|_\infty + | f |_\nu < \infty \}.
\]
The space $\cH_\nu$ is a Banach algebra.
The following result about density of $\cH_\nu$ is based on very classical ideas, but we include it here for completeness:
\begin{lem} \label{lem:htheta_dense}
For any function $\varphi \in L^1(\bP_\star)$, there exists $(\varphi_\epsilon)_\epsilon \subset \cH_\nu$ such that $\varphi_\epsilon \to \varphi$ in $L^1(\bP_\star)$. Moreover, if $\varphi$ is bounded, $(\varphi_\epsilon)_\epsilon$ can be chosen such that $\sup_\epsilon \| \varphi_\epsilon \|_\infty < \infty$.
\end{lem}

\begin{proof}
We first consider the case where $\varphi = \mathds{1}_B$ is the indicator function of a measurable set $B \subset \Omega_\star$.
We endow $\Omega_\star$ with the metric $d_\star((\bomega, \bw), (\bomega', \bw')) = d_{\Omega}(\bomega, \bomega') + d_\nu(\bw, \bw')$, where $d_\Omega$ is any metric defining the product topology on $\Omega$. The metric $d_\star$ defines the product topology on $\Omega_\star$. For any open set $O \subset \Omega_\star$, we define 
\[
\varphi_{k, O}(\bomega, \bw)  = \min \{ k \, d_\star( (\bomega, \bw), \Omega_\star \setminus O), 1\}.
\]
We clearly have $0 \le \varphi_{k, O} \le \varphi_{k+1, O} \le \mathds{1}_O \le 1$, and $\lim_k \varphi_{k,O}(\bomega, \bw) = \mathds{1}_O(\bomega, \bw)$ for all $(\bomega, \bw) \in \Omega_\star$. The function $\varphi_{k,O}$ is clearly lipschitzian with respect to the metric $d_\star$, which also implies that $ \varphi_{k,O} \in \cH_\nu$.
Since $\bP_\star$ is a probability measure on the compact metric space $\Omega_\star$, it is outer regular, see \cite[Theorem 2.17]{rudin1987real}: for any Borel set $B \subset \Omega_\star$ and any $\epsilon > 0$, there exists an open set $O_\epsilon \subset \Omega_\star$ such that $B \subset O_\epsilon$ and $\bP_\star(O_\epsilon \setminus B) \le \epsilon$. By the dominated convergence theorem
\[
\lim_{k \to \infty} \int_{\Omega_\star} \left| \varphi_{k, O_\epsilon} - \mathds{1}_{O_\epsilon} \right| \, d \bP_\star = 0.
\]
We choose $k_\epsilon \ge 0$ such that $\int_{\Omega_\star} \left| \varphi_{k, O_\epsilon} - \mathds{1}_{O_\epsilon} \right| \, d \bP_\star \le \epsilon$ and set $\varphi_\epsilon = \varphi_{k_\epsilon, O_\epsilon}$. By the above arguments, we have $\varphi_\epsilon \in \cH_\nu$, $\| \varphi_\epsilon \|_\infty \le 1$, and 
\[
\| \varphi_\epsilon - \mathds{1}_B \|_{L^1(\bP_\star)} \le \| \varphi_\epsilon - \mathds{1}_{O_\epsilon} \|_{L^1(\bP_\star)} + \bP_\star(O_\epsilon \setminus B) \le 2 \epsilon,
\]
which proves the convergence in $L^1(\bP_\star)$. 
Next, assume $\vf\in L^\infty$. Without loss of generality we can assume $\vf\geq 0$ and $\|\vf\|_\infty=2$. Let $B_1=\{\xi\in\Omega_\star\;:\; \vf(\xi)\geq 1\}$ and
\[
B_k=\left\{\xi\in\Omega_\star\;:\; \vf(\xi)\geq 2^{-k}+\sum_{j=0}^{k-1}\Id_{B_j}2^{-j}\right\}.
\]
By construction $\|\vf-\sum_{j=0}^{k-1}\Id_{B_j}2^{-j}\|_\infty\leq 2^{-k}$, hence we can use the above approximations of the characteristic functions to approximate $\vf$ in $L^1$ with a sequence with norm bounded by $2$. The case $\vf\in L^1$ can be obtained by approximation by bounded functions.
\end{proof}
Next, we state a useful technical lemma.
\begin{lem}\label{eq:nonunif-abs}
Under assumptions  {\bf (Pos)}, {\bf (Exp)} and  {\bf (Ell)}  there exists $\gamma_*\in L^\infty(\Omega,\bP)$, $\gamma_\star>0$ such that, for all $n>n_\star$ and $\bw\in\cW^\bN$, we have
\[
\gamma_\star(\omega)\leq \frac{p(\bomega, n, \bw_0 \ldots \bw_{n-1})}{p(\tau_{\bw_0} \bomega, n-1, \bw_1 \ldots \bw_{n-1})} \leq \gamma_\star(\omega)^{-1}
\]
\end{lem}
\begin{proof}
 For all $n > n_\star$, we have
\[
\begin{aligned}
&\frac{p(\bomega, n, \bw_0 \ldots \bw_{n-1})}{p(\tau_{\bw_0} \bomega, n-1, \bw_1 \ldots \bw_{n-1})}  = \frac{p(\bomega, n, \bw_0 \ldots \bw_{n-1})}{p(\bomega, n-1, \bw_0 \ldots \bw_{n-2})} \frac{p(\bomega, n-1, \bw_0 \ldots \bw_{n-2})}{p(\tau_{\bw_0} \bomega, n-1, \bw_1 \ldots \bw_{n-1})} \\
& = \bP_{\bomega}( \bw_{n-1} \; | \; \bw_{0} \ldots \bw_{n-2}) \frac{p(\bomega, n-1, \bw_0 \ldots \bw_{n-2})}{p(\tau_{\bw_0} \bomega, n-1, \bw_1 \ldots \bw_{n-1})} \\
& \ge \left[ \bP_{\tau_{\bw_0} \bomega}( \bw_{n-1} \; | \; \bw_1 \ldots \bw_{n-2}) - \Const \nu^n \right] \frac{p(\bomega, n-1, \bw_0 \ldots \bw_{n-2})}{p(\tau_{\bw_0} \bomega, n-1, \bw_1 \ldots \bw_{n-1})} \\
& = \left[\frac{p(\tau_{\bw_0} \bomega, n-1, \bw_1 \ldots \bw_{n-1})}{p(\tau_{\bw_0} \bomega, n-2, \bw_1 \ldots \bw_{n-2})} - \Const \nu^n \right] \frac{p(\bomega, n-1, \bw_0 \ldots \bw_{n-2})}{p(\tau_{\bw_0} \bomega, n-1, \bw_1 \ldots \bw_{n-1})} \\
& = \left[ 1 - \Const \nu^n \frac{p(\tau_{\bw_0} \bomega, n-2, \bw_1 \ldots \bw_{n-2})}{p(\tau_{\bw_0} \bomega, n-1, \bw_1 \ldots \bw_{n-1})} \right] \frac{p(\bomega, n-1, \bw_0 \ldots \bw_{n-2})}{p(\tau_{\bw_0} \bomega, n-2, \bw_1 \ldots \bw_{n-2})} \\
& \ge (1 - \Const \gamma_0^{-1 } \nu^n ) \frac{p(\bomega, n-1, \bw_0 \ldots \bw_{n-2})}{p(\tau_{\bw_0} \bomega, n-2, \bw_1 \ldots \bw_{n-2})}\\
&\geq \prod_{j=n_\star +1}^n(1 - \Const \gamma_0^{-1 } \nu^j ) \frac{p(\bomega, n_\star+1, \bw_0 \ldots \bw_{n_\star})}{p(\tau_{\bw_0} \bomega, n_\star, \bw_1 \ldots \bw_{n_\star})},
\end{aligned}
\]
where we have used {\bf (Exp)} at the third line and {\bf (Ell)} at the last line. The lower bound follows then by {\bf (Pos)}, provided that $\Const \gamma_0^{-1 } \nu^{n_\star}<1$, which we can always ensure by eventually redefining $n_\star$.
The upper bound can be established similarly.\footnote{ Note however that {\bf (Ell)} is not needed to prove the upper bound.}
\end{proof}
\begin{rem} Note that a slight strengthening of {\bf (Pos)} would imply that $\gamma_\star$ can be chosen to be constant.\footnote{ That is, one could ask, for all $n>n_\star$, $\inf_{\omega\in\Omega}p(\bomega, n, \bw_0 \ldots \bw_{n-1})>0$, which holds in all the example we have in mind.} Then Lemma \ref{eq:nonunif-abs} would imply:
for each $k\in\bN$, there exists $C_k>0$ such that for $\bP$-a.e. $\bomega \in \Omega$, all $n \ge 0$ and all $\bw \in \cW^\bN$,
\[
C_k^{-1} \le \sum_{(w_1, \ldots, w_k) \in \cW^k} \frac{p(\tau_{-(w_1 + \ldots + w_k)} \bomega, n+k, w_1 \ldots w_k \bw_0 \ldots \bw_{n-1})}{p(\bomega, n, \bw_0 \ldots \bw_{n-1})} \le C_k.
\]
Hence the all point of {\bf (Abs)} rests in the uniformity with respect to $k$.
\end{rem}

We now define what will turn out to be the potential associated to $\cL_\star$:

\begin{lem} \label{lem:potential}
There exists a measurable function $J : \Omega \times \cW^\bN \to \bR^+$ such that for all $n \ge 0$
\begin{equation} \label{eqn:potential}
\underset{\bomega \in \Omega}\essup \sup_{\bw \in \cW^\bN} \left| \frac{p(\bomega, n, \bw_0 \ldots \bw_{n-1})}{p(\tau_{\bw_0} \bomega, n-1, \bw_1 \ldots \bw_{n-1})} - J(\bomega, \bw) \right| \le \Const \nu^n.
\end{equation}
Moreover, $J$ belongs to $\cH_\nu$  and $J(\bomega, \bw) > 0$ for $\bP$-a.e $\bomega \in \Omega$ and all $\bw \in \cW^\bN$.
\end{lem}

\begin{proof} Let $p_n(\bomega,  \bw) = \frac{p(\bomega, n,  \bw_0 \ldots \bw_{n-1})}{p(\tau_{ \bw_0} \bomega, n-1, \bw_1 \ldots \bw_{n-1})}$. Using {\bf (Exp)}, we have
\[
\begin{aligned}
&\left| p_{n}(\bomega, \bw) - p_{n+1}(\bomega, \bw) \right| = \frac{p(\bomega, n, \bw_0 \ldots \bw_{n-1})}{p(\tau_{\bw_0} \bomega, n, \bw_1 \ldots \bw_n)}\\
&\phantom{\left| p_{n+1}(\bomega, \bw) - p_n(\bomega, \bw) \right| =}
\times \left| \frac{p(\tau_{\bw_0} \bomega, n, \bw_1 \ldots \bw_n)}{p(\tau_{\bw_0} \bomega, n-1, \bw_1 \ldots \bw_{n-1})} - \frac{p(\bomega, n+1, \bw_0 \ldots \bw_n)}{p(\bomega, n, \bw_0 \ldots \bw_{n-1})}\right| \\
& = \frac{p(\bomega, n, \bw_0 \ldots \bw_{n-1})}{p(\tau_{\bw_0} \bomega, n, \bw_1 \ldots \bw_n)} \left| \bP_{\tau_{\bw_0} \bomega}(\bw_n \; | \; \bw_1 \ldots \bw_{n-1}) - \bP_{\bomega}( \bw_n \; | \; \bw_0 \ldots \bw_{n-1}) \right| \\
& \le \Const \nu^n \frac{p(\bomega, n, \bw_0 \ldots \bw_{n-1})}{p(\tau_{\bw_0} \bomega, n, \bw_1 \ldots \bw_n)}.
\end{aligned}
\]
From {\bf (Abs)}, substituting $\tau_{\bw_0}\bomega$ to $\bomega$, we have
\begin{equation} \label{eqn:bounded_J}
p(\bomega, n, \bw_0 \ldots \bw_{n-1}) \le \Const p(\tau_{\bw_0} \bomega, n-1, \bw_1 \ldots \bw_{n-1}),
\end{equation}
from which it follows, using {\bf (Ell)}, for all $n \ge n_\star$,
\[
\frac{p(\bomega, n, \bw_0 \ldots \bw_{n-1})}{p(\tau_{\bw_0} \bomega, n, \bw_1 \ldots \bw_n)} \le \Const \frac{p(\tau_{\bw_0} \bomega, n-1, \bw_1 \ldots \bw_{n-1})}{p(\tau_{\bw_0} \bomega, n, \bw_1 \ldots \bw_n)} \le \Const \gamma_0^{-1}.
\]
We thus get for all $n \ge n_\star$
\[
\left| p_{n+1}(\bomega, \bw) - p_n(\bomega, \bw) \right| \le \Const \nu^n,
\]
and so for any $m\ge 0$, 
\begin{equation} \label{eqn:cauchy_prob}
\left| p_{n+m}(\bomega, \bw) - p_n(\bomega, \bw) \right| \le \Const \sum_{k=0}^{m-1} \nu^{n+k} \le \Const \nu^n.
\end{equation}
It follows that $(p_n(\bomega, \bw))_n$ is a Cauchy sequence for $\bP$-a.e. $\bomega \in \Omega$ and all $\bw \in \cW^\bN$, and has thus a limit $J(\bomega, \bw)$. Taking the limit $m \to \infty$ in \eqref{eqn:cauchy_prob}, we obtain \eqref{eqn:potential} for all $n \ge n_\star$. From \eqref{eqn:bounded_J}, it follows that $\|J\|_\infty < \infty$, which also allows to deduce \eqref{eqn:potential} for all $n \ge 0$. The fact that $|J|_\nu <\infty$  is a direct consequence of \eqref{eqn:potential}.

The positivity of $J$ follows then by Lemma \ref{eq:nonunif-abs}.
\end{proof}

Accordingly, $\log J$ is H\"{o}lder with respect to the usual metric on the shift. Hence it can be seen as a potential of a Gibbs measure. Of course, such a Gibbs measure is random, depending on $\bomega$, and non translation invariant, but it is a natural generalisation of the usual random walk in random environment situation in which one has a random Markov chain on $\cW^\bN$.

The transfer operator $\cL_\star$ has the following expression:

\begin{lem} \label{lem:transfer_fstar} For any $f \in L^1(\bP_\star)$, we have
\begin{equation} \label{eqn:transfer_fstar}
\cL_\star f (\bomega, \bw) = \sum_{w \in \cW} J( \tau_{- w} \bomega, w \bw) f(\tau_{- w} \bomega, w \bw).
\end{equation}
\end{lem}

\begin{proof} We have to prove that, for all $f \in L^1(\bP_\star)$ and $g \in L^\infty(\bP_\star)$,  
\begin{equation} \label{eqn:lstar_dual}
\int_{\Omega_\star} f   \, g \circ \cF_\star \, d \bP_\star = \int_{\Omega_\star} \cL_\star f \, g \, d \bP_\star,
\end{equation}
where $\cL_\star f$ is given by \eqref{eqn:transfer_fstar}. We first assume that both $f$ and $g$ are bounded, and depend only on $(\bomega, \bw_0, \ldots, \bw_{k-1})$ for some $k \ge 1$. For any $n \ge k$, we have
\[
\begin{split}
& \int_{\Omega_\star}  f \, g \circ \cF_\star \, d \bP_\star =  \int_{\Omega} \int_{\cW^\bN} f(\bomega, \bw_0, \ldots, \bw_{k-1}) \, g(\tau_{\bw_0} \bomega, \bw_1, \ldots, \bw_k) \bP_{\bomega}(d \bw) \bP(d \bomega) \\
 & = \int_{\Omega} \sum_{\bw_0, \ldots, \bw_n \in \cW} p(\bomega, n+1, \bw_0 \ldots \bw_n) f(\bomega, \bw_0, \ldots, \bw_{k-1}) g(\tau_{\bw_0} \bomega, \bw_1, \ldots, \bw_k) \bP(d \bomega) \\
& = \hskip-12 pt \sum_{\bw_0, \ldots, \bw_n \in \cW} \int_{\Omega} p(\tau_{- \bw_0} \bomega, n+1, \bw_0 \ldots \bw_n) f(\tau_{- \bw_0} \bomega, \bw_0, \ldots, \bw_{k-1}) g(\bomega, \bw_1, \ldots, \bw_k) \bP(d \bomega) \\
& = \int_\Omega \left( \sum_{w \in \cW} \frac{p(\tau_{- w} \bomega, n+1, w \bw_0 \ldots \bw_{n-1})}{p(\bomega, n, \bw_0 \ldots \bw_{n-1})} f(\tau_{- w} \bomega, w, \bw_0, \ldots) \right) g(\bomega, \bw) \bP_\star(d \bomega, d\bw), \\
\end{split}
\]
where we have used the translation invariance of $\bP$ at the third line. Taking the limit as $n \to \infty$ and using Lemma \ref{lem:potential}, we obtain \eqref{eqn:lstar_dual}. The result for general $f$ and $g$ is obtained by approximation.
\end{proof}

Define for each $k\ge 1$, 
\[
J_k(\bomega, \bw) = \prod_{i=0}^{k-1} J( \cF_\star^i (\bomega, \bw)) = \lim_{n \to \infty} \frac{p(\bomega, n, \bw_0 \ldots \bw_{n-1})}{p(\tau_{\bw_0 + \ldots \bw_{k-1}}\bomega, n-k, \bw_k \ldots \bw_{n-1})}.
\]

It is immediate to verify that, for any $f \in L^1(\bP_\star)$,
\[
\cL_\star^k f(\bomega, \bw) = \sum_{w^k \in \cW^k} J_k(\tau_{- w^k} \bomega, w^k \bw) f(\tau_{-w^k} \bomega, w^k\bw)
\]
where, for $w^k=(w_0,\dots, w_{k-1})$, $\tau_{- w^k}=\tau_{-(w_0+\cdots+w_{k-1})}$.

We introduce, for $w^k = (w_0, \ldots, w_{k-1}) \in \cW^k$, the map \[
\psi_{w^k}(\bomega, \bw) = (\tau_{- w^k} \bomega, w^k \bw),
\]
 so that 
\[
\cL_\star^k f = \sum_{w^k \in \cW^k} J_k \circ \psi_{w^k} f \circ \psi_{w^k}.
\]

\begin{lem} \label{lem:potential_iterated}
$J_k$ belong to $\cH_\nu$ for all $k \ge 1$.
\end{lem}

\begin{proof} Recall the notation $p_n(\bomega, \bw) = \frac{p(\bomega, n, \bw_0 \ldots \bw_{n-1})}{p(\tau_{\bw_0} \bomega, n-1, \bw_1 \ldots \bw_{n-1})}$, and set 
\[
p_{n,k}(\bomega, \bw) := \frac{p(\bomega, n, \bw_0 \ldots \bw_{n-1})}{p(\tau_{\bw_0 + \ldots \bw_{k-1}}\bomega, n-k, \bw_k \ldots \bw_{n-1})} = \prod_{i=0}^{k-1} p_{n-i}(\cF_\star^i(\bomega, \bw)).
\]
By Lemma \ref{lem:potential}, we have
\[
\left| J(\cF_\star^i(\bomega, \bw)) - p_{n-i}(\cF_\star^i(\bomega, \bw)) \right| \le \Const \nu^{n-i},
\]
and, consequently, using the inequality
\[
\left| \prod_{i=0}^{k-1} a_i - \prod_{i=0}^{k-1} b_i \right| \le \sum_{i=0}^{k-1} |a_i - b_i| \prod_{j \neq i} \max\{a_j, b_j\},
\]
valid for all non-negative sequences $(a_i), (b_i)$, we obtain
\[
\begin{aligned}
\left| J_k(\bomega, \bw) - p_{n,k}(\bomega, \bw) \right| & \le \Const \nu^n \sum_{i=0}^{k-1} \nu^{-i} \prod_{j \neq i} \max\{ \|J\|_\infty, \| p_{n-i} \|_\infty \} \\
& = C_k \nu^n.
\end{aligned}
\]
where $C_k$ depends only on $k$, since $\sup_n \|p_n\|_\infty < \infty$ by {\bf (Abs)} and $\| J \|_\infty < \infty$ by Lemma \ref{lem:potential}. The lemma follows immediately.
\end{proof}

\begin{lem} \label{lem:operator_bounded}
There exists $\Const >0$ such that $\Const^{-1} \le \cL_\star^k \mathds{1}(\bomega, \bw) \le \Const$ for all $k\ge 0$, $\bP$-a.e. $\bomega \in \Omega$ and all $\bw \in \cW^\bN$.
\end{lem}

\begin{proof} This is a simple reformulation of {\bf (Abs)}, as
\[
\begin{aligned}
\cL_\star^k \mathds{1}(\bomega, \bw) & = \sum_{(w_0, \ldots, w_{k-1}) \in \cW^k} J_k(\tau_{- (w_0 + \ldots + w_{k-1})} \bomega, w_0 \ldots w_{k-1} \bw) \\
& = \lim_{n \to \infty} \sum_{(w_0, \ldots, w_{k-1}) \in \cW^k} \frac{p(\bomega, n, \bw_0 \ldots \bw_{n-1})}{p(\tau_{\bw_0 + \ldots \bw_{k-1}}\bomega, n-k, \bw_k \ldots \bw_{n-1})}.
\end{aligned}
\]
\end{proof}

\begin{lem} \label{lem:ly_lstar} There exist $\Const > 0$  and $\xi \in (0,1)$ such that for all $n \ge 0$ and all $f \in \cH_\nu$,
\[
\begin{aligned}
&\| \cL_\star^n f\|_\infty \le \Const \| f \|_\infty, \\
&\| \cL_\star^n f \|_\nu \le \Const \xi^n \| f \|_\nu + \Const \|f \|_\infty.
\end{aligned}
\]
\end{lem}

\begin{proof} For $f \in \cH_\nu$, we have
\[
\begin{aligned}
| \cL_\star^n f | & \le \sum_{w^n \in \cW^n} J_n \circ \psi_{w^n}| f |\circ \psi_{w^n}  \le \|f \|_\infty \cL_\star^n \mathds{1}  \le \Const \|f\|_\infty,
\end{aligned}
\]
by Lemma \ref{lem:operator_bounded}. This proves that $\| \cL_\star^n f \|_\infty \le C \|f \|_\infty$.
We also have, setting $\eta = \psi_{w^n}(\bomega, \bw)$ and $\eta' = \psi_{w^n}(\bomega, \bw')$
\[
\begin{split}
&| \cL_\star^n f(\bomega, \bw) - \cL_\star^n f(\bomega, \bw') | \le \sum_{w^n \in \cW^n} J_n(\eta) | f(\eta) - f(\eta')| +  | J_n(\eta) - J_n(\eta')| |f(\eta')| \\
&\phantom{| \cL_\star^n f} \le \left( \sum_{w^n \in \cW^n} J_n(\eta)  |f |_\nu + \sum_{w^n \in \cW^n} |J_n|_\nu  \|f\|_\infty \right) d_\nu(w^n\bw, w^n \bw') \\
&\phantom{| \cL_\star^n f} \le \left( \cL_\star^n \mathds{1}(\bomega, \bw)  |f|_\nu +  ( \sharp \cW)^n |J_n|_\nu \|f \|_\infty \right)  \nu^n d_\nu(\bw, \bw').
\end{split}
\]
By Lemma \ref{lem:operator_bounded}, this shows that, for all $n \ge 0$ and $f \in \cH_\nu$,
\[
\begin{aligned}
\| \cL_\star^n f \|_\nu & \le \Const \nu^n |f|_\nu + (1 + ( \sharp \cW)^n \nu^n |J_n|_\nu) \|f \|_\infty \\
& \le \Const \nu^n \|f\|_\nu + C_n \|f\|_\infty.
\end{aligned}
\]
In particular, $\cL_\star : \cH_\nu \to \cH_\nu$ is a continuous operator. 

Take $k \ge 0$ such that the term $\widetilde{\nu} :=\Const \nu^k$ in front of $\| f\|_\nu$ is strictly less than 1 and set $\xi=\widetilde\nu^{\frac 1k}$. Writing $n = q k + r$, with $0 \le r <k$, we have, by iterating the previous inequality,
\[
\begin{aligned}
\| \cL_\star^n f \|_\nu & \le \widetilde{\nu}^q \| \cL_\star^r f \|_\nu + \Const C_k (1-\widetilde{\nu}^{-1}) \| f\|_\infty \\
& \le \widetilde{\nu}^q \sup_{r < k} \| \cL_\star^r \|_{\cH_\nu \to \cH_\nu} \|f \|_\nu + \Const \|f \|_\infty \\
& \le \Const \xi^n \|f\|_\nu + \Const \|f\|_\infty.
\end{aligned}
\]
\end{proof}

\begin{lem} \label{lem:ergodic_average} There exists a continuous projection $\Pi : L^1(\bP_\star) \to L^1(\bP_\star)$ with $ \Pi( L^1(\bP_\star) ) = \ker ({\rm id} - \cL_\star)$ such that
\[
\frac{1}{n} \sum_{k=0}^{n-1} \cL_\star^k \to \Pi
\]
in the strong operator topology.
\end{lem}

\begin{proof}
For $h \in \cH_\nu$, Lemma \ref{lem:ly_lstar} implies that $\left\{ \frac{1}{n} \sum_{k=0}^{n-1} \cL_\star^k h \right\}_{n \ge 1}$ is bounded in $L^{\infty}(\bP_\star)$. By the Banach-Alaoglu theorem, since $L^\infty(\bP_\star)$ is the dual of $L^1(\bP_\star)$, the set $\left\{ \frac{1}{n} \sum_{k=0}^{n-1} \cL_\star^k h \right\}_{n \ge 1}$ is weakly relatively compact in $L^1(\bP_\star)$. This holds for all $h \in \cH_\nu$, which is dense in $L^1(\bP_\star)$ by Lemma \ref{lem:htheta_dense}, and so by the Kakutani-Yosida theorem \cite[VIII.5.2, 5.3]{dunford1988linear}, the operators $\frac{1}{n} \sum_{k=0}^{n-1} \cL_\star^k$ converge in the strong operator topology to the projection $\Pi$ with range the set of fixed points of $\cL_\star$ in $L^1(\bP_\star)$ and kernel the closure of $({\rm id} - \cL_\star)(L^1(\bP_\star))$.
\end{proof}

Define
\[
h_\star = \Pi \mathds{1} = \lim_{n\to\infty} \frac{1}{n} \sum_{k=0}^{n-1} \cL_\star^k \mathds{1},
\]
in $L^1(\bP_\star)$. By Lemma \ref{lem:ergodic_average},we have $\cL_\star h_\star = h_\star$. We clearly have $\int_{\Omega_\star} h_\star d \bP_\star = 1$, and the fact that $\Const^{-1} \le h_\star \le \Const$, $\bP_\star$-a.e., is an immediate consequence of Lemma \ref{lem:operator_bounded}. Consequently, the probability measure $\bQ_\star$ defined by 
\begin{equation} \label{eqn:inv_proba}
d \bQ_\star = h_\star d \bP_\star
\end{equation}
is $\cF_\star$-invariant and equivalent to $\bP_\star$.

Next, we show that $\Pi (L^1(\bP_\star))$ is the one-dimensional subspace generated by $h_\star$. Firstly, we prove a useful inclusion in $\cH_\nu$.
\begin{lem} \label{lem:invariant_set_finite}
If $f \in L^\infty(\bP_\star)$ and $\cL_\star f = f$, then $f\in\cH_\nu$.\footnote{ That is, there exists an element in the equivalence class of $f$ that belongs to $\cH_\nu$.} 
\end{lem}
\begin{proof}
Let $(\varphi_\epsilon)_\epsilon \subset \cH_\nu$ be such that $\| f - \varphi_\epsilon \|_{L^1(\bP_\star)} = \cO (\epsilon)$ and $\| \varphi_\epsilon \|_\infty = \cO (1)$, such a sequence exists by Lemma \ref{lem:htheta_dense}. We have
\[
\begin{aligned}
f =  \cL_\star^{n} f & = \cL_\star^{n} \varphi_\epsilon + \cL_\star^{n}(f - \varphi_\epsilon) \\
& =: \widehat{\varphi}_\epsilon^{(n)} + \gamma_\epsilon^{(n)}.
\end{aligned}
\]
This decomposition satisfies
\[
\| \gamma_\epsilon^{(n)} \|_{L^1(\bP_\star)} = \| \cL_\star^{n}(f - \varphi_\epsilon) \|_{L^1(\bP_\star)} \le \|f - \varphi_\epsilon \|_{L^1(\bP_\star)} = \cO (\epsilon),
\]
and
\[
\| \widehat{\varphi}_\epsilon^{(n)} \|_\nu = \| \cL_\star^{n} \varphi_\epsilon \|_\nu \le \Const \xi^{n} \| \varphi_\epsilon \|_\nu + \Const \| \varphi_\epsilon \|_\infty = \cO ( \xi^{n} \| \varphi_\epsilon \|_\nu + 1),
\]
using Lemma \ref{lem:ly_lstar}. If we choose $n_\epsilon$ such that $\xi^{n_\epsilon } \| \varphi_\epsilon\|_\nu = \cO (1)$ and set $\widehat{\varphi}_\epsilon = \widehat{\varphi}_\epsilon^{(n_\epsilon)}$ and $\gamma_\epsilon = \gamma_\epsilon^{(n_\epsilon)}$, we then have $f = \widehat{\varphi}_\epsilon + \gamma_\epsilon$ with $\| \gamma_\epsilon \|_{L^1(\bP_\star)} = \cO (\epsilon)$ and $\| \widehat{\varphi}_\epsilon \|_\nu = \cO (1)$.

For $\delta > 0$, we define 
\[
B_{\epsilon, \delta} = \{ | f - \widehat{\varphi}_\epsilon | > \delta \} = \{ | \gamma_\epsilon | > \delta \},
\]
which satisfies $\bP_\star(B_{\epsilon, \delta}) \le \delta^{-1} \| \gamma_\epsilon \|_{L^1(\bP_\star)}$ by Markov's inequality.

For $\bP$-a.e. $\bomega \in \Omega$ and all $\bw, \bw' \in \cW^\bN$ such that both $(\bomega, \bw)$ and $(\bomega, \bw')$ do not belong to $B_{\epsilon, \delta}$, we have
\[
\begin{aligned}
| f( \bomega, \bw) - f (\bomega, \bw') | & \le | \widehat{\varphi}_\epsilon (\bomega, \bw) - \widehat{\varphi}_\epsilon(\bomega, \bw')| + | \gamma_\epsilon(\bomega, \bw) - \gamma_\epsilon (\bomega, \bw') | \\
& \le | \widehat{\varphi}_\epsilon |_\nu d_\nu( \bw, \bw') + | \gamma_\epsilon(\bomega, \bw) | + | \gamma_\epsilon(\bomega, \bw')| \\
& \le C d_\nu(\bw, \bw') + 2 \delta.
\end{aligned}
\]
We set $B_\delta = \bigcap_{k \ge 0} \bigcup_{j \ge k} B_{2^{-j}, \delta}$, which satisfies $\bP_\star (B_\delta ) = 0$, since 
\[
\bP_\star \left( \bigcup_{j \ge k} B_{2^{-j}, \delta}\right) = \cO \left( \sum_{j \ge k} \| \gamma_{2^{-j} }\|_{L^1(\bP_\star)} \right) = \cO \left(\sum_{j \ge k} 2^{-j} \right) = o (1).
\]
Thus, $B=\cup_{n\in\bN}B_{1/n}$ is also of zero measure and, eventually changing $f$ on the zero measure set $B$, we have $f\in\cH_\nu$.
\end{proof}

We can now prove the main theorem:

\begin{proof}[{\bf\emph{Proof of Theorem \ref{thm:gibbs_walk_erg}}}]
The probability measure $\bQ_\star$ defined by \eqref{eqn:inv_proba} is $\cF_\star$-invariant and equivalent to $\bP_\star$. If $B \subset \Omega_\star$ is a $\cF_\star$-invariant set, we have
\[
\cL_\star(\mathds{1}_B h_ \star) = \cL_\star((\mathds{1}_B \circ \cF_\star ) h_\star) = \mathds{1}_B \cL_\star(h_\star) = \mathds{1}_B h_\star,
\]
and so $\mathds{1}_B h_\star$ is a fixed point of $\cL_\star$ in $L^\infty(\bP_\star)$. 
By Lemma \ref{lem:invariant_set_finite}, we have $\mathds{1}_B h_\star \in \cH_\nu$, and so $\mathds{1}_B=h_\star^{-1}(h_\star \mathds{1}_B)\in\cH_\nu$.\footnote{ Since $h_\star$ belongs to $\cH_\nu$ by Lemma \ref{lem:invariant_set_finite} and so does $h_\star^{-1}$ since $\inf h_\star > 0$.} This implies that there exists $N_B>0$ such that $\mathds{1}_B(\bomega,\bw)=\mathds{1}_B(\bomega, w_0,\dots, w_{N_B-1})$. 

By the invariance of $B$ it follows, for each $m\geq N_B$,
\begin{equation}\label{eq:invariance}
\begin{split}
\Id_B(\bomega, w_0,\dots, w_{N_B-1})&=\Id_B\circ \cF^{m}_\star(\bomega, w_0,\dots, w_{N_B-1})\\
&=\Id_B(\tau_{w_0+\cdots+ w_{m-1}}\bomega, w_{m},\dots, w_{N_B+m-1}).
\end{split}
\end{equation}
By {\bf (Pro)} we can choose $m$ and $w_{N_B}, \dots, w_{m-1}$ such that $w_0+\cdots w_{m-1}=0$. It follows that $\mathds{1}_B (\bomega,\bw)=\mathds{1}_B(\bomega)$. Then \eqref{eq:invariance} implies $\tau_{w_0+ \cdots+ w_{m-1}}B\subset B$ for all $(w_0, \ldots, w_{m-1}) \in \cW^{m}$. Accordingly, $B$ is invariant for the group generated by $\cW$ and, by {\bf (Pro)} again, it is either of zero or full measure due to ergodicity of $\bP$, which concludes the proof.
\end{proof}

\begin{lem} \label{lem:one_dimensional}
For all $f \in L^1(\bP_\star)$, we have
\[
\Pi f = \left( \int_{\Omega_\star} f \, d \bP_\star \right) h_\star.
\]
\end{lem}
\begin{proof}
For each $\vf\in L^\infty$ we have
\[
\begin{split}
\int_{\Omega_\star} \vf\Pi fd\bP_*&=\lim_{n\to\infty}\frac 1n\sum_{k=0}^{n-1}\int_{\Omega_\star} \vf\cL_\star^kf d\bP_*=\int_{\Omega_\star} \lim_{n\to\infty}\frac 1n\sum_{k=0}^{n-1}\vf\circ\cF_\star^k\cdot fd\bP_\star\\
&=\left[\int_{\Omega_\star} \vf h_\star d\bP_\star\right]\left[\int_{\Omega_\star}  fd\bP_\star\right]
\end{split}
\]
where, in the second equality, we have used Lebesgue dominated convergence Theorem and, in the second line, we have used the Birkhoff theorem and the ergodicity of $\bQ_\star$ (and hence of $\bP_\star$) established in Theorem \ref{thm:gibbs_walk_erg}.
\end{proof}

%%%%%%%%%%%%%%%%%%%%%%%%%%%%%%%%%%%%%%%%%%%%%%%%%%%%%
\subsection{Application to deterministic walks in random environment} \label{sec:app_dwre}

Deterministic walks in random environment, as presented in Section \ref{sec:walk-det}, naturally define random processes as described in the previous subsections. Indeed, if $\cA = \{(f_\alpha, \cM, \cG_\alpha) \}_{\alpha \in A}$ is a deterministic walk in random environment, where all maps $f_\alpha$ are non-singular with respect to some reference measure $m$ on $\cM$, and the initial condition is given by an absolutely continuous probability measure $d \mu = h_0 dm$, then the probabilities $p(\bomega, n, w_0 \ldots w_{n-1})$ are given by 
\begin{equation} \label{eqn:transition_proba}
p(\bomega, n, w_0 \ldots w_{n-1}) = \int_{\cM} \cL_{\bomega, z_{n-1}, w_{n-1}} \ldots \cL_{\bomega, z_0, w_0} h_0 dm,
\end{equation}
as we have seen in Section \ref{sec:walk-det}. Recall that $w_n = e(\bomega_{z_n}, x_n)$, where $(x_n, z_n) = \bF_{\bomega}^n(x_0, z_0)$.

\begin{rem}\label{rem:environ2}
Note that we have a priori defined two different notions of environment as seen from the particle, in subsections \ref{sec:environ} and \ref{sec:environ2}, but the map $\Phi : \Omega \times \cM \to \Omega \times \cW^{\bN}$ defined by $\Phi(\bomega, x) = (\bomega, \bw)$ with $\bw = (w_n)_n$, is a semi-conjugacy between $(\Omega \times \cM, \cF)$ and $(\Omega \times \cW^{\bN}, \cF_\star)$, and if the maps $f_\alpha$ are expansive, it is invertible a.e.
\end{rem}

If we are able to check the assumptions {\bf (Pos), (Exp), (Abs), (Ell)} and {\bf (Pro)}, then Theorem \ref{thm:gibbs_walk_erg} applies, and we deduce the existence of a deterministic drift $V$.

An particular situation, which we have already encountered in Lemma \ref{lem:povop}, occurs when all maps $f_\alpha$ preserve the same invariant measure $d \lambda = h_0 dm$, and the set $\cG_\alpha$ is deterministic, i.e. $G_{\alpha, w } = G_w$ does not depend on $\alpha \in A$. In this case, the measure $\bP_\star$ on $\Omega \times \cW^{\bN}$ is invariant under $\cF_\star$, since it is the push-forward of $\bP_0 = \bP \times \lambda$, which is $\cF$-invariant and the condition {\bf (Abs)} is automatically satisfied by Lemma \ref{lem:abs_cond}.

\begin{rem}
Lemma \ref{lem:zero-velocity} and Remark \ref{rem:zero-velocity} agree, since
\[
\begin{aligned}
\sum_{w \in \cW} w \int_\Omega p(\bomega, 1, w) \bP(d \bomega) & = \sum_{w \in \cW} w \int_\Omega \int_{\cM} \cL_{\bomega, 0, w} h_0 \, dm \, \bP( d \bomega) \\ & = \sum_{w \in \cW} w \int_\Omega \int_{\cM} \cL_{f_{\bomega_w}} \mathds{1}_{G_w} h_0 \, dm \, \bP( d \bomega) \\ &= \sum_{w \in \cW} w \int_{G_w} h_0 dm
\end{aligned}
\]
and 
\[
\begin{aligned}
\bE_0(e \circ \pi) & = \int_\Omega \int_{\cM} e(\pi(\bomega, x)) h_0(x) m(dx) \bP(d \bomega) \\ & = \int_\Omega \int_{\cM} e(\bomega_0, x) h_0(x) m(dx) \bP( d \bomega) \\ &= \sum_{w \in \cW} w \int_{G_w} h_0 dm.
\end{aligned}
\]
\end{rem}

%%%%%%%%%%%%%%%%%%%%%%%%%%%%%%%%%%%%%%%%%%%%%%%%%
\section{Examples}\label{sec:sup-simp}
In this section we discuss several concrete examples of the previous abstract models. 
\subsection{Lorentz gas as a deterministic walk in random environment}
It is easy to verify that the Random Lorentz gas presented in Section \ref{sec:lorentz} is a special example of a deterministic walk in random environment. However the Random Lorentz gas has several special features:
\begin{enumerate}
\item the maps $f_\alpha$ have all the same invariant measure $\lambda$, the Lebesgue measure;
\item the set $\cG_\alpha$ is non random, i.e. it does not depend from $\alpha$;
\item the associated random process is reversible, in particular $\lambda(e)=0$.
\end{enumerate}
The first two facts are obvious. Let us discuss the third, which we believe to be a key property in this research program (see \cite{Toth1,Toth2,Toth3, Toth4} where, in special models, a similar property is instrumental in establishing the CLT).
The dynamics is reversible under the involution $(q,p)\to (q,-p)$, which, at the level of the Poincar\'e map, writing $x\in\cup_{w \in \cW} G_w$ as $x=(w, s,\theta)$,\footnote{ The curvilinear coordinate $s$ is chosen such that the segments $s \to (w, s)$ and $s \to (-w, s)$, $w \neq 0$ go in the same direction.} reads $i(w,s,\theta,z)=(-w, s, \theta,z+w)$.\footnote{ Recall that we are considering the case of deterministic gates, although the following consideration easily extend to the general case. In particular, $e$ is a function of $x$ only, and we have $e(w,s,\theta) = w$.} Also, let $\pi(w,s,\theta,z)=(w,s,\theta)$ and $i_1(w,s,\theta) = (-w, s, \theta)$ so that $i_1 \circ \pi = \pi \circ i$. Thus, choosing as initial measure the common invariant measure of the Poincar\'e maps $h_0$ (hence $h_0\circ i_1=h_0$), we have
\[
\begin{split}
\bP_\star(\{z(1),\dots, z(n)\}\;|\;\bomega)&=\int_\cB  \prod_{k=0}^{n-1} \Id_{G_{w(k)}}\circ\pi\circ\bF^k_\bomega(x,0) h_0(x)dx\\
&=\int_\cB  \prod_{k=0}^{n-1} \Id_{G_{w(k)}}\circ\pi\circ\bF^k_\bomega\circ i(i_1(x), e(x)) h_0(x)dx\\
&=\int_\cB  \prod_{k=0}^{n-1} \Id_{G_{-w(k)}}\circ\pi\circ\bF^{-k}_\bomega(x,-e(x)) h_0(x)dx
\end{split}
\]
\[
\phantom{\bP_\star(\{z(1),\dots, z(n)\}\;|\;\bomega)}=\int_\cB  \prod_{k=0}^{n-1} \Id_{G_{-w(k)}}(f_{\bomega_{z(k)}}^{-1}\circ\cdots\circ f_{\bomega_{z(1)}}^{-1}(x)) h_0(x)dx,
\]
where, in the third line, we have used the invariance of the measure with respect to $i_1$  and the relation $e(i_1(x)) = -e(x)$; while, in the last line, we have used the formula $\bF_{\bomega}(x,z)=(f_{\bomega_z}^{-1}(x),z-e(f_{\bomega_z}^{-1}(x)))$. Next, using the invariance of the measure with respect to the maps $f_{\bomega_z}$,
\[
\begin{split}
\bP_\star(\{z(1),\dots, z(n)\}\;|\;\bomega)&=\int_\cB  \prod_{k=0}^{n-1} \Id_{G_{-w(k)}}(f_{\bomega_{z(k+1)}}\circ\cdots\circ f_{\bomega_{z(n-1)}}(x)) h_0(x)dx.
\end{split}
\]
Then, setting $\tilde w(k)=-w(n-1-k)$, $\tilde z(k)=\sum_{j=0}^{k-1}\tilde w(j)$ and $\widetilde \omega=\tau_{z(n)}\bomega$,
\begin{equation}\label{eq:reversible}
\begin{split}
\hskip-.2cm\bP_\star(\{z(1),\dots, z(n)\}\;|\;\bomega)&=\int_\cB  \prod_{k=0}^{n-1} \Id_{G_{\tilde w(k)}}(f_{\widetilde\omega_{\tilde z(k)}}\circ\cdots\circ f_{\widetilde\omega_{\tilde z(1)}}(x)) h_0(x)dx\\
&=\bP_\star(\{\tilde z(n-1),\dots,\tilde z(0)\}\;|\;\tau_{z(n)}\bomega).
\end{split}
\end{equation}
Note that $\tilde z(k)=z(n-k)-z(n)$.
Which is the reversibility of the random process. In particular Lemma \ref{lem:rev} applies.

The Lorentz gas has a dynamics that is hard to study. To make further progresses let us consider simpler local dynamics.

%%%%%%%%%%%%%%
\subsection{Markovian models} \label{sec:mark-ex}
To try to get a better feeling for the difficulties involved in studying the above questions, let us try to invent a model stripped of all the technical difficulties present in the Lorentz gas dynamics. For simplicity let us discuss the case $d=1$, although similar considerations hold in any higher dimensional lattice. To simplify the dynamics $f$ in \eqref{eq:poinc-map} let us suppose that it is a map from $[0,1]$ to itself. Hence the map $\mathbb{F}_{\bomega}$ acts on $[0,1] \times \bZ$. Also, we assume that the environment is a random variable distributed according to a Bernoulli product measure over the space $\Omega= A^{\bZ} = \{-1,1\}^{\bZ}$.

\subsubsection*{\bfseries Example 1} The dynamics is defined by the map $f_\alpha(x)=4x\mod 1$ for $\alpha \in A$, with  $G_{-1, -1} = [0, 1/4]$, $G_{-1,+1} = [1/4,1]$ and $G_{+1, -1} = [0,3/4]$, $G_{+1,+1} = [3/4, 1]$.
\begin{rem} Here we are considering a more general situation than the one described for the Lorentz  gas insofar also the gates are random. This is indeed the general case also for the Lorentz gas. We considered the case of deterministic gates only to simplify the exposition.
\end{rem}
Also, we consider the initial distribution $h_0=1$.
Then an elementary computation shows that
\[
\bP_\star(z(n+1)-z(n)=\pm 1\;|\;\bomega, z(n),\dots , z(0)) = \left| G_{\bomega_{z(n)}, \pm 1} \right| =\frac 12\mp\frac{\bomega_{z(n)}}4.
\]
This is an example of Sinai's walk, hence we do not have the classical CLT.

\subsubsection*{\bfseries Example 2}
Assume that $G_{\alpha, -1}= G_{-1} = [0,1/2]$ and $G_{\alpha, +1}= G_{+1} = (1/2,1]$ for any $\alpha$ and the maps are defined by
\[
f_{-1}(x)=\begin{cases} 2x\quad&x\in [0,1/4]\\
4x\mod 1&x>1/4
\end{cases}
\]
\[
f_{+1}(x)=\begin{cases} 4x\mod 1\quad&x\in [0,3/4]\\
2x-1&x>3/4.
\end{cases}
\]
Again let us consider the initial distribution $h_0=1$. 
Denote by $\mathcal{L}_{\alpha, w}$ the operator $\mathcal{L}_{\alpha, w}(\phi) = \mathcal{L}_{f_\alpha}(\mathds{1}_{G_w} \phi)$. The two dimensional vector space  $\bV=\{a_{-1} \Id_{G_{-1}}+a_{+1}\Id_{G_{+1}}\;:\; a_{-1}, a_{+1}\in\bR\}$ is left invariant by the operators $\{\mathcal{L}_{\alpha, w}\}_{\alpha, w}$. Since $h_0 \in \bV$, this allows to compute the transition probabilities by using formula \eqref{eq:iterate_transfer}.

If $ \phi = a_{-1} \Id_{G_{-1}}+a_{+1}\Id_{G_{+1}}$, a direct computation shows that $\mathcal{L}_{\alpha, w}(\phi) = a_w \mathcal{L}_{\alpha, w}(\Id)$, and thus $\mathcal{L}_{\alpha', w'} \mathcal{L}_{\alpha, w} (\phi) = a_w \mathcal{L}_{\alpha', w'} \mathcal{L}_{\alpha, w} \Id$.
For any $\bomega$ and $z(1), \ldots,$ $ z(n), z(n+1)$, denote by $\alpha_k = \bomega_{z(k)+ w(k)}$ and $w_k = w(k) = z(k+1) - z(k)$. We have
\[
\mathbb{P}_\star(z(1), \ldots, z(n) \; | \; \bomega) = \int \mathcal{L}_{\alpha_{n-1}, w_{n-1}} \ldots \mathcal{L}_{\alpha_0, w_0} \Id.
\]
Set $\phi = \mathcal{L}_{\alpha_{n-2}, w_{n-2}} \ldots \mathcal{L}_{\alpha_0, w_0} \Id = a_{-1} \Id_{G_{-1}} + a_{+1} \Id_{G_{+1}} \in \bV$.
We have 
\[
\mathbb{P}_\star(z(1), \ldots, z(n) \; | \; \omega) = \int \mathcal{L}_{\alpha_{n-1}, w_{n-1}} \phi = a_{w_{n-1}} \int \mathcal{L}_{\alpha_{n-1}, w_{n-1}} \Id
\]
and
\[
\mathbb{P}_\star(z(1), \ldots, z(n), z(n+1) \; | \; \bomega) = a_{w_{n-1}} \int \mathcal{L}_{\alpha_n, w_n} \mathcal{L}_{\alpha_{n-1}, w_{n-1}} \Id.
\]
It follows that
\[
 \mathbb{P}_\star( z(n+1) \; | \; z(1), \ldots, z(n), \bomega) = \frac{\int \mathcal{L}_{\alpha_n, w_n} \mathcal{L}_{\alpha_{n-1}, w_{n-1}} \Id}{\int \mathcal{L}_{\alpha_{n-1}, w_{n-1}} \Id}
\]
which is a function of $z(n-1)$, $z(n)$ and $z(n+1)$ only. We have obtained a {\em persistent random walk}, that is a walk where the transition probability depends not only on the current position of the particle but also on its previous position.
\subsubsection*{\bfseries Initial conditions} A natural question that arises at this point is what happens if one starts by a different initial measure.
A moment thought shows that this is a non trivial issue. For instance, in the first example, there exists a Cantor set $C$ (of zero Lebesgue measure) that corresponds to the coordinates $x(n)$ never belonging to $(1/4, 3/4)$. For such points $x \in C$, the set $\{x(n) \in G_{\alpha, w}\}$ does not depend on $\alpha$, and so the process $(z(n))$ is completely unaffected by the environment. If we identify naturally the Cantor set C with $\{- 1, + 1 \}^{\bN}$ (in such a way that $x \in C$ is identified with the sequence $(i_n)$ such that $x(n) \in I_{i_n}$ for all $n \ge 0$, where $I_{-1} = [0, 1/4]$ and $I_{+1} = [3/4, 1]$), then the initial distribution of $x$ can be identified with a probability measure on $\{-1, +1\}^{\bN}$ and this measure will be the distribution law of the random process $(w(n))$. In particular, if we consider the Bernoulli measure with equal probabilities on such a Cantor set as the initial distribution of $x$, then we obtain a standard random walk which has a very different behaviour than the Sinai's walk.

Without going to such extremes, one can (perhaps more naturally) start from a measure absolutely continuous with respect to Lebesgue and wonder which kind of process this will yield. We do not discuss this issue at present because is it part of the more general discussion that we will start in the next section.
\begin{rem} 
The above examples (among other obvious limitations) are unreasonable in one key aspect: their Markov structure. It is inevitable to ask what happens when the Markov structure is absent (as for billiards). The next section is devoted to investigating such a situation.
\end{rem}
%%%%%%%%%%%%%%%%%%%%%%%%%%%%%%%%%%%%%%%%%%%%%%%%%%%%%%%%%%%%%
\section{Non-Markovian examples: general discussion}\label{sec:model}
We now consider a model of $d$-dimensional deterministic random walk in random environment $\cA=\{(f_\alpha, \cM, \cG_\alpha)\}_{\alpha\in A}$ for a finite set $A$, where $\cM = [0,1]$, all maps $f_\alpha :[0,1] \to [0,1 ]$ are piecewise $C^2$ and uniformly expanding (i.e. $|f_\alpha'| \ge \lambda > 1$), and the partitions $\cG_\alpha = \{G_{\alpha, w}\}_{w \in \cW}$ are made of subintervals of $[0,1]$, for a given bounded subset $\cW \subset \bZ^d$.

Let $\bP$ be a translation invariant probability on the set $\Omega = A^{{\bZ^d}}$. For a given environment $\bomega \in \Omega$, we have the dynamics $\bF_{\bomega} ( \cdot, \cdot) : \cM \times \bZ^d \to \cM \times \bZ^d$ given by $\bF_{\bomega}(x,z) = (f_{\bomega_{z+e(\bomega_z,x)}}(x), z + e(\bomega_z, x))$, where $e(\alpha, x) = \sum_{w \in \cW} \Id_{G_{\alpha, w}}(x) w$.

We are interested in the quenched evolution, $(x_n,z_n)= \bF^n_{\bomega}(x_0,z_0)$, of such a system when the initial condition $x_0$ is distributed according to the probability measure
\[
\mu(\vf)=\int_0^1 \vf(x,0)h_0(x)\, dx
\] for some $h_0\in {\rm BV}$, with $\inf h_0 > 0$.

Recall the definition of the probability measure $\bP_\star$ and the dynamical system $(\Omega_\star, \cF_\star)$ of the point of view of the particle, from Section \ref{sec:walk-rand}.

Our goal is to reduce the study of this model to a probabilistic one. To this end we need some technical conditions. \\
%%%%%%%%%%%%%%%%%%%%%%%%%%%%%%%%%%
\subsection{ Conditions {\em (C1), (C2), (C3)}}\label{sec:cs} ~ \newline
Let $\cT = \{f_\alpha\}_{\alpha \in A}$ be a finite set of maps on $[0,1]$, and $\cH$ be the (finite) set of all the possible intervals of the partitions, i.e. $\cH = \{G_{\alpha, w}\}_{w \in \cW, \alpha \in A}$.
\begin{notation}The set $\mathcal{T} \times \mathcal{H}$ is canonically isomorphic to $\Sigma=A \times (A \times \cW)$,  with correspondence $\rho : \Sigma \to \cT \times \cH$ given, for $\sigma=(\alpha,\beta)$, by $\rho(\sigma)=(f_\alpha, G_\beta)$. Hence, we can use the notation
$(T_\sigma, H_\sigma)=(f_{\pi_1\circ\rho^{-1}(\sigma)}, G_{\pi_2\circ\rho^{-1}(\sigma)})$, $\pi_1(\alpha,\beta)=\alpha$ and $\pi_2(\alpha,\beta)=\beta$. We will often write $\Sigma=\mathcal{T} \times \mathcal{H}$.
\end{notation}
Let $\tau : \Sigma^{\mathbb{N}} \to \Sigma^{\mathbb{N}}$ be the unilateral shift.
We denote $T_\sigma^n = T_{\sigma_n} \circ \ldots \circ T_{\sigma_1}$ and $H_{\sigma}^n = \bigcap_{j=0}^{n-1} (T_{ \sigma}^j)^{-1}(H_{\sigma_{j+1}})$. 

Let $\mathcal{L}_{\sigma_k}$ be the transfer operator of the map $T_ {\sigma_k}$ with respect to the Lebesgue measure, i.e. 
\[
\mathcal{L}_{\sigma_k} f(x) = \sum_{T_{\sigma_k} y = x} \frac{f(y)}{|T_{\sigma_k} ' (y)|}.
\]

We set $\Hl_{\sigma_k} f = \mathcal{L}_{\sigma_k}(f \mathds{1}_{H_{\sigma_k}})$ and $\Hl_{\sigma}^n = \Hl_{\sigma_n} \circ \ldots \circ \Hl_{\sigma_1}$. We can write $\Hl_{\sigma}^n f(x) = \sum_{T_{\sigma}^n y = x} g_{\sigma}^n (y) f(y)$, where $g_{\sigma}^n = g_{\sigma_1} \times \ldots \times g_{\sigma_n} \circ T_{\sigma}^{n-1}$, with $g_{\sigma_k} = \mathds{1}_{H_{\sigma_k}} \frac{1}{|T_{\sigma_k}'|}$. Let $\Theta^{-1} = \inf_{T \in \mathcal{T}} \inf_x |T'(x)|$, then $\|g_{\sigma}^n\|_{\infty} \le \Theta^n$.

We will only consider systems that satisfy 

\begin{description}
\item[{\em (C1)}]\label{C1}There exists $\ds > 0$ such that for all $\sigma_1 \in \mathcal{T} \times \mathcal{H}$, $\inf \Hl_{\sigma_1} \mathds{1} \ge \ds$.
\end{description}

Observe that this condition is satisfied if, for any choice of $T$ and $H$, $T$ admits at least one full branch inside $H$. By iteration, we also have that $\inf \Hl_\sigma^n \mathds{1} \ge \ds^n$ for any $\sigma \in \Sigma^{\mathbb{N}}$ and $n \ge 1$.
Next, we define the functionals 
\[
\Lambda_{\sigma}(f) = \lim_{n \to \infty} \inf \frac{\Hl_{\sigma}^n f}{\Hl_{\sigma}^n \mathds{1}}.
\]
This limit exists since the sequence is increasing and bounded. Indeed,
\[
\begin{aligned}
\inf \frac{\Hl_\sigma^{n+1}f}{\Hl_\sigma^{n+1} \mathds{1}} & \ge \inf \frac{\Hl_{\sigma_{n+1}}(\Hl_\sigma^n \mathds{1} \frac{\Hl_\sigma^n f}{\Hl_\sigma^n \mathds{1}})}{\Hl_\sigma^{n+1} \mathds{1}} 
\\ & \ge \inf \frac{\Hl_\sigma^n f}{\Hl_\sigma^n \mathds{1}} \inf \frac{\Hl_{\sigma_{n+1}}( \Hl_\sigma^n \mathds{1})}{\Hl_\sigma^{n+1} \mathds{1}}  = \inf \frac{\Hl_\sigma^n f}{\Hl_\sigma^n \mathds{1}} ;
\end{aligned}
\]
and $- \|f \|_\infty \le \inf \frac{\Hl_\sigma^n f}{\Hl_\sigma^n \mathds{1}} \le \|f\|_\infty$.
In particular, for all $n \ge 0$, we have
\begin{equation} \label{eq:monotone_limit}
\Lambda_\sigma(f) \ge \inf \frac{\Hl_{\sigma}^n f}{\Hl_{\sigma}^n \mathds{1}}.
\end{equation}

The $\Lambda_\sigma$ satisfy the following properties:
\begin{itemize}
\item $\Lambda_\sigma( \mathds{1} ) = 1$;
\item $\left| \Lambda_\sigma (f) \right| \le \| f \|_\infty$;
\item $f \ge g$ implies $\Lambda_\sigma(f) \ge \Lambda_\sigma(g)$ (monotonicity);
\item $\Lambda_\sigma( \lambda f) = \lambda \Lambda_\sigma(f)$, for $\lambda > 0$ (positive homogeneity);
\item $\Lambda_\sigma(f + g)  \ge \Lambda_\sigma(f) + \Lambda_\sigma(g)$ (super-additivity);
\item $\Lambda_\sigma(f + b) = \Lambda_\sigma(f) + b$ for all $b \in \bR$.
\end{itemize}
All the above follows immediately from the definition. Note that it is not clear at the moment if $\Lambda_\sigma$ is linear or not.

Next, we define
\begin{equation}\label{eq:rho-def}
\rho_{\sigma} = \Lambda_{\tau \sigma}(\Hl_{\sigma_1} \mathds{1})\;;\quad \rho = \inf_{\sigma\in\Sigma^{\mathbb{N}}} \rho_\sigma.
\end{equation}
Let $\mathcal{Z}_{\sigma}^n$ be the partition of smoothness intervals of $T_{\sigma}^n$, and $\widehat{\mathcal{Z}}_{\sigma}^n$ be the coarsest partition which is finer than $\mathcal{Z}_{\sigma}^n$ and enjoying the property that the elements of the partition are either disjoint from $\overline{H_{\sigma}^{n}}$ or contained in $H_{\sigma}^{n}$.

Let us define the collections of intervals
\begin{equation}\label{eq:variousets}
\begin{split}
\mathcal{Z}_{\sigma, \star}^n &= \{ Z \in \widehat{\mathcal{Z}}_{\sigma}^n \, | \, Z \subset H_{\sigma}^{n} \},\\
\mathcal{Z}_{\sigma, b}^n &= \{ Z \in \mathcal{Z}_{\sigma, \star}^n \, | \, \Lambda_{\sigma}(\mathds{1}_Z) = 0 \}\\ 
\mathcal{Z}_{\sigma, g}^n &= \{ Z \in \mathcal{Z}_{\sigma, \star}^n \, | \, \Lambda_{\sigma}(\mathds{1}_Z) > 0 \}.
\end{split}
\end{equation}

\begin{defin} We will call contiguous two elements of $\cZ_{\sigma, \star}^n$ that are either contiguous in the usual sense, or separated by a connected component of $(H_\sigma^n)^c = \bigcup_{j=0}^{n-1} (T_\sigma^j)^{-1}(H_{\sigma_{j+1}}^c)$.
\end{defin} 

We can now introduce the second condition needed to state our results.

\begin{description}
\item[{\em (C2)}]there exist constants $K \ge 0$ and $\xi \ge 1$ such that for any $n$ and $\sigma \in \Sigma^{\mathbb{N}}$, at most $K \xi^n$ elements of $\mathcal{Z}_{\sigma, b}^n$ are contiguous. In addition, $\theta:= \xi \Theta <\min\{ \rho,1\}$. In particular, $\rho > 0$. 
\end{description}

Let us turn to the third and last condition. 
\begin{description}
\item[{\em (C3($N, N'$))}] $\hat\epsilon(N,N'):=\min\limits_{k \le N} \inf\limits_{\sigma\in\Sigma^{\bN}}\min\limits_{Z \in \mathcal{Z}_{\sigma,g}^k}\inf\limits_{x\in [0,1]}\frac{(\Hl_{\sigma}^{N'} \mathds{1}_Z)(x)}{(\Hl_{\sigma}^{N'} \mathds{1})(x)} > 0$.
\end{description} 

Note that, by \eqref{eq:monotone_limit}, condition {\em C3($N, N'$)} implies 
$\Lambda_\sigma(\mathds{1}_Z) \ge\epsilon(N,N') > 0$ for all $Z \in \mathcal{Z}_{\sigma,g}^n$, $n \le N$, and $\sigma \in \Sigma^{\mathbb{N}}$. This is morally what we need in the following, however in Lemma \ref{lem:ly} the above more precise condition is used.

We remark also that $m'\geq m$ implies $\hat\epsilon(n,m')\geq \hat\epsilon(n,m)$ and $n \ge n'$ implies $\hat\epsilon(n,m) \le \hat\epsilon(n',m)$. Thus, setting
\[
\epsilon_\star(n)=\min\{ \hat\epsilon(n,m)\;:\; m\geq n, \hat\epsilon(n,m)>0\},
\] 
we have $\epsilon_\star(N)>0$ if and only if {\em (C3($N, N'$))} holds for some $N'\geq N$. Also if $\epsilon_\star(n)>0$, then $\epsilon_\star(n')>0$ for all $n'\leq n$, however $\epsilon_\star$ is not necessarily a decreasing function.

\subsection{ The results} ~\\
We now state our main result, whose proof is given in Section \ref{sec:equivalence}. 

\begin{thm} \label{thm:main}
There exists an integers $n_2 \geq  1$, depending only on the classes $\mathcal{T}$ and $\mathcal{H}$, explicitly computable (see Remark \ref{rem:n2_value}), such that if {\em (C1)}, {\em (C2)} and {\em (C3($n_2, n_3$))} hold for some $n_3 \ge n_2$, then the condition {\bf (Exp)} holds.
In particular, the property of loss memory from Lemma \ref{lem:old-forget} is verified.
\end{thm}
%%%%%%
\begin{rem} \label{rem:n2_value} The definition of $n_2$ is a bit cumbersome but explicit:\\
 given ($\mathcal{T}$, $\mathcal{H}$) define the constants (see Lemma \ref{lem:ly} for their meaning)\footnote{  Note that $C_n=\infty$ if $\epsilon_\star(n)=0$.}
 \[
 C_\star =3(C+1)+2K(3C+2); \quad C_n= \frac{(3C+2)(2K\xi +1) \Theta}{\epsilon_\star(n)} 
 \]
where $C$ is such that $\bigvee_Z g_\sigma^n \le C \| g_\sigma^n \|_\infty$ for all $n, \sigma$ and $Z \in \mathcal{Z}_{\sigma, \star}^n$. Recall the definitions of $\rho$ and $\theta$ from {\em (C2)} and define (as in Lemma \ref{lem:cone_inv}) $n_0=  \lceil \frac{\ln 4C_\star^2}{\ln\rho\theta^{-1}}\rceil$. If $\epsilon_\star(n_0)=0$, then define $n_2= n_0$.\footnote{ In this case {\em (C3($n_2, n_3$))} fails for all $n_3\geq n_2$ and hence  Theorem \ref{thm:main} does not apply.} If $\epsilon_\star(n_0)>0$, then $C_n<\infty$ for all $n\leq n_0$, and we can set (as in Remark \ref{rem:a_value}) $a=\max\{1, \frac{15}{11} \max_{i \le n_0} \frac{C_i}{C_\star \theta^i}\}$ and $B = 1+2a C_\star$ (used in Lemma \ref{lem:squeeze}). Finally we define (as introduced in Lemmata \ref{lem:partition},  \ref{lem:inf_partition})
\[
n_2=\left\lceil \frac{\ln{4aB(1+2C_{n_0}\rho^{-n_0})}}{\ln\theta^{-1}\rho}\right\rceil.
\]

 \end{rem}
According to Remark \ref{rem:n2_value}, to check the hypothesis of Theorem \ref{thm:main} one has first to check {\em (C1), (C2)}; compute $n_0$ and  find $N'$ for which {\em (C3($n_0,N'$))} holds; use it to compute $n_2$ and look for an $n_3$ for which {\em (C3($n_2,n_3$))} holds. Given ($\mathcal{T}$, $\mathcal{H}$), this can be rather laborious as we will see in Section \ref{sec:ex-bis}.

As we already pointed out in Section \ref{sec:sup-simp}, the choice of the initial condition might play an important role. The following result states that if we restrict ourselves to initial conditions absolutely continuous to Lebesgue, with density in ${\rm BV}$ bounded uniformly away from 0, this difference is not so important in the sense that for large times, the transition probabilities are exponentially close. For two different initial densities $h_0, h_0' \in {\rm BV}$ with $\inf h_0 > 0$ and $\inf h_0' > 0$, we denote by $\bP_\star$ and $\bP_\star '$ the probability measures corresponding to $h_0$ and $h_0'$ respectively. In section \ref{sec:equivalence} we prove:

\begin{thm} \label{thm:dependence_initial}
Under the assumptions of Theorem \ref{thm:main}, we have for all realisation of the environment $\bomega \in \Omega$, $n \ge 0$ and all densities $h_0, h_0'$ as above:
\[
\left| \mathbb{P}_\star( z(n) \; | \; z(1), \ldots, z(n-1), \bomega) - \mathbb{P}_\star' ( z(n) \; | \; z(1), \ldots, z(n-1), \bomega) \right| \le C_{h_0, h_0 '} \nu^n,
\]
where $C_{h_0, h_0 '} >0$ depends only on the densities $h_0$ and $h_0 '$.
\end{thm}

Next, we consider the situation where all maps $f_\alpha$ preserve a common density $h_0 \in {\rm BV}$ such that $\inf h_0 > 0$, and when the partitions $\cG_\alpha$ are deterministic, i.e. $G_{\alpha, w} = G_w$ does not depend on $\alpha \in A$. In this situation, the dynamical system $(\Omega_\star, \cF_\star, \bP_\star)$ is measure-preserving, and so condition {\bf (Abs)} holds by Lemma \ref{lem:abs_cond}.

\begin{thm} \label{thm:main_ergodic}
Under the assumptions of Theorem \ref{thm:main}, if the maps $f_\alpha$ preserve a common density, the partitions $\cG_\alpha$ are deterministic, and if furthermore condition {\bf (Pro)} of Section \ref{sec:environ2} holds, then the dynamical system $(\Omega_\star, \cF_\star, \bP_\star)$ is ergodic. In particular, $\bP_\star$-a.e., 
\[
\lim_{n \to \infty} \frac{1}{n} z(n) = \sum_{w \in \cW} w \int_{G_w} h_0 dm.
\]
\end{thm}

\begin{rem} The assumption that the maps all preserve a common measure and that the partitions are deterministic is only used to check the validity of condition {\bf (Abs)} thanks to Lemma \ref{lem:abs_cond}, and to have an explicit formula for the drift. If for a concrete example, one is able to check {\bf (Abs)} by any other mean, then Theorem \ref{thm:gibbs_walk_erg} applies and there exists $V \in \bR^d$ such that $\lim_{n \to \infty} \frac{1}{n} z(n) = V$, $\bP_\star$-a.s. .
\end{rem}

 The proofs of Theorem \ref{thm:main_ergodic} will be provided in Section \ref{sec:equivalence}.

%%%%%%%%%%%%%%%%%%%%%%%%%%%%%%%%%%%%%%%%%%%%%%%%%%%%%%%%%%%%%%%
\section{Existence of Non-Markovian examples: $\beta$-maps}\label{sec:ex-bis}
To ensure that the hypotheses of Theorem \ref{thm:main} are non empty it suffices to verify them in some limiting regime, e.g. when the dynamics has a lot of expansion.  This is the aim of the present section. To further simplify things we will limit ourselves to $\beta$ maps, a popular class of dynamical systems.
\subsection{General $\beta$-maps}
More precisely, we consider the situation where the class of maps is $\cT = \{T_{\beta_1}, T_{\beta_2}\}$ for $\beta_2 > \beta_1 > 1$, with $T_\beta(x) = \beta x \mod 1$ and the partitions $\cG_\alpha$ are such that $\cH = \{[0, \frac 1 2], (\frac 1 2, 1]\}$. Note that $\cW$ is not specified, and $\cG_\alpha$ can be random.\footnote{For instance, $\cW = \{- 1, +1\}$ and whether $[0,\frac 1 2]$ and $(\frac 1 2 , 1]$ correspond to $-1, +1$ or $+1,-1$ respectively is random.} For $ \sigma \in \Sigma^{\mathbb{N}} = (\cT \times \cH)^{\bN}$, we denote by $\beta_{\sigma_i}$ the value of $\beta$ such that $T_{\sigma_i} = T_{\beta_{\sigma_i}}$. For simplicity, we will fix $\betavar >1$ and assume that $\beta_1 = \beta$ and $\beta_2 = \betavar \beta$. We will show that assumptions {\em (C1)}, {\em (C2)} and {\em (C3)} are verified for a large set of $\beta$. When needed, we will denote by $\mathcal{Z}_\sigma^n(\beta)$ the partition of smoothness intervals of $T_\sigma^n$, to emphazise the dependence on $\beta$, and similarly for the objects defined in equation \eqref{eq:variousets}. We will do the same with the subsets of the partition defined in Section \ref{sec:model}.

\begin{prop} \label{prop:result_beta}
There exists a set $\cB \subset (1, \infty)$ and $C>0$, with ${\rm Leb}(\cB \cap (1,t)) \leq C \log t$, for $t>1$, such that the model described above satisfies the assumptions of Theorem \ref{thm:main} when $\beta  \notin \cB$.
\end{prop}

The rest of the section is devoted to the proof of Proposition \ref{prop:result_beta}. It suffices to check conditions {\em (C1), (C2)} and {\em (C3)} as explained in Remark \ref{rem:n2_value}.

\subsubsection{Condition {\em (C1)}}

This condition is satisfied if every map $T \in \cT$ admits at least one full branch inside any interval $H \in \cH$. This is the case whenever $\beta \ge 3$.

\subsubsection{Condition {\em (C2)}} \label{sec:condition_c2}

We first give a general criterion to check this condition.  Let $\mathcal{Z}_{\sigma,f}^n$ be the collection of elements $Z$ in $\mathcal{Z}_{\sigma, \star}^n$ such that $T_\sigma^n Z = [0,1]$, and let $\mathcal{Z}_{\sigma,u}^n = \mathcal{Z}_{\sigma,  \star}^n \setminus \mathcal{Z}_{\sigma,f}^n$.

We call a system $\xi$-full branched, $\xi > 0$, if there exists $K >0$ such that for all $\sigma \in \Sigma$ and $n$, the number of contiguous elements in $\mathcal{Z}_{\sigma,u}^n$ is less than $K \xi^n$.
A system $\xi$-full branched satisfies {\em (C2)} with the same $K,\xi$.

\begin{lem} \em Calling $C_\sigma^n$ the maximal number of contiguous elements in $Z_{\sigma,u}^n$, holds $$C_\sigma^n \le 2 \sum_{i=0}^{n-1} (C^{(1)} +2)^i C^{(1)},$$ where $C^{(1)}$ is the supremum over all $\sigma$ of $C_\sigma^1$.
\end{lem}

\begin{proof}
The proof is by induction on $n$. Clearly it is true for $n=1$. Let us suppose it true for $n$. The elements of the partition $\mathcal{Z}_{\sigma, \star}^{n+1}$ are formed by $\{T_{\sigma_1}^{-1} Z \cap Z_1 \}$ where $Z \in \mathcal{Z}_{\tau \sigma, \star}^{n}$ and $Z_1 \in \mathcal{Z}_{\sigma, \star}^{1}$. Now, if $Z_1 \in \mathcal{Z}_{\sigma, f}^{1}$, the elements maintain the same nature, i.e. if $Z \in \mathcal{Z}_{\tau \sigma, f}^{n}$ (resp. $\mathcal{Z}_{\tau \sigma, u}^{n}$) then $T_{\sigma_1}^{-1} Z \cap Z_1 \in \mathcal{Z}_{\sigma, f}^{n+1}$ (resp. $ \mathcal{Z}_{\sigma, u}^{n+1}$). So we have in $Z_1$ at most $C_{\tau\sigma}^n$ contiguous elements of $\mathcal{Z}_{\sigma, u}^{n+1}$. The only problem arises when a block of contiguous elements ends at the boundary of $Z_1$ since in such a case it can still be contiguous to others elements of $\mathcal{Z}_{\sigma, u}^{n+1}$. Yet, if the contiguous elements of $Z_1$ are in $\mathcal{Z}_{\sigma, f}^{1}$, then there can be at most a block of length $2 C_{\tau\sigma}^n$. One must then analyze what can happen if $Z_1 \in \mathcal{Z}_{\sigma, u}^{1}$. In this case, a set of contiguous elements can either have only partial preimage in $Z_1$, hence we get a shorter group of contiguous elements, or all the group can have preimage. In this last case, the worst case scenario is when the elements contiguous to the groups (that must belong to $\mathcal{Z}_{\tau \sigma, f}^{n}$) are cut while taking preimages. This means that at most two new contiguous elements can be generated, but in this case the group must end at the boundary of $Z_1$. Since there are at most $C_\sigma^1$ contiguous elements in $ \mathcal{Z}_{\sigma, u}^{1}$ in this way we can generate at most $C_\sigma^1 (C_{\tau \sigma}^n+2)$ contiguous elements that, again in the worst case scenario, can be contiguous to two blocks belonging to the neighboring elements in $\mathcal{Z}_{\sigma, f}^{1}$. Accordingly, 
\[
C_\sigma^{n+1} \le C_\sigma^1 (C_{\tau \sigma}^n+2) + 2C_{\tau \sigma}^n = (C_\sigma^1+2)C_{\tau \sigma}^n + 2C_\sigma^1 \le 2 \sum_{i=0}^n (C^{(1)} +2)^i C^{(1)},
\]
where we have used the induction hypothesis.
\end{proof}

Hence any system is $\xi$-full branched with $\xi = C^{(1)} +2$ and $K = \frac{2 C^{(1)}}{C^{(1)} +1}$.

Next, we estimate $\rho$. Remark that $\rho_\sigma \ge \inf \Hl_{\sigma_1} \mathds{1}$. Let $N$ to be the minimal number of full branches of $T$ inside $H$ for any $T \in \mathcal{T}$ and $H \in \mathcal{H}$,
\[
\Hl_{\sigma_1} \mathds{1}(x) = \sum_{T_{\sigma_1} y = x} \mathds{1}_{H_{\sigma_1}}(y) \frac{1}{|T_{\sigma_1}'(y)|} \ge \frac{N}{M},
\]
with $M = \sup_{T} |T'(x)|$. Thus $\rho \ge \frac{N}{M}$, and condition {\em (C2)} is satisfied provided $(C^{(1)} +2) \Theta < \frac{N}{M}$. 

For $\beta$ transformations $C^{(1)} = 2$,  $\Theta = \beta^{-1}$, $M = \betavar \beta$ and $N \ge \lfloor \frac{\beta}{2} \rfloor - 1$. Thus, remembering \eqref{eq:rho-def}, $\rho\geq\frac{ \lfloor \frac{\beta}{2} \rfloor - 1}{\varpi \beta}$; $\xi = 4$; $\theta=4\beta^{-1}$;$K = \frac 4 3$ and {\em (C2)} is satisfied if $4 \betavar \beta < \beta ( \lfloor \frac{\beta}{2} \rfloor - 1)$, which happens if $\beta \ge 8 \betavar +4$.

%%%%%%%%%%%%%%%%%%%%%%%%%%%%%%%%%%
\subsubsection{Condition {\em (C3)}} We start with a parameter selection.
\begin{lem} \label{lem:bad_beta_cover} 
For each $m \ge 1$, there exists a set $\cB_m \subset (1,   \infty)$ and $C_m>0$ such that ${\rm Leb} (\cB_m \cap (1,t)) \leq C_m \log t$ for all $t >1$ and if $1<\beta \notin \cB_m$, then
\begin{equation} \label{eq:covering_beta}
\begin{split}
&\forall n \le m, \: \: \forall \sigma \in \Sigma^{\mathbb{N}}, \: \: \forall Z \in \mathcal{Z}_{\sigma, \star}^n(\beta), \\
 &T_{\sigma}^{n+1}(Z \cap H_\sigma^{n+1}) = [0,1] \: \: {\text or } \: \: T_\sigma^n Z \subset H_{\sigma_{n+1}}^c.
\end{split}
\end{equation}
\end{lem}

This lemma, proven shortly, establishes a dichotomy between good and bad elements of $\cZ_{\sigma, \star}^n(\beta)$ when $\beta \notin \cB_m$: either $T_\sigma^n Z \cap H_{\sigma_{n+1}}$ is large enough to cover $[0,1]$ after one more iteration if $Z$ is good, or it is empty if $Z$ is bad.

\begin{cor} If $\beta \notin \bigcup_{k \le n}\cB_k$ then $\hat\epsilon(n,n+1) \geq  \frac{1}{(2\betavar \beta)^{n+1}}$.
\end{cor}
\begin{proof}
By Lemma \ref{lem:bad_beta_cover}, for $Z \in \mathcal{Z}_{\sigma, g}^k(\beta)$, $k \le n$, we have $T_{\sigma}^{n+1}(Z \cap H_\sigma^{n+1}) = T_{\tau^{k+1} \sigma}^{n-k}( T_{\sigma}^{k+1}(Z \cap H_\sigma^{k+1}) \cap H_{\tau^{k+1} \sigma}^{n-k}) =  [0,1]$ and then
\[
\Hl_\sigma^{n+1} \mathds{1}_Z(x) = \sum_{T_\sigma^{n+1} y = x} \frac{\mathds{1}_Z(y) \mathds{1}_{H_\sigma^{n+1}}(y)}{|(T_\sigma^{n+1})'(y)|} \ge \frac{1}{\sup |(T_\sigma^{n+1})'|} \ge \frac{1}{(\betavar \beta)^{n+1}},
\]
and $\Hl_\sigma^{n+1} \mathds{1}(x) \le \mathcal{L}_\sigma^{n+1} \mathds{1}(x) \le 2^{n+1}$, since $\mathcal{L}_{T_{\beta_i}} \mathds{1}(x) \le 2$.
\end{proof}

\begin{proof}[{\bf\emph{Proof of Proposition \ref{prop:result_beta}}}] 
We choose $\cB = (1, \beta_0) \cup \bigcup_{k=1}^6 \cB_k$ for some $\beta_0 > 2\cdot 25^2\varpi+1$.
Since $K = \frac 4 3$, we  have $C_\star = \frac{25}{3}$. Note that if $\beta\geq \beta_0$, then $\lceil \frac{\ln 4C_\star^2}{\ln\rho\theta^{-1}}\rceil<1$ and, by Remark \ref{rem:n2_value}, $n_0=\inf\{n\geq 1\;:\; \epsilon_\star(n)>0\}$. But $\epsilon_\star(1)\geq \hat\epsilon(1,2)\geq (2\varpi\beta)^{-2}$, hence $n_0=1$.
Accordingly, $C_1\leq\frac{280}{3}\varpi^2\beta$;  $a\leq \frac {42}{11}\varpi^2\beta^2$ and $B\leq 1+ \frac{700}{11}\varpi^2\beta^2$ which yields, for all $\beta_0$ large enough, 
\[
n_2 =  \left\lceil \frac{\ln{4aB(1+2C_{1}\rho^{-1})}}{\ln\theta^{-1}\rho}\right\rceil\leq \frac{5\ln\beta+D_1}{\ln\beta-D_2}
\]
for some constants $D_i\geq 0$ depending only on $\varpi$. We can thus choose $\beta_0$ large enough so that $n_2\leq 6$.
\end{proof}

The rest of the section is devoted to the proof of Lemma \ref{lem:bad_beta_cover}.  The basic idea is to discard the $\beta$ for which the $n$-th iterates of elements of $\partial \widehat{Z}_\sigma^n(\beta)$ come too close to $\left\{ 0 , \frac 1 2 , 1 \right\}$.

\begin{lem} \label{lem:image_zn} For all $\sigma \in \Sigma^{\mathbb{N}}$ and $n \ge 1$, one has $T_\sigma^n(\partial \widehat{\mathcal{Z}}_\sigma^n(\beta)) \subset \mathcal{Q}_n(\beta) := \{0,1\} \cup \{ T_{\sigma'}^i(1), T_{\sigma'}^i(\frac 1 2 ) \, / \, i=1, \ldots, n, \, \sigma' \in \Sigma^{\mathbb{N}} \}$. 
\end{lem}

\begin{proof} We proceed by induction, the result being clearly true for $n=1$. Note that $\mathcal{Q}_n(\beta) \cup T_{\tau^n \sigma}^1 (\mathcal{Q}_n(\beta))  \subset \mathcal{Q}_{n+1}(\beta)$. Since every $Z \in \widehat{\mathcal{Z}}_\sigma^{n+1} (\beta)$ is of the form $Z = Z' \cap (T_\sigma^n)^{-1}(Z'')$ for $Z' \in \widehat{\mathcal{Z}}_\sigma^n(\beta)$ and $Z'' \in \widehat{\mathcal{Z}}_{\tau^n \sigma}^1(\beta)$, if $a \in \partial \widehat{\mathcal{Z}}_\sigma^{n+1} (\beta)$, then either $a \in \partial \widehat{\mathcal{Z}}_\sigma^n(\beta)$ or $T_\sigma^n a \in \partial \widehat{\mathcal{Z}}_{\tau^n \sigma}^1(\beta)$. In the first case, $T_\sigma^{n+1 } a \in T_{\tau^n \sigma}^1( T_\sigma^n (\widehat{\mathcal{Z}}_\sigma^n(\beta))) \subset \mathcal{Q}_{n+1}(\beta)$, and in the second case, $T_\sigma^{n+1} a \in T_{\tau^n \sigma}^1 (\partial \widehat{\mathcal{Z}}_{\tau^n \sigma}^1(\beta)) \subset \mathcal{Q}_1(\beta) \subset \mathcal{Q}_{n+1}(\beta)$.
\end{proof}

We thus see that we need to control all the orbits of $1$ and $\frac 1 2$. To this end, for $x\in (0,1]$ and $\sigma \in \Sigma^{\mathbb{N}}$,\footnote{Note that we will only consider $x =  \frac 1 2$ and $x = 1$ in the following.} we introduce the map $\phi_n^{x,\sigma} : (1,\infty) \to [0,1]$  defined by $\phi_n^{x, \sigma}(\beta) = T_\sigma^n (x)$.\footnote{Recall that $T_{\sigma_i} = T_{\beta_{\sigma_i}}$ and $\beta_{1}=\beta$, $\beta_{2}=\varpi\beta$.} For non negative integers $i_1, \ldots, i_n$, we define 
\[
I_{i_1, \ldots, i_n}^{x, \sigma} = \left\{ \beta \in (1, \infty)\, | \, \lfloor \beta_{\sigma_k} \phi_{k-1}^{x,\sigma} (\beta) \rfloor = i_k, \:\:k=1, \ldots, n \right\}.
\]
Note that $I_{i_1}^{x,\sigma} = [\frac{i_1}{x}, \frac{i_1+1}{x})$ if $\beta_{\sigma_1} = \beta_1$,  $I_{i_1}^{x,\sigma} = [\frac{i_1}{\betavar x}, \frac{i_1+1}{\betavar x})$ if $\beta_{\sigma_1} = \beta_2$, and that in both cases, $I_{i_1}^{x,\sigma} \subset [\frac{i_1}{\betavar x}, \frac{i_1 +1}{x})$. The family $\{I_{i_1, \ldots, i_n,i_{n+1}}^{x, \sigma}\}_{i_{n+1} \ge 0}$ forms a partition into finitely many (at most $ \lfloor \betavar \frac{i_1 + 1}{x} \rfloor +1$) intervals of $I_{i_1, \ldots, i_n}^{x,\sigma}$. From the relation $\phi_{n+1}^{x, \sigma}(\beta) = \beta_ {\sigma_{n+1}} \phi_n^{x, \sigma}(\beta) \mod 1$, we deduce easily by induction:

\begin{lem} \label{lem:derivative_beta} The map $\phi_n^{x, \sigma}$ is $\mathcal{C}^1$ and strictly increasing on each interval $I_{i_1, \ldots, i_n}^{x,\sigma}$, and verifies $(\phi_n^{x,\sigma})'(\beta) \ge \beta^{n-1} x$.
\end{lem}

\begin{lem} \label{lem:bad_beta_x} For each $x \in (0,1]$  and $n \ge 1$, there exists $C_{x,n}>0$ and a set $\mathcal{B}_{n,x} \subset (1, \infty)$, with ${\rm Leb}(\mathcal{B}_{n,x}\cap (1,t)) \leq C_{x,n}\log t$, such that, if $\beta \notin \mathcal{B}_{n,x}$,
\[
\forall \sigma \in \Sigma^{\mathbb{N}}, \:\: d(\phi_n^{x, \sigma}(\beta),\{0, 1/ 2 , 1 \}) > 3 \beta^{-1}.
\]
\end{lem}

\begin{proof}  Fix $n \ge 1$ and $\sigma \in \Sigma$, and consider $\beta \in I_{i_1, \ldots, i_n}^{x, \sigma}$. We have $\beta^{-1} \le\betavar x i_1^{-1}$ and thus $d(\phi_n^{x, \sigma}(\beta), \{0, 1/ 2, 1\}) > 3 \beta^{- 1}$ whenever 
\begin{equation} \label{eq:bad_beta}
d(\phi_n^{x, \sigma}(\beta), \{0, 1/ 2, 1\}) > 3 \betavar x i_1^{-1}. 
\end{equation}
Define $\mathcal{B}_{n, i_1}^{x, \sigma}$ to be the set of $\beta$ in $I_{i_1}^{x, \sigma}$ which do not satisfy \eqref{eq:bad_beta}.
By Lemma \ref{lem:derivative_beta}, the Lebesgue measure of $ \mathcal{B}_{n, i_1}^{x, \sigma} \cap I_{i_1, \ldots, i_n}^{x, \sigma}$ is less 
than\footnote{In the following, $C_{\betavar, x, n}$ will denote a constant, the value of which may change from one line to another, depending on $\betavar$, $x$ and $n$, but not on $i_1$.}
\[
3 \left( \inf_{I_{i_1, \ldots, i_n}^{x, \sigma}}  |(\phi_n^{x, \sigma})'|\right)^{-1}3 \betavar x  i_1^{-1} \le C_{\betavar, x, n} i_1^{-n}.
\]
As $\{I_{i_1, \ldots, i_n}^{x,\sigma}\}_{i_2, \ldots, i_n}$ forms a partition of $I_{i_1}^{x, \sigma}$ into at most $C_{\betavar, x, n} i_1^{n-1}$ elements, the set $\mathcal{B}_{n, i_1}^{x, \sigma}$ has a measure less than $C_{\betavar, x, n} i_1^{-1} $. For $j=1,2$, set $\mathcal{B}_{n,i_1}^{x,j} = \bigcup_{\{\sigma \, | \, \beta_{\sigma_1} = \beta_j\}} \mathcal{B}_{n, i_1}^{x, \sigma}$. For $n$ fixed, the condition \eqref{eq:bad_beta} depends on $\sigma$ only through its $n$ first terms, and thus $\mathcal{B}_{n,i_1}^{x,j}$ is of measure less than $C_{\betavar, x, n} i_1^{-1}$. The set $\mathcal{B}_{n,x} = \bigcup_{j=1,2} \bigcup_{i_1 \ge 1} \mathcal{B}_{n,i_1}^{x,j}$ then satisfies the conclusion of the lemma.
\end{proof}

We can now conclude the proof:

\begin{proof}[{\bf\emph{Proof of Lemma \ref{lem:bad_beta_cover}}}] 
We set $\cB_m = \cup_{1 \leq n\leq m}\cB_{n,1/2}\cup \cB_{n,1}$ and we proceed by induction over $n \le m$. Note that if $Z \in \cZ_{\sigma, \star}^n(\beta)$, then $Z \cap H_{\sigma}^k = Z$ for all $k \le n$. 

If $Z \in \cZ_{\sigma, \star}^1(\beta)$, then either $Z$ is a full interval of $\cZ_{\sigma, \star}^1 (\beta)$ and so $T_\sigma^2(Z \cap H_\sigma^2) = T_{\sigma_2}([0,1] \cap H_{\sigma_2}) = [0,1]$, or one of the endpoints of $Z$ is $\frac 1 2$ or $1$. In the latter case, the other endpoint of $Z$ must be sent after one iteration to $0$ or $1$ and so if $ T_\sigma^1 Z  \cap H_{\tau \sigma}^1$ is not empty, then $| T_\sigma^1 Z  \cap H_{\tau \sigma}^1| > 3 \beta^{-1}$ by Lemma  \ref{lem:bad_beta_x}. The interval $T_\sigma^1 Z \cap H_{\tau \sigma}^1$ therefore contains at least one full interval of $\cZ_{\tau \sigma, \star}^1(\beta)$, which implies $T_\sigma^2(Z \cap H_\sigma^2) = [0,1]$.

Now, we suppose that \eqref{eq:covering_beta} holds for $n$ and we prove it  holds for $n+1< m$. 

Any $Z \in \cZ_{\sigma, \star}^{n+1}(\beta)$ is of the form $Z = Z' \cap (T_\sigma^n)^{-1}(Z'')$ with $Z' \in {\cZ}_{\sigma, \star}^n(\beta)$ and $Z'' \in {\cZ}_{\tau^n \sigma, \star}^1 (\beta)$. If both endpoints of $Z$ belong to the interior of $Z'$, then $T_\sigma^n Z = Z'' \in {\cZ}_{\tau^n \sigma, \star}^1 (\beta)$ and we have 
\[
T_\sigma^{n+2}(Z \cap H_\sigma^{n+2}) = T_{\tau^n \sigma}^2(T_\sigma^n Z \cap H_{\tau^n \sigma}^2) = T_{\tau^n \sigma}^2(Z'' \cap H_{\tau^n \sigma}^2) = [0,1],
 \]
or
\[
T_\sigma^{n+1} Z = T_{\tau^n \sigma}^1 (T_\sigma^n Z) = T_{\tau^n \sigma}^1 Z'' \subset H_{\sigma_{n+1}}^c,
\]
according to whether $Z''$ is a good a or bad element of $\cZ_{\tau^n \sigma, \star}^1(\beta)$ respectively.

If one endpoint of $Z$ is also an endpoint of $Z'$, then $T_\sigma^n Z = [T_\sigma^n a, b]$\footnote{We write [x,y] to denote the interval joining $x$ and $y$, disregarding whether $x \le y$ or $x \ge y$.} with $a \in \partial {\cZ}_{\sigma, \star}^n(\beta)$ and $b \in \partial {\cZ}_{\tau^n \sigma, \star}^1(\beta)$.
By Lemma \ref{lem:image_zn}, $T_\sigma^n a = T_{\sigma '}^i x$, with $\sigma' \in \Sigma$, $0 \le i \le n$ and $x \in \left\{\frac 1 2 , 1 \right\}$.\footnote{Note that if $T_\sigma^n a = 0$, then we are reduced to the previous situation.} We consider two subcases: either $b \notin \left\{ \frac 1 2 , 1 \right\}$ or $b \in \left\{ \frac 1 2 , 1 \right\}$.

In the first subcase, we have $y :=T_{\sigma_{n+1}}b \in \left\{ 0 , 1 \right\}$. Therefore, $T_\sigma^{n+1} Z = [T_{\sigma_{n+1}} T_{\sigma'}^i x, y]$. Since $i+1 \le n+1 \le m$, $\beta \notin \cB_{n+1, y}$ and one of the endpoints of $T_\sigma^{n+1} Z$ belongs to $\left\{0,  1\right\}$, it follows that $T_\sigma^{n+1} Z \cap H_{\sigma_{n+2}}$, if it is non empty, is an interval of length strictly larger than $3 \beta^{-1}$ by Lemma \ref{lem:bad_beta_x} and thus contains at least one full interval of ${\cZ}_{\tau^{n+1} \sigma, \star}^1(\beta)$. Consequently, $T_{\sigma}^{n+2}(Z \cap H_{\sigma}^{n+2}) = [0,1]$.

In the second subcase, $T_\sigma^n Z = [T_{\sigma'}^i x, b]$ has one endpoint belonging to $\left\{ 0, \frac 1 2 ,1\right\}$, and so, as above, we obtain that $T_\sigma^n Z \cap H_{\sigma_n}$, if it is non empty, is of length strictly larger than $3 \beta^{-1}$ by Lemma \ref{lem:bad_beta_x}. We deduce that $T_\sigma^n Z \cap H_{\sigma_n}$ contains at least one full interval of $\cZ_{\tau^n \sigma, \star}^1(\beta)$, which implies that $T_\sigma^{n+2}(Z \cap H_\sigma^{n+2}) = [0,1]$.

Finally, if both endpoints of $Z$ are also endpoints of $Z'$, then $Z = Z' \in {\cZ}_{\sigma, \star}^n (\beta)$. Since $Z \in \cZ_{\sigma, \star}^{n+1}(\beta)$, we have $T_\sigma^n Z \subset H_{\sigma_{n+1}}$ and so by our induction hypothesis, we can only have $T_\sigma^{n+1} Z = T_\sigma^{n+1}(Z' \cap H_\sigma^{n+1}) = [0,1]$. This implies $T_{\sigma}^{n+2}(Z \cap H_\sigma^{n+2}) = [0,1]$ and concludes the induction.
\end{proof}
\subsection{ Markov maps with non-Markov gates}\label{sec:markovdetg}
The reader might wonder if it is possible to produce an example more similar to the Lorentz gas. In particular, one in which the invariant measure of the maps is always the same and the gates are deterministic, so that one knows explicitly the invariant measure of the process of the environment as seen from the particle. 

This is indeed possible: let $\cT = \{T_{\beta_ 1}, T_{\beta_2}\}$, with $\beta_2 > \beta_1$ both integers, and  partitions $\cG_\alpha = \cG$ deterministic with $\cW = \{-1, 0, +1\}$ and $G_{-1} = [0, y]$, $G_{ 0} = (y, 1-y]$ and $G_{ +1} = (1-y, 1]$, for $0< y <\frac 1 2$. Proceeding similarly to the previous section, but using $y$ as a parameter, instead of $\beta$, we have:

\begin{prop} \label{prop:result_beta_gates}
Let $\betavar>1$. Then for each $1 < \beta_1 < \beta _2 < \betavar \beta_1$ integers, there exists a measurable set $\cB_{\beta_1, \beta_2} \subset (0, \frac 1 2 )$, with ${\rm Leb}(\cB_{\beta_1, \beta_2}) = \cO(  \beta_1^{-1})$ as $\beta_1 \to \infty$, such that the model described above, with $y\not\in \cB_{\beta_1, \beta_2}$, satisfies the assumptions of Theorems \ref{thm:main}.
\end{prop}

In particular, this class of models is non-empty when $\beta_1$ is large enough. Since all maps in $\cT$ preserve the Lebesgue measure and all gates are deterministic, Theorem \ref{thm:main_ergodic} also applies\footnote{Condition {\bf (Pro)} is satisfied here, since $\cW$ is symmetric.} and we therefore have $\frac{1}{n}z(n) \to 0$ a.e., since the drift is equal to 
\[
V = \sum_{w \in \cW} w \int_{G_w} h_0 dm = \sum_{w \in \{ -1, 0 , +1\}} w |G_w| = 0.
\]
By corollary \ref{cor:recurrence}, the walk $(z_n)$ is then recurrent.

%%%%%%%%%%%%%%%%%%%%%%%%%%%%%%%%%%%%%%%%%%%%%%%%%%%%%%%%%%%%%%%
\section{Equivalence with a Gibbs random walk}\label{sec:equivalence}
In this section, we prove Theorem \ref{thm:main}. The proof will rely on a property of exponential loss of memory for compositions of the operators $\Hl_{\sigma}$. More precisely, we will investigate the properties of compositions of the form $\Hl_\sigma^n f$, in order to understand better the asymptotics of the probabilities $p(\bomega, n, w_0  \ldots w_{n-1})$.
For convenience, we will consider bi-infinite sequences $\sigma \in \Sigma^{\bZ} = (\cT \times \cH)^{\bZ}$, with $\tau : \Sigma^{\bZ} \to \Sigma^{\bZ}$ the bilateral shift. We can extend the definitions of $\Hl_\sigma^n$, $\rho_\sigma$, $C_\sigma^a$ and $\Lambda_\sigma$ to the case $\sigma \in \Sigma^{\bZ}$ in a straightforward way.

What we mean by loss of memory is made precise by the following result proven  in Section \ref{sec:decay}.
\begin{prop} \label{prop:decay} Under the hypothesis of Theorem \ref{thm:main} there exist $\nu  \in (0,1)$, a family of positive numbers $\{\rho_\sigma \}_{\sigma  \in \Sigma^{\bZ}}$ and a family of positive functions $\{ h_\sigma \}_{\sigma  \in \Sigma^{\bZ}}$ in ${\rm BV}$ such that
$\Lambda_\sigma(h_\sigma)=1$ and for all $\sigma \in \Sigma^{\bZ}$, $f \in {\rm BV}$ and $n \ge 0$:
\begin{equation} \label{eqn:exp_convergence}
\left\| \frac{\Hl_\sigma^n f}{\rho_\sigma \cdots \rho_{\tau^{n-1} \sigma}} - \Lambda_{\sigma}(f) h_{\tau^n \sigma} \right\|_\infty \le \Const \nu^n \| f \|_{\rm BV}.
\end{equation}
\end{prop}
\begin{rem} The above statement is similar to the one in \cite{MO14}. Note however that here the setting is very different insofar in \cite{MO14} only small holes and nearby maps are considered. The upgrade of the result to large holes and arbitrary maps (as we inescapably need) turns out to be highly non trivial.
\end{rem}
In the following we will need the following information on $\Lambda_\sigma$ and $h_\sigma$, which are proven at the end of Section \ref{sec:decay}.
\begin{lem} \label{lem:property_inv_density} Under the hypothesis of Theorem \ref{thm:main} there exist $\exalpha>0$ such that, for all $n \ge 0$ and $\sigma \in \Sigma^{\mathbb{Z}}$, $\inf h_\sigma \ge \exalpha$ and $\|h_\sigma\|_{\infty} \le 1+a$.\footnote{ The constant $a$ is defined in Remark \ref{rem:n2_value}} In addition, $\sup_\sigma \rho_\sigma < \infty$ and for all $f\in BV$, $\Lambda_{\tau^n \sigma}(\Hl_\sigma^n f) = \rho_\sigma \ldots \rho_{\tau^{n-1} \sigma}\Lambda_{\sigma}(f)$; $\Hl_\sigma^n h_\sigma = \rho_\sigma \ldots \rho_{\tau^{n-1} \sigma}h_{\tau^n\sigma}$. 
\end{lem}

To prove Proposition \ref{prop:decay}, we will adapt the strategy of \cite{liverani2003lasota}. More precisely, we will show that the family of cones\footnote{By using a generalised variation, e.g. \cite{Ke85, Bu13}, one could probably treat more singular maps. We refrain from exploring this possibility to keep the exposition simpler.}
\[
C_\sigma^a = \{h \in {\rm BV} \; | \; h \neq 0, h\ge 0, \bigvee h \le a \Lambda_\sigma(h)\}
\]
is stricly invariant under compositions of large enough length of transfer operators (i.e. $\Hl_\sigma^n C_\sigma^a \subset C_{\tau^n \sigma}^{a/2}$ for all $n \ge n_0$) for a suitable $a>0$, see Lemma \ref{lem:cone_inv}. From this, we will deduce that $\Hl_\sigma^n C_\sigma^a$ has uniform finite diameter in $C_{\tau^n \sigma}^a$ for the corresponding Hilbert metric (Lemma \ref{lem:finite_diameter}), which will imply that $\Hl_\sigma^n$ is a strict contraction for the Hilbert metric, and then enjoys exponential loss of memory.

To deduce Theorem \ref{thm:main} from Proposition \ref{prop:decay}, we will also need the following technical lemma, which will be proven at the end of Section \ref{sec:decay}:
\begin{lem} \label{lem:tech_lemma} There exists $\exalpha' >0$ such that for all $n \ge 0$ and $\sigma \in \Sigma^{\mathbb{N}}$, we have $\inf \Hl_\sigma^n \mathds{1} \ge \exalpha' \Lambda_{\tau^n \sigma}(\Hl_\sigma^n \mathds{1})$ and $\| \Hl_\sigma^n \mathds{1} \|_\infty \le B \Lambda_{\tau^n \sigma}(\Hl_\sigma^n \mathds{1})$.\footnote{ The constant B is defined in Remark \ref{rem:n2_value}.}
\end{lem}

Recall from Section \ref{sec:walk-rand} that the transition probabilities are given by \eqref{eqn:transition_proba}. For $(\bomega, \bw) \in \Omega_\star$, we define $\sigma = \sigma(\bomega, \bw) \in \Sigma^\bN$ such that $\Hl_{\sigma_n} = \cL_{\bomega, z_{n-1}, \bw_{n-1}}$, so that
\[
p(\bomega, n, \bw_0 \ldots \bw_{n-1}) = \int \Hl_\sigma^n h_0 dm.
\]
We will still denote by $\sigma \in \Sigma^{\bZ}$ an arbitrary element of $\Sigma^{\bZ}$ which coincides with $\sigma$ for future components (for instance, given an arbitrary $\sigma_\star \in \Sigma^{\bZ}$, we identify $\sigma \in \Sigma^\bN$ with the element $\widetilde{\sigma} \in \Sigma^{\bZ}$ defined by $\widetilde{\sigma}_i = \sigma_i$ if $i \ge 1$ and  $\widetilde{\sigma}_i =(\sigma_\star)_i$ if  $i \le 0$).

We are now ready to prove the announced results.

\begin{proof}[{\bf\emph{Proof of Theorem \ref{thm:main}}}]
We have 
\begin{equation} \label{eqn:cond_prob1}
\bP_{\bomega}(\bw_n \; | \; \bw_{0} \ldots \bw_{n-1}) = \frac{p(\bomega, n+1, \bw_0 \ldots \bw_n)}{p(\bomega, n, \bw_0 \ldots \bw_{n-1})} = \frac{\int \Hl_\sigma^{n+1} h_0 dm}{\int \Hl_\sigma^n h_0 dm} 
\end{equation}
and
\begin{equation} \label{eqn:cond_prob2}
\begin{split}
\bP_{\tau_{z_m} \bomega}(\bw_n \; | \; \bw_{m} \ldots \bw_{n-1}) &= \frac{p(\tau_{z_m} \bomega, n -m +1, \bw_m \ldots \bw_n)}{p(\tau_{z_m} \bomega, n-m, \bw_m \ldots \bw_{n-1})}\\
& = \frac{\int \Hl_{\tau_{z_m}\sigma}^{n -m+1} h_0 dm}{\int \Hl_{\tau_{z_m}\sigma}^{n-m} h_0 dm} .
\end{split}
\end{equation}
We can then write 
\[
\frac{\int \Hl_\sigma^{n+1} h_0 dm}{\int \Hl_\sigma^n h_0 dm} = \rho_{\tau^n \sigma} \frac{\int \Hl_\sigma^{n+1} h_0 dm}{\rho_\sigma \ldots \rho_{\tau ^n \sigma}} \frac{\rho_\sigma \ldots \rho_{\tau^{n-1} \sigma}}{\int \Hl_\sigma^n h_0 dm}.
\]
Using Proposition \ref{prop:decay}, we have
\begin{equation} \label{eqn:exp_convergence_integral}
\left| \frac{\int \Hl_\sigma^n h_0 dm}{\rho_\sigma \cdots \rho_{\tau^{n-1} \sigma}} - \Lambda_\sigma(h_0) \int h_{\tau^n \sigma} dm \right| \le \Const \nu^n \| h_0 \|_{\rm BV},
\end{equation}
which also implies
\begin{equation} \label{eqn:exp_convergence_integral2}
\left| \frac{\rho_{\sigma} \cdots \rho_{\tau^{n-1} \sigma}}{\int \Hl_{\sigma}^{n} h_0 dm} - \left( \Lambda_{\sigma}(h_0) \int h_{\tau^n \sigma} dm \right)^{-1}\right| \le \Const \nu^{n} \frac{\| h_0 \|_{\rm BV}}{(\inf h_0)^2},
\end{equation}
since,  by Lemmata \ref{lem:property_inv_density} and \ref{lem:tech_lemma},
\[
\begin{aligned}
\frac{\int \Hl_{\sigma}^{n} h_0 dm}{\rho_{ \sigma} \cdots \rho_{\tau^{n-1} \sigma}} \ge (\inf h_0) \frac{\inf \Hl_{\sigma}^{n} \mathds{1}}{\rho_{ \sigma} \cdots \rho_{\tau^{n-1} \sigma}}\ge (\inf h_0) \frac{ \exalpha' \Lambda_{\tau^n  \sigma}(\Hl_\sigma^n \mathds{1})}{\rho_\sigma \ldots \rho_{\tau^{n-1} \sigma}} = \exalpha' (\inf h_0),
\end{aligned}
\]
and $\Lambda_{\sigma}(h_0) \int h_{\tau^n \sigma} dm \ge \exalpha \inf h_0$.

Note that we also have, using again Lemmata \ref{lem:tech_lemma} and \ref{lem:property_inv_density},
\[
\frac{\int \Hl_\sigma^n h_0 dm}{\rho_\sigma \cdots \rho_{\tau^{n-1} \sigma}} \le \| h_0\|_\infty \frac{\| \Hl_\sigma^n \mathds{1}\|_\infty}{\rho_\sigma \cdots \rho_{\tau^{n-1} \sigma}} \le B \|h_0\|_\infty.
\]
Consequently, using \eqref{eqn:exp_convergence_integral} with $n$ replaced by $n+1$, \eqref{eqn:exp_convergence_integral2}, Lemma \ref{lem:property_inv_density} and the above inequalities:
\[
\begin{aligned}
&\left|\frac{\int \Hl_\sigma^{n+1} h_0 dm}{\int \Hl_\sigma^n h_0 dm} - \rho_{\tau^n\sigma} \frac{\int h_{\tau^{n+1} \sigma} dm}{\int h_{\tau^n \sigma} dm} \right| \le  \left[ \left| \frac{\int \Hl_\sigma^{n+1} h_0 dm}{\rho_\sigma \cdots \rho_{\tau^{n} \sigma}} - \Lambda_\sigma(h_0) \int h_{\tau^{n+1} \sigma} dm \right|\right. \\
& \left.\times\frac{\rho_{\sigma} \cdots \rho_{\tau^{n} \sigma}}{\int \Hl_{ \sigma}^{n} h_0 dm} + \rho_{\tau^n \sigma}\Lambda_\sigma(h_0) \int h_{\tau^{n+1} \sigma} dm \left|\frac{\rho_{\sigma} \cdots \rho_{\tau^{n-1} \sigma}}{\int \Hl_{\sigma}^{n} h_0 dm} - \frac{1}{\Lambda_{\sigma}(h_0) \int h_{\tau^n \sigma} dm } \right| \right]  \\ 
& \le \Const \nu^n \left( \frac{\|h_0\|_{\rm BV}}{ \inf h_0} + \frac{\|h_0\|_{\rm BV}^2}{(\inf h_0)^2} \right).
\end{aligned}
\]
Hence, there exists $C_{h_0}$ depending only on $h_0$ such that, for all $n \ge 0$, 
\begin{equation} \label{eqn:sup_proba}
\sup_{\sigma \in \Sigma} \left| \frac{\int \Hl_\sigma^{n+1} h_0 dm}{\int \Hl_\sigma^n h_0 dm} - \rho_{\tau^n \sigma} \frac{\int h_{\tau^{n+1} \sigma} dm}{\int h_{\tau^n \sigma} dm} \right| \le C_{h_0} \nu^n.
\end{equation}
Equation \eqref{eqn:sup_proba}, by \eqref{eqn:cond_prob1} and \eqref{eqn:cond_prob2}, implies {\bf (Exp)}.
\end{proof}

\begin{proof}[{\bf\emph{Proof of Theorem \ref{thm:dependence_initial}}}]
By \eqref{eqn:sup_proba}, we have
\[
\sup_{\sigma \in \Sigma} \left| \frac{\int \Hl_\sigma^{n+1} h_0 dm}{\int \Hl_\sigma^n h_0 dm} - \frac{\int \Hl_\sigma^{n+1} h_0 ' dm}{\int \Hl_\sigma^n h_0 ' dm} \right| \le (C_{h_0} + C_{h_0 '}) \nu^n,
\]
for all $n \ge 0$, and the theorem follows with $C_{h_0, h_0 '} = C_{h_0} + C_{h_0 '}$.
\end{proof}

\begin{proof}[{\bf\emph{Proof of Theorem \ref{thm:main_ergodic}}}]
By Theorem \ref{thm:gibbs_walk_erg}, the discussion in Section \ref{sec:app_dwre}, and since {\bf (Exp)} already holds by Theorem \ref{thm:main}, it is enough to show that assumptions {\bf (Pos)} and {\bf (Ell)} are satisfied for the probabilities defined by \eqref{eqn:transition_proba}.
\\ ~
\paragraph{ \em Verification of {\bf (Pos)}}
By Proposition \ref{prop:decay}, Lemma \ref{lem:property_inv_density} and equation \eqref{eq:rho-def}
\[
\begin{aligned}
&p(\bomega, n, \bw_0 \ldots \bw_{n-1}) = \int \Hl_\sigma^n h_0 dm  = \rho_\sigma \cdots \rho_{\tau^{n-1} \sigma} \frac{\int \Hl_\sigma^n h_0 dm}{\rho_\sigma \cdots \rho_{\tau^{n-1} \sigma}} \\ 
& \ge  \rho^n  \left( \Lambda_\sigma(h_0) \int h_{\tau^n \sigma} - \Const \nu^n \|h_0\|_{\rm BV} \right) 
 \ge  \rho^n  \left( \exalpha \inf h_0 - \Const \nu^n \|h_0\|_{\rm BV} \right).
\end{aligned}
\]
Consequently, $p(\bomega, n, \bw_0 \ldots \bw_{n-1}) > 0$ for all $n$ large enough, and since this quantity is non-increasing, this proves the positivity for all $n \ge 0$.
\\ ~
\paragraph{ \em Verification of {\bf (Ell)}}
By \eqref{eqn:sup_proba} and Lemma \ref{lem:property_inv_density}, we have for all $n \ge 0$,
\[
\begin{aligned}
\bP_{\bomega} (\bw_n \; | \; \bw_{n-1} \ldots \bw_0) & = \frac{\int \Hl_\sigma^{n+1} h_0 dm}{\int \Hl_\sigma^n h_0 dm}  \ge \rho_{\tau^n \sigma} \frac{\int h_{\tau^{n+1} \sigma} dm}{\int h_{\tau^n \sigma} dm} - C_{h_0} \nu^n \\
& \ge \rho \frac{\exalpha}{1+a} - C_{h_0} \nu^n,
\end{aligned}
\]
which proves {\bf (Ell)}.
\end{proof}
%%%%%%%%%%%%%%%%%%%%%%%%%%%%%%%%%%%%%%%%%%%%%%%%%%
\section{Loss of Memory}\label{sec:decay}

To prove Proposition \ref{prop:decay}, we will adapt the strategy of \cite{liverani2003lasota} to our non-stationary case, and employ the theory of Hilbert metrics, that we recall below. 

\begin{defin} Let $\mathcal{V}$ be a vector space. We will call convex cone a subset $\mathcal{C} \subset \mathcal{V}$ which enjoys the following properties:

\begin{enumerate}[(i)]

\item $\mathcal{C} \cap - \mathcal{C} = \emptyset$.
\item $\forall \lambda > 0$, $\lambda \mathcal{C} = \mathcal{C}$.
\item $\mathcal{C}$ is a convex set.
\item $\forall f,g \in \mathcal{C}$, $\forall \alpha_n \in \mathbb{R}$ $\alpha_n \to \alpha$, $g - \alpha_n f \in \mathcal{C} \Rightarrow g - \alpha f \in \mathcal{C} \cup \{0\}$.

\end{enumerate}
\end{defin}

We now define the Hilbert metric on $\mathcal{C}$:

\begin{defin} The distance $d_{\mathcal{C}}(f,g)$ between two points $f,g$ in $\mathcal{C}$ is given by

\[
\begin{aligned}
&\alpha(f,g) &= &\sup \{ \lambda > 0 \; | \; g - \lambda f \in \mathcal{C}\}, \\
&\beta(f,g) &= &\inf \{ \mu > 0 \; | \; \mu f - g \in \mathcal{C}\}, \\
&d_{\mathcal{C}}(f,g) &= &\log \frac{\beta(f,g)}{\alpha(f,g)},
\end{aligned}
\]
where we take $\alpha=0$ or $\beta = \infty$ when the corresponding sets are empty.
\end{defin}

The next theorem shows that every positive linear operator is a contraction, provided that the diameter of the image is finite.

\begin{thm}[{\cite[Theorem 1.1]{Liverani95-2}}] \label{thm:contraction_cone} Let $\mathcal{V}_1$ and $\mathcal{V}_2$ be two vector spaces, $\mathcal{C}_1 \subset \mathcal{V}_1$ and $\mathcal{C}_2 \subset \mathcal{V}_2$ two convex cones and $L : \mathcal{V}_1 \to \mathcal{V}_2$ a positive linear operator (which implies $L(\mathcal{C}_1) \subset \mathcal{C}_2$). If we denote 
\[
\Delta = \sup_{f,g \in L(\mathcal{C}_1)} d_{\mathcal{C}_2}(f,g),
\]
then
\[
d_{\mathcal{C}_2}(Lf,Lg) \le \tanh\left( \frac{\Delta}{4} \right) d_{\mathcal{C}_1}(f,g) \text{ } \forall f,g \in \mathcal{C}_1.
\]
\end{thm}

The following lemma links the Hilbert metric to suitable norms on $\mathcal{V}$:

\begin{lem}[{\cite[Lemma 2.2]{LSV98}}] \label{lem:link_hilbert} Let $\| \cdot \|$ be a norm on $\mathcal{V}$ such that
\[
\forall f,g \in \mathcal{V} \text{ } g-f, g+f \in \mathcal{C} \Rightarrow \|f \| \le \|g\|
\]
and let $\ell: \mathcal{C} \to \mathbb{R}^{+}$ be a homogeneous and order preserving function, i.e.
\[
\begin{aligned}
&\forall f \in \mathcal{C}, \forall \lambda \in \mathbb{R}^+& &\ell(\lambda f) = \lambda \ell(f), \\
&\forall f,g, \in \mathcal{C}& &g-f \in \mathcal{C} \Rightarrow \ell(f) \le \ell(g),
\end{aligned}
\]
then
\[
\forall f, g \in \mathcal{C} \text{ } \ell(f) = \ell(g) > 0 \Rightarrow \| f - g \| \le (e^{d_{\mathcal{C}}(f,g)} - 1) \min(\|f\|, \|g\|).
\]

\end{lem}

From now on, we will always assume that conditions {\em (C1)} and {\em (C2)} hold. Our main tool will be the following Lasota-Yorke type inequality:
\begin{lem} \label{lem:ly} If condition {\em (C3($N, N'$))} holds, then for any $n \le N$, for any $\sigma \in \Sigma^{\mathbb{N}}$ and $h \in {\rm BV}$, we have 
\[
\bigvee \Hl_\sigma^n h \le C_\star (\xi \Theta)^n \bigvee h + C_N \Lambda_\sigma (|h|),
\] 
where $C_\star =3(C+1)+2K(3C+2)$,  $C_N= \frac{(3C+2)(2K\xi +1) \Theta}{\epsilon_\star(N)}$, and $C$ is such that $\bigvee_Z g_\sigma^n \le C \| g_\sigma^n \|_\infty$ for all $n, \sigma$ and $Z \in \mathcal{Z}_{\sigma, \star}^n$.\footnote{The constant $C$ exists by the usual bounded distortion estimates. In particular, note that if all the maps in $\mathcal{T}$ are piecewise linear, then $C=0$.}
\end{lem}
\begin{proof}
Remark that the case $n=0$ is immediate, since $C_\star \ge 1$ and $C_N \ge 0$. We thus only consider $n \ge 1$.

First notice that $\Hl_\sigma^n (h \mathds{1}_Z) = 0$ if $Z \in \widehat{\mathcal{Z}}_{\sigma}^n \setminus \mathcal{Z}_{\sigma, \star}^n$. We can then write 

\[
\Hl_\sigma^n h = \sum_{Z \in \mathcal{Z}_{\sigma, \star}^n} \Hl_\sigma^n( \mathds{1}_Z h) = \sum_{Z \in \mathcal{Z}_{\sigma, \star}^n} (\mathds{1}_Z g_\sigma^n h) \circ (T_{\sigma, Z}^n)^{-1},
\]
where $ (T_{\sigma, Z}^n)^{-1}$ is the inverse branch of $T_\sigma^n$ restricted to $Z$. Accordingly, 

\[
\bigvee \Hl_\sigma^n h \le \sum_{Z \in \mathcal{Z}_{\sigma, \star}^n} \bigvee \mathds{1}_{T_\sigma^n Z} (g_\sigma^n h) \circ (T_{\sigma, Z}^n)^{-1}.
\]

We estimate each term of the sum separately.
\[
\begin{aligned}
\bigvee \mathds{1}_{T_\sigma^n Z} (g_\sigma^n h) \circ (T_{\sigma, Z}^n)^{-1} & \le \bigvee_Z h g_\sigma^n + 2 \sup_Z |h g_\sigma^n| \\ 
&\le 3 \bigvee_Z h g_\sigma^n + 2 \inf_Z |h g_\sigma^n| \\
&\le 3 \| g_\sigma^n\|_{\infty} \bigvee_Z h + 3 \sup_Z |h| \bigvee_Z g_\sigma^n + 2 \inf_Z |h g_\sigma^n| \\
&\le 3 \|g_\sigma^n\|_{\infty} \bigvee_Z h + 3C \sup_Z |h| \|g_\sigma^n\|_{\infty} + 2 \|g_\sigma^n\|_{\infty} \inf_Z |h| \\
& \le 3(C+1) \|g_\sigma^n\|_{\infty} \bigvee_Z h+ (3C +2) \|g_\sigma^n\|_{\infty} \inf_Z |h|.
\end{aligned}
\]
By assumption {\em (C3($N, N'$))}, we have for each $x \in [0,1]$,
\[
\inf_{Z \in \mathcal{Z}_{\sigma, g}^n} \frac{\Hl_\sigma^{N'} \mathds{1}_Z(x)}{\Hl_\sigma^{N'} \mathds{1}(x)}  \ge \hat{\epsilon}(N, N') \ge \epsilon_\star(N) > 0.
\]

Accordingly, for each $x \in [0,1]$, $h \in {\rm BV}$ and $Z \in \mathcal{Z}_{\sigma, g}^n$ holds  $$\Hl_\sigma^{N'} (|h| \mathds{1}_Z)(x) \ge \inf_Z |h| \Hl_\sigma^{N'} \mathds{1}_Z(x) \ge \inf_Z |h| \epsilon_\star(N) \Hl_\sigma^{N'} \mathds{1}(x).$$ 

To deal with elements in $\mathcal{Z}_{\sigma, b}^n$, we use condition {\em (C2)} which insures that elements of $Z_{\sigma, g}^n$ can be separated by at most $K \xi^n$ elements of $\mathcal{Z}_{\sigma,b}^n$. For each $Z \in \mathcal{Z}_{\sigma, b}^n$, let $I_{\pm}(Z)$ be the union of the contiguous elements of $\mathcal{Z}_{\sigma,b}^n$ on the left and on the right of $Z$ respectively. Clearly, for each $Z' \subset I_{\pm}(Z)$, holds
\[
\inf_{Z'} |h| \le \inf_Z |h| + \bigvee_{I_{\pm}(Z)} h.
\]
Accordingly, 
\[
\sum_{Z \in \mathcal{Z}_{\sigma,b}^n} \inf_Z |h| \le 2 K \xi^n \left[ \sum_{Z \in \mathcal{Z}_{\sigma,g}^n} \inf_Z |h| + \bigvee h \right].
\]
For all $x$, we thus have 
\[
\begin{aligned}
\sum_{Z \in \mathcal{Z}_{\sigma, \star}^n} \inf_Z |h| &\le (2 K \xi^n +1) \sum_{Z \in \mathcal{Z}_{\sigma,g}^n} \inf_Z |h| + 2 K\xi^n \bigvee h \\
& \le (2 K \xi^n +1) \frac{1}{\epsilon_\star(N)} \sum_{Z \in \mathcal{Z}_{\sigma,g}^n} \frac{\Hl_\sigma^{N'} (|h| \mathds{1}_Z)(x)}{\Hl_\sigma^{N'} \mathds{1}(x)} + 2 K\xi^n \bigvee h \\
& \le (2 K \xi^n +1) \frac{1}{\epsilon_\star(N)} \frac{\Hl_\sigma^{N'} (|h|)(x)}{\Hl_\sigma^{N'} \mathds{1}(x)} + 2K \xi^n \bigvee h.
\end{aligned}
\]
We can then conclude
\[
\begin{split}
\bigvee \Hl_\sigma^n h \le& \left[3(C+1) + (3C+2) 2 K \xi^n\right] \|g_\sigma^n \|_{\infty} \bigvee h \\
&+ (3C+2) (2K \xi^n +1) \|g_\sigma^n\|_{\infty} \frac{1}{\epsilon_\star(N)}  \frac{\Hl_\sigma^{N'} (|h|)(x)}{\Hl_\sigma^{N'} \mathds{1}(x)}.
\end{split}
\] 
Taking the inf over $x$ and recalling that $\|g_\sigma^n\|_{\infty} \le \Theta^n$, $ \inf_x \frac{\Hl_\sigma^{N'} (|h|)(x)}{\Hl_\sigma^{N'} \mathds{1}(x)} \le \Lambda_\sigma(|h|)$, $\xi \geq 1$ and $\xi \Theta < 1$ by {\em (C2)}, we obtain the result.
\end{proof}

We will show that the family of cones $$\mathcal{C}_\sigma^a = \{ h \in {\rm BV} \; | \; h \neq 0, h \ge 0, \bigvee h \le a \Lambda_\sigma(h)\}$$ is strictly invariant under the transfer operators defined above. 

Recall that, under assumption {\em(C2)}, $\theta = \xi \Theta < \rho$.

\begin{lem} \label{lem:iteration} For all $\sigma$ and all $g \in {\rm BV}$, $g \ge 0$, we have $\Lambda_{\tau \sigma}(\Hl_{\sigma_1} g) \ge \rho_\sigma \Lambda_\sigma (g)$. In particular, $\Lambda_{\tau^n \sigma}(\Hl_\sigma^n g) \ge \rho^n \Lambda_\sigma (g)$ and $\Lambda_{\tau^n \sigma}(\Hl_\sigma^n \mathds{1}) \ge \rho^n$.
\end{lem}

\begin{proof} If we prove the first part of the statement, the second part follows by iteration, since $\rho_\sigma \ge \rho$ for all $\sigma$. For each $g \in {\rm BV}$, $g \ge 0$ and $x \in [0,1]$,

\[
\frac{\Hl_{\tau \sigma}^n \Hl_{\sigma_1} g(x)}{\Hl_{\tau \sigma}^n \mathds{1}(x)} \ge \frac{\Hl_{\sigma_{n+1}}\left[\frac{\Hl_\sigma^n g}{\Hl_\sigma^n \mathds{1}} \Hl_\sigma^n \mathds{1} \right] (x)}{\Hl_{\tau \sigma}^n \mathds{1}(x)} \ge \frac{\Hl_{\tau \sigma}^n (\Hl_{\sigma_1} \mathds{1})(x)}{\Hl_{\tau \sigma}^n \mathds{1}(x)} \inf \frac{\Hl_\sigma^n g}{\Hl_\sigma^n \mathds{1}}
\]

and taking the inf on $x$ and the limit $n \to \infty$, we get the result.
\end{proof}

\begin{lem} \label{lem:cone_inv}  Let $n_0\geq\lceil \frac{\ln 4C_\star^2}{\ln\rho\theta^{-1}}\rceil$ such that {\em C3($n_0, N'$)} holds for some $N'\geq n_0$, then for all $a \ge a_0 = \frac{15}{11} \max_{i \le n_0} \frac{C_i}{C_\star \theta^i}$ and $\sigma$, we have 
\[
 \Hl_\sigma^n \mathcal{C}_\sigma^a \subset C_{\tau^n \sigma}^{2a C_\star} \: \: \forall n \ge 0 \: \: \text{   and   } \: \:\Hl_\sigma^n \mathcal{C}_\sigma^a \subset \mathcal{C}_ {\tau^n \sigma}^{a/2} \: \: \forall n \ge n_0.
 \]
\end{lem}
\begin{proof} Let $n_0 \in \mathbb{N}$ which will be chosen later. Let $h \in \mathcal{C}_\sigma^a$, then we can write each $n$ as $n = k n_0 + m$, $m<n_0$, and by Lemma \ref{lem:ly}, we have
\begin{equation} \label{eqn:iterate_LY}
\begin{aligned}
&\bigvee \Hl_\sigma^n h  \le C_\star \theta^{n_0} \bigvee \Hl_\sigma^{(k-1)n_0 +m} h + C_{n_0} \Lambda_{\tau^{(k-1)n_0 + m} \sigma}( \Hl_\sigma^{(k-1)n_0 + m} h) \\
& \le C_\star^k \theta^{k n_0} \bigvee \Hl_\sigma^m h + \sum_{i=0}^{k-1} C_{n_0} (C_\star \theta^{n_0})^i \Lambda_{\tau^{(k-i-1)n_0 + m} \sigma}( \Hl_\sigma^{(k-i-1)n_0 +m} h) \\
& \le C_\star^{k+1} \theta^n \bigvee h + \sum_{i=0}^{k-1} C_{n_0} (C_\star \theta^{n_0})^i  \Lambda_{\tau^{(k-i-1)n_0 + m} \sigma}(\Hl_\sigma^{(k-i-1)n_0 +m} h) \\
&\phantom{=}+ C_m (C_\star \theta^{n_0})^k \Lambda_\sigma(h).
\end{aligned}
\end{equation}
Using Lemma \ref{lem:iteration}, we obtain
\[
\bigvee \Hl_\sigma^n h \le \left[ \left(a + \frac{C_m}{C_\star \theta^m} \right) \frac{C_\star^{k+1} \theta^n}{\rho^n} +\frac{C_{n_0}}{\rho^{n_0}} \sum_{i=0}^{k-1} \left( \frac{C_\star \theta^{n_0}}{\rho^{n_0}} \right)^i \right] \Lambda_{\tau^n \sigma}(\Hl_\sigma^n h).
\]

If $k=0$ we have 
\[
\bigvee \Hl_\sigma^n h \le 2a C_\star \Lambda_{\tau^n \sigma}(\Hl_\sigma^n h).
\]
For $k>0$ let us set $\tau=\theta\rho^{-1}$, by ({\bf C2}) $\tau<1$, and let $n_0$ such that $\alpha=C_\star^2\tau^{n_0}\leq \frac 14$. Since $C_\star\geq 3$ then $C_\star\tau^{n_0}\leq\frac 1{12}$. Hence
\[
\bigvee \Hl_\sigma^n h \le \left\{\frac{1}{4} a+a_0\frac{11}{15}  \left[ \frac 14 + \frac{C_\star^2\tau^{n_0}}{C_\star-\alpha } \right] \right\}\Lambda_{\tau^n \sigma}(\Hl_\sigma^n h)\leq \frac 12 a\Lambda_{\tau^n \sigma}(\Hl_\sigma^n h).
\]
\end{proof}

\begin{lem} \label{lem:squeeze} Let $B = 1 + 2a_0 C_\star$. If {\em (C3($n_0, N'$))} holds,  then for each $h \in {\rm BV}$, $h \ge 0$, $n \in \mathbb{N}$ and $\sigma \in \Sigma^{\mathbb{N}}$, $$\Lambda_{\tau^n \sigma}(\Hl_\sigma^n \mathds{1}) \Lambda_\sigma(h) \le \Lambda_{\tau^n \sigma}(\Hl_\sigma^n h) \le B \Lambda_{\tau^n \sigma}(\Hl_\sigma^n \mathds{1}) \Lambda_\sigma(h).$$
\end{lem}
\begin{proof} For $x \in [0,1]$, we have
\[
\frac{\Hl_{\tau^n \sigma}^m (\Hl_\sigma^n h)(x)}{\Hl_{\tau^n \sigma}^m \mathds{1}(x)} \ge \frac{\Hl_{\tau^m \sigma}^n \left[\frac{\Hl_\sigma^m h}{\Hl_\sigma^m \mathds{1}} \Hl_\sigma^m \mathds{1} \right](x)}{\Hl_{\tau^n \sigma}^m \mathds{1}(x)} \ge \frac{\Hl_{\tau^n \sigma}^m (\Hl_\sigma^n \mathds{1}) (x)}{\Hl_{\tau^n \sigma}^m \mathds{1}(x)} \inf \frac{\Hl_\sigma^m h}{\Hl_\sigma^m \mathds{1}}
\]
where we have used twice the fact that $\Hl_{\sigma^n \sigma}^m \Hl_\sigma^n = \Hl_{\sigma^m \sigma}^n \Hl_\sigma^m$. Taking the inf on $x$ and the limit $m \to \infty$, we get the first inequality.
For the second, for $x \in [0,1]$, we have 
\[
\frac{\Hl_{\tau^n \sigma}^m(\Hl_\sigma^n h) (x)}{\Hl_{\tau^n \sigma}^m \mathds{1}(x)} = \frac{\Hl_{\tau^n \sigma}^m(\Hl_\sigma^n h) (x)}{\Hl_{\tau^n \sigma}^m(\Hl_\sigma^n \mathds{1})(x)} \frac{\Hl_{\tau^n \sigma}^m (\Hl_\sigma^n \mathds{1})(x)}{\Hl_{\tau^n \sigma}^m \mathds{1} (x)} \le \frac{\Hl_\sigma^{n+m} h(x)}{\Hl_\sigma^{n+m} \mathds{1}(x)} \| \Hl_\sigma^n \mathds{1} \|_\infty,
\]
which, by taking the inf on $x$ and the limit $m \to \infty$, yields $$\Lambda_{\tau^n \sigma}(\Hl_\sigma^n h) \le \| \Hl_\sigma^n \mathds{1}\|_\infty \Lambda_\sigma(h).$$
By applying Lemma \ref{lem:cone_inv} to $\mathds{1} \in \mathcal{C}_\sigma^{a_0}$, we obtain $\bigvee \Hl_\sigma^n \mathds{1} \le 2 a_0 C_\star \Lambda_{\tau^n \sigma}(\Hl_\sigma^n \mathds{1})$. Thus 
\[ 
\| \Hl_\sigma^n \mathds{1} \|_\infty \le \Lambda_{\tau^n \sigma}(\Hl_\sigma^n \mathds{1}) + \bigvee \Hl_\sigma^n \mathds{1} \le (1 + 2a_0 C_\star) \Lambda_{\tau^n \sigma}(\Hl_\sigma^n \mathds{1})
= B\Lambda_{\tau^n \sigma}(\Hl_\sigma^n \mathds{1}).
\]
\end{proof}

\begin{lem} \label{lem:partition}  Let $n_1(\delta)=\lceil\frac{\ln\delta^{-1}}{\ln\theta^{-1}\rho}\rceil$. Then, for each $\delta \in (0,1)$ and $n \ge n_1(\delta)$, the partition $\widehat{\mathcal{Z}}_\sigma^n$ has the property
\[
\sup_{Z \in \widehat{\mathcal{Z}}_\sigma^n} \Lambda_\sigma(\mathds{1}_Z) \le \delta.
\]
\end{lem}

\begin{proof}
Remark that $n_1(\delta)$ is well defined due to condition {\em (C2)} and is such that for all $n \ge n_1(\delta)$, $\theta^n \rho^{-n} \le \delta$. Then, for $Z \in \widehat{\mathcal{Z}}_\sigma^n$,
\[
\Hl_\sigma^n \mathds{1}_Z(x) = \sum_{T_\sigma^n y =x} g_\sigma^n(y) \mathds{1}_Z(y) \le \|g_\sigma^n\|_\infty \le \theta^n.
\]

Accordingly, for each $x \in [0,1]$,

\[
\frac{\Hl_{\tau^n \sigma}^m \Hl_\sigma^n \mathds{1}_Z(x)}{\Hl_{\tau^n \sigma}^m \Hl_\sigma^n \mathds{1}(x)} \le  \frac{\theta^n}{\inf\frac{\Hl_{\tau^n \sigma}^m(\Hl_\sigma^n \mathds{1})}{\Hl_{\tau^n \sigma}^m \mathds{1}}}.
\]

Taking the inf on $x$ and the limit $m \to \infty$, this yields $$\Lambda_\sigma(\mathds{1}_Z) \le  \frac{\theta^n}{\Lambda_{\tau^n \sigma}(\Hl_\sigma^n \mathds{1})} \le \theta^n \rho^{-n} \le \delta $$
where we have used Lemma \ref{lem:iteration}.
\end{proof}

\begin{lem} \label{lem:inf_partition}  Let $n_2(a)=n_1\left([4aB(1+2C_{n_0}\rho^{-n_0})]^{-1}\right)$. If {\em (C3($n_0, N'$))} holds, then for each $a \ge a_0$, $n \ge n_2$ and $h \in \mathcal{C}_\sigma^a$ there exists $Z \in \mathcal{Z}_{\sigma,g}^n$ with
\[
\inf_Z h \ge \frac 1 4 \Lambda_\sigma(h).
\]
\end{lem}
\begin{proof}
For each $n,m$ with $n <m$, we can write
\[
\Hl_\sigma^m h(x) = \sum_{Z \in \widehat{\mathcal{Z}}_\sigma^n} \Hl_\sigma^m (h \mathds{1}_Z)(x) = \sum_{Z \in \mathcal{Z}_{\sigma, \star}^n} \Hl_\sigma^m (h \mathds{1}_Z)(x).
\]
Suppose the lemma is not true. Then, we have
\[
\begin{aligned}
\Hl_\sigma^m h(x) &= \sum_{Z \in \mathcal{Z}_{\sigma,g}^n} \Hl_\sigma^m(h \mathds{1}_Z)(x) + \sum_{Z \in \mathcal{Z}_{\sigma,b}^n} \Hl_\sigma^m (h \mathds{1}_Z)(x) \\
& \le \sum_{Z \in \mathcal{Z}_{\sigma,g}^n} \Hl_\sigma^m \mathds{1}_Z(x) \frac{\Lambda_\sigma(h)}{4} +  \sum_{Z \in \mathcal{Z}_{\sigma,g}^n} \Hl_\sigma^m \mathds{1}_Z(x) \bigvee_Z h + \| h \|_\infty \sum_{Z \in \mathcal{Z}_{\sigma,b}^n} \Hl_\sigma^m \mathds{1}_Z(x) \\
& \le \Hl_\sigma^m \mathds{1}(x) \frac{\Lambda_\sigma(h)}{4} + \sum_{Z \in \mathcal{Z}_{\sigma,g}^n} \left[ \Lambda_{\tau^m \sigma}(\Hl_\sigma^m \mathds{1}_Z) + \bigvee \Hl_\sigma^m \mathds{1}_Z \right] \bigvee_Z h \\
&\phantom{=}+ \| h \|_\infty \sum_{Z \in \mathcal{Z}_{\sigma,b}^n} \Hl_\sigma^m \mathds{1}_Z(x).
\end{aligned}
\]
If $Z \in \mathcal{Z}_{\sigma,b}^n$, by Lemma \ref{lem:squeeze}, we have $\Lambda_{\tau^m \sigma}(\Hl_\sigma^m \mathds{1}_Z) \le B \Lambda_{\tau^m \sigma}(\Hl_\sigma^m \mathds{1}) \Lambda_\sigma(\mathds{1}_Z) = 0$, which implies 
\[
\Hl_\sigma^m \mathds{1}_Z (x) \le \bigvee \Hl_\sigma^m \mathds{1}_Z  \le 2 C_\star^{m/n_0 +1} \theta^m \le 2 C_\star (C_\star^{1/n_0} \theta \rho^{-1})^m \Lambda_{\tau^m \sigma}(\Hl_\sigma^m \mathds{1}),
\]
by inequality \eqref{eqn:iterate_LY} and Lemma \ref{lem:iteration}.

If $Z \in \mathcal{Z}_{\sigma,g}^n$, the same argument gives
\[
\begin{aligned}
\bigvee \Hl_\sigma^m \mathds{1}_Z & \le 2 C_\star^{m/n_0 +1} \theta^m + 2 C_{n_0} \rho^{-n_0} \Lambda_{\tau^m \sigma}(\Hl_\sigma^m \mathds{1}_Z) \\
& \le \left[2 C_\star (C_\star^{1/n_0} \theta \rho^{-1})^m + 2 C_{n_0} \rho^{-n_0} B \Lambda_\sigma(\mathds{1}_Z) \right]\Lambda_{\tau^m \sigma}(\Hl_\sigma^m \mathds{1}) .
\end{aligned}
\]
Setting $\kappa = C_\star^{1/n_0} \theta \rho^{-1} \le 4^{-1/n_0}$ and using Lemma \ref{lem:squeeze} again, we have 
\[
\begin{aligned}
\Lambda_{\tau^m\sigma}(\Hl_\sigma^m h) \le & \frac{\Lambda_\sigma(h)}{4} \Lambda_{\tau^m \sigma}(\Hl_\sigma^m \mathds{1})\\
& + \sum_{Z \in \mathcal{Z}_{\sigma,g}^n} \Lambda_{\tau^m \sigma}(\Hl_\sigma^m \mathds{1}) \bigvee_Z h \left[ B (1 + 2 C_{n_0} \rho^{-n_0})  \Lambda_\sigma( \mathds{1}_Z) + 2 C_\star \kappa^m\right] \\ &+ \|h \|_\infty \sum_{Z \in \mathcal{Z}_{\sigma,b}^n} 2 C_\star \kappa^m \Lambda_{\tau^m \sigma}(\Hl_\sigma^m \mathds{1}) .
\end{aligned}
\]
Dividing this inequality by $\Lambda_{\tau^m \sigma}(\Hl_\sigma^m \mathds{1})$ and taking the limit $m \to \infty$ yields
\[
\begin{aligned}
\Lambda_\sigma(h) & \le \frac{\Lambda_\sigma(h)}{4} + B (1 + 2C_{n_0} \rho^{-n_0}) \bigvee h \sup_{Z \in \mathcal{Z}_{\sigma,g}^n} \Lambda_\sigma(\mathds{1}_Z) \\
& \le \left[ \frac 1 4 + a B (1 + 2C_{n_0} \rho^{-n_0}) \sup_{Z \in \mathcal{Z}_{\sigma,g}^n} \Lambda_\sigma(\mathds{1}_Z) \right] \Lambda_\sigma(h) \le \frac 1 2 \Lambda_\sigma(h)
\end{aligned}
\]
where we applied Lemma \ref{lem:partition}. This yields the announced contradiction.
\end{proof}
\begin{rem} \label{rem:a_value} From now on we set $a = \max \{a_0, 1\}$.
\end{rem}

\begin{rem} \label{rem:tech_facts}
Since $\cT \times \cH$ is finite, there exists $M>1$ such that $\|\Hl_\sigma^1 \mathds{1} \|_\infty \leq M$. Hence, by condition {\em (C1)},  $\ds^n\leq \inf \Hl_\sigma^n \mathds{1}\leq  \|\Hl_\sigma^n \mathds{1} \|_\infty \le M^n$ for all $n \ge 0$ and $\sigma \in \Sigma^{\mathbb{N}}$.
\end{rem}

\begin{lem} \label{lem:inf_bounded} If {\em (C3($n_2, n_3$))} holds for some $n_3 \ge n_2$, then there exists $\exalpha > 0$ such that
\[
\inf \Hl_\sigma^n f \ge \exalpha \Lambda_{\tau^n \sigma}(\Hl_\sigma^n f),
\]
for all $n\ge n_3$, all $\sigma \in \Sigma^{\mathbb{N}}$ and all $f \in C_\sigma^a$.
\end{lem}

\begin{proof} We first remark that, by Lemma \ref{lem:inf_partition}, for each $f \in C_\sigma^a$, there exists $Z \in \mathcal{Z}_{\sigma, g}^{n_2}$ such that 
\[
\inf_Z f \ge \frac 1 4 \Lambda_\sigma(f).
\]
Consequently, for any $\ell \ge 0$, we have 
\[
\inf \Hl_\sigma^\ell f \ge \frac 1 4 \Lambda_\sigma(f) \inf \frac{\Hl_\sigma^\ell \mathds{1}_Z}{\Hl_\sigma^\ell \mathds{1}} \inf \Hl_\sigma^\ell \mathds{1}.
\]
Since condition {\em (C3($n_2,n_3$))} holds, one has
\[
\inf \frac{\Hl_\sigma^{\ell} \mathds{1}_Z}{\Hl_\sigma^{\ell} \mathds{1}} \ge \epsilon_\star(n_2)
\]
for all $\sigma \in \Sigma^{\mathbb{N}}$, $Z  \in \mathcal{Z}_{\sigma,g}^{n_2}$ and $\ell \ge n_3$, and we thus get
\begin{equation} \label{eqn:inf_L}
\inf \Hl_\sigma^{\ell} f \ge \frac{\epsilon_\star(n_2)}{4} \Lambda_\sigma(f) \inf \Hl_\sigma^{\ell} \mathds{1}.
\end{equation}
We first prove Lemma \ref{lem:inf_bounded} when $n=n_3$ and then extend it to all $n \ge n_3$.
By Lemma \ref{lem:squeeze}, 
\[
\Lambda_\sigma(f) \ge B^{-1} \frac{\Lambda_{\tau^{n_3}\sigma}(\Hl_\sigma^{n_3} f)}{\Lambda_{\tau^{n_3}\sigma}(\Hl_\sigma^{n_3} \mathds{1})} \ge B^{-1} \frac{\Lambda_{\tau^{n_3}\sigma}(\Hl_\sigma^{n_3} f)}{\sup \Hl_\sigma^{n_3} \mathds{1}}.
\]
By Remark \ref{rem:tech_facts} and \eqref{eqn:inf_L} with $\ell=n_3$, we obtain
\[
\inf \Hl_\sigma^{n_3} f \ge \hat\exalpha \Lambda_{\tau^{n_3} \sigma}(\Hl_\sigma^{n_3} f),
\]
with $\hat\exalpha = B^{-1} \frac{\epsilon_\star(n_2)}{4} (\frac{\ds}{M})^{n_3}$.

By Lemma \ref{lem:cone_inv}, $\Hl_\sigma^{n_0} C_\sigma^a \subset C_{\tau^{n_0} \sigma}^a$ for all $\sigma$. We thus have
\[
\inf \Hl_\sigma^{n_3 + kn_0} f \ge \hat\exalpha \Lambda_{\tau^{n_3 + kn_0} \sigma}(\Hl_\sigma^{n_3 + kn_0} f),
\]
for all $k \ge 0$, $\sigma \in \Sigma^{\mathbb{N}}$ and $f \in C_\sigma^a$.

Let now $n \ge n_3$. We write $n = kn_0 + n_3 +r = n' + r$ with $r < n_0$. We have
\[
\begin{aligned}
\inf \Hl_\sigma^n f = \Hl_{\tau^{n'} \sigma}^r \Hl_\sigma^{n'} f & \ge \hat\exalpha \Lambda_{\tau^{n'} \sigma}(\Hl_\sigma^{n'}f) \inf \Hl_{\tau^{n'}\sigma}^r \mathds{1} \\ & \ge \hat\exalpha \ds^r \Lambda_{\tau^{n'} \sigma}(\Hl_\sigma^{n'}f) \\ & \ge\hat\exalpha \ds^{n_0} \Lambda_{\tau^{n'} \sigma}(\Hl_\sigma^{n'}f).
\end{aligned}
\]

But, 
\[
\begin{aligned}
\Lambda_{\tau^{n'} \sigma}(\Hl_\sigma^{n'}f) &= \lim_{k \to \infty} \inf \frac{\Hl_{\tau^{n'+r}\sigma}^k \Hl_{\tau^{n'} \sigma}^r \Hl_{\sigma}^{n'} f}{\Hl_{\tau^{n'+r}\sigma}^k \Hl_{\tau^{n'}\sigma}^r \mathds{1}} \\
& \ge M^{-r} \lim_{k \to \infty} \inf \frac{\Hl_{\tau^{n'+r} \sigma}^k \Hl_\sigma^{n'+r} f}{\Hl_{\tau^{n'+r} \sigma}^k \mathds{1}} \\
& = M^{-r} \Lambda_{\tau^{n'+r} \sigma}(\Hl_\sigma^{n'+r} f) \\ & \ge M^{-n_0} \Lambda_{\tau^n \sigma}(\Hl_\sigma^n f).
\end{aligned}
\]
We have thus proved the result with $\exalpha = \hat\exalpha (\frac{\ds}{M})^{n_0}$.
\end{proof}
\begin{cor}\label{cor:rho} It holds true $\rho\leq 1$, where $\rho$ is defined in \eqref{eq:rho-def}.\footnote{ Hence the min in condition $(\bf{C2})$ was, a posteriori, superfluous.}
\end{cor}
\begin{proof}
By Lemma \ref{lem:iteration} and Lemma  \ref{lem:inf_bounded}, with $f=1$, we see that, for all $n\geq n_3$, $\inf \widehat\cL^n_\sigma 1\geq \varrho \rho^n$ but $\inf \widehat\cL^n_\sigma 1\leq \int \widehat\cL^n_\sigma 1\leq 1$, so $\rho\leq 1$.
\end{proof}

\begin{rem} \label{rem:inf_bounded1}
Since by Remark \ref{rem:tech_facts}, $ \| \Hl_\sigma^n \mathds{1} \|_\infty \le M^n$ and $\inf \Hl_\sigma^n \mathds{1} \ge \ds^n$ for all $n \ge 0$ and $\sigma \in \Sigma^{\mathbb{N}}$, we have 
\[
\inf \Hl_\sigma^n \mathds{1} \ge \exalpha' \Lambda_{\tau^n \sigma} (\Hl_\sigma^n \mathds{1}),
\]
with $\exalpha' = \min \{\exalpha, 1, \frac{\ds}{M}, \ldots, (\frac{\ds}{M})^{n_3 - 1} \}$.
\end{rem}
\begin{lem} \label{lem:finite_diameter} 
If {\em (C3($n_2, n_3$))} holds for some $n_3 \ge n_2$, then for all $\sigma \in \Sigma^{\mathbb{N}}$ and $n \ge n_3$, one has $$\Hl_\sigma^n \mathcal{C}_\sigma^a \subset C_{\tau^n \sigma}^a$$ with finite diameter less than $\Delta_n$, uniformly in $\sigma$. 
\end{lem}

\begin{proof}
By \eqref{eqn:inf_L}, we have
\[
\inf \Hl_\sigma^n h \ge \frac {\epsilon_\star(n_2)} 4 \Lambda_\sigma(h)  \inf \Hl_\sigma^n \mathds{1},
\]
for any $\sigma \in \Sigma^{\mathbb{N}}$, $h \in C_\sigma^a$ and $n \ge n_3$ and, using Lemmata \ref{lem:cone_inv} and \ref{lem:squeeze},
\[
\sup \Hl_\sigma^n h \le \Lambda_{\tau^n \sigma}(\Hl_\sigma^n h) + \bigvee \Hl_\sigma^n h \le \left(1 + \frac a 2\right) B \Lambda_\sigma(h) \Lambda_{\tau^{n} \sigma}(\Hl_\sigma^n \mathds{1}).
\]
 A simple adaptation of the proof of \cite[Lemma 3.1]{Liverani1995} yields
\[
d_{\mathcal{C}_\sigma^a}(g, \mathds{1}) \le \log \left[\frac{\max \left\{(1+\nu) \Lambda_\sigma(g), \sup g \right\}}{\min\left\{(1- \nu) \Lambda_\sigma(g), \inf g\right\}} \right]
\]
for any $g \in \mathcal{C}_{\sigma}^{\nu a}$ with $0<\nu<1$. By Lemma \ref{lem:cone_inv} $\Hl_\sigma^n (\mathcal{C}_\sigma^a) \subset \mathcal{C}_{\tau^n \sigma}^{a/2}$, so we can choose $\nu = \frac 1 2$ and, recalling Remark \ref{rem:tech_facts}, we obtain:
\[
{\rm diam}_{\mathcal{C}_{\tau^n \sigma}^a} \Hl_\sigma^n (\mathcal{C}_\sigma^a) \le 2 \log \left[ \frac{\max\left\{\frac 3 2 , B M^n (1 + \frac a 2)\right\}}{\min \left\{ \frac 1 2 , \frac{{\epsilon_\star(n_2)} \ds^n}{4}\right\}} \right] =: \Delta_n < \infty
\]
for any $\sigma \in \Sigma^{\mathbb{N}}$.
\end{proof}

Since we are interested in functions of the form $\Hl_{\tau^{-n} \sigma}^n f$, we will need to consider functions $f$ that belong to the intersections of all the cones $C_\sigma^a$, $\sigma \in \Sigma^{\bZ}$. For this purpose, we introduce the family of cones
\[
C_{\inf}^a = \{ f \in {\rm BV} \, : \, f \neq 0, \, f \ge 0, \, \bigvee f \le a \inf f \}.
\]
We have $C_{\inf}^a \subset C_{\sigma}^a$ for any $\sigma \in \Sigma^{\bZ}$, and thus $d_{C_\sigma^a} \le d_{C_{\inf}^a}$ by Theorem \ref{thm:contraction_cone}.

\begin{lem} \label{lem:quasiinv_density}
There exist $\nu  \in (0,1)$ and a family of positive functions $\{ h_\sigma \}_{\sigma  \in \Sigma^{\bZ}}$ in ${\rm BV}$ such that for all $f \in C_{\inf}^a$, $\sigma \in \Sigma^{\bZ}$ and $n \ge 0$:
\[
\left\| \frac{\Hl_{\tau^{-n} \sigma}^n f}{\Lambda_\sigma(\Hl_{\tau^{-n} \sigma}^n f)} - h_\sigma \right\|_\infty \le \Const \nu^n,
\]
and
\[
\left\| \frac{\Lambda_\sigma(\Hl_{\tau^{-n} \sigma}^n f)}{\Hl_{\tau^{-n} \sigma}^n f} - h_\sigma^{-1} \right\|_\infty \le \Const \nu^n.
\]
Furthermore, $h_\sigma \in C_\sigma^{a/2}$, $\|h_\sigma\|_\infty \le 1+a$ and $\inf h_\sigma \ge \exalpha > 0$ for all $\sigma \in \Sigma^{\bZ}$, where $\exalpha$ is defined in Lemma \ref{lem:inf_bounded}.
\end{lem}

\begin{proof}
Writing $n \ge 2 n_3$ as $n = k n_3 + r$, with $k \ge 2$ and $r < n_3$, for all $f \in C_{\inf}^a$, $\sigma \in \Sigma^{\bZ }$ and $m \ge 0$,  we have by Theorem \ref{thm:contraction_cone} and Lemma \ref{lem:finite_diameter}:
\[
d_{C_\sigma^a}(\Hl_{\tau^{-n} \sigma}^n f, \Hl_{\tau^{-(n+m)} \sigma}^{n+m} f) \le \gamma^{k-2} d_{C_{\tau^{-(k-2)n_3} \sigma}^a} ( \Hl_{\tau^{-n} \sigma}^{2n_3 + r} f, \Hl_{\tau^{-(n+m)} \sigma}^{2n_3 + r + m} f),
\]
with $\gamma = \tanh \left( \frac{\Delta_{n_3}}{4} \right) <1$.

Since both $\Hl_{\tau^{-n} \sigma}^{n_3 + r}f$ and $\Hl_{\tau^{-(n+m)} \sigma}^{n_3 + r + m} f$ belong to $C_{\tau^{-(k-1) n_3} \sigma}^a$ by Lemma \ref{lem:finite_diameter} again, as $f \in C_{\inf}^a \subset C_{\tau^{-n} \sigma}^a \cap C_{\tau^{-(n+m)} \sigma}^a$, we have, using Lemma \ref{lem:finite_diameter} one more time:
\[
d_{C_{\tau^{-(k-2)n_3} \sigma}^a} ( \Hl_{\tau^{-n} \sigma}^{2n_3 + r} f, \Hl_{\tau^{-(n+m)} \sigma}^{2n_3 + r + m} f) \le \Delta_{n_3}.
\]
Consequently, for all $n\ge 2 n_3$, $m \ge 0$, $\sigma \in \Sigma^{\bZ}$ and $f \in C_{\inf}^a$,
\[
d_{C_\sigma^a}(\Hl_{\tau^{-n} \sigma}^n f, \Hl_{\tau^{-(n+m)} \sigma}^{n+m} f) \le \Const \nu^n,
\]
with $\nu = \gamma^{\frac{1}{n_3}}$.

Using Lemma \ref{lem:link_hilbert} with $\| \cdot \| = \| \cdot \|_\infty$ and $ \ell( \cdot ) = \Lambda_\sigma ( \cdot )$, we get
\[
\left\| \frac{\Hl_{\tau^{-n} \sigma}^n f}{\Lambda_\sigma(\Hl_{\tau^{-n} \sigma}^n f)} - \frac{\Hl_{\tau^{-(n+m)} \sigma}^{n+m} f}{\Lambda_\sigma(\Hl_{\tau^{-(n+m)} \sigma}^{n+m} f)} \right\|_\infty \le (e^{d_{C_\sigma^a}(\Hl_{\tau^{-n} \sigma}^n f, \Hl_{\tau^{-(n+m)} \sigma}^{n+m} f)} - 1) \left\| \frac{\Hl_{\tau^{-n} \sigma}^n f}{\Lambda_\sigma(\Hl_{\tau^{-n} \sigma}^n f)} \right\|_\infty.
\]
Since $\Hl_{\tau^{-n} \sigma}^n f \in C_\sigma^a$, we have 
\begin{equation} \label{eqn:sup_hsigma}
\| \Hl_{\tau^{-n} \sigma}^n f \|_\infty \le \Lambda_\sigma( \Hl_{\tau^{-n} \sigma}^n f ) + \bigvee \Hl_{\tau^{-n} \sigma}^n f \le (1+a) \Lambda_\sigma(\Hl_{\tau^{-n} \sigma}^n f),
\end{equation}
and we deduce that
\[
\left\| \frac{\Hl_{\tau^{-n} \sigma}^n f}{\Lambda_\sigma(\Hl_{\tau^{-n} \sigma}^n f)} - \frac{\Hl_{\tau^{-(n+m)} \sigma}^{n+m} f}{\Lambda_\sigma(\Hl_{\tau^{-(n+m)} \sigma}^{n+m} f)} \right\|_\infty \le \Const \nu^n.
\]
This implies that $\frac{\Hl_{\tau^{-n} \sigma}^n f}{\Lambda_\sigma(\Hl_{\tau^{-n} \sigma}^n f)}$ is a Cauchy sequence in $L^\infty$, and thus converges to a function $h_\sigma \in L^\infty$. Since $C_{\sigma}^{a/2}$ is closed in $L^\infty$ and $\Hl_{\tau^{-n} \sigma}^n f \in C_\sigma^{a/2}$ for $n \ge n_0$ by Lemma \ref{lem:cone_inv}, we have $h_\sigma \in C_{\sigma}^{a/2}$. Passing to the limit $m \to \infty$ in the previous relation, we obtain
\[
\left\| \frac{\Hl_{\tau^{-n} \sigma}^n f}{\Lambda_\sigma(\Hl_{\tau^{-n} \sigma}^n f)} - h_\sigma \right\|_\infty \le \Const \nu^n,
\]
for all $n \ge 2 n_3$, and all $\sigma \in \Sigma^{\bZ}$.
Using the same reasoning, we have for any pair $f, f' \in C_{\tau^{-n} \sigma}^a$, 
\begin{equation} \label{eqn:loss_memory}
\left\| \frac{\Hl_{\tau^{-n} \sigma}^n f}{\Lambda_\sigma(\Hl_{\tau^{-n} \sigma}^n f)} - \frac{\Hl_{\tau^{-n} \sigma}^{n} f'}{\Lambda_\sigma(\Hl_{\tau^{-n} \sigma}^{n} f')} \right\|_\infty \le \Const \nu^n,
\end{equation}
which proves that the limit $h_\sigma$ does not depend on the choice of $f \in C_{\inf}^a$.
By Lemma \ref{lem:inf_bounded}, $\Hl_{\tau^{-n} \sigma}^n f \ge \exalpha \Lambda_\sigma( \Hl_{\tau^{-n} \sigma}^n f)$ for all $n \ge n_3$, whence we obtain $\inf h_\sigma \ge \exalpha$. Remark that \eqref{eqn:sup_hsigma} implies $\| h_\sigma \|_\infty \le 1 + a$.
When $n < 2 n_3$, we have
\[
\left\| \frac{\Hl_{\tau^{-n} \sigma}^n f}{\Lambda_\sigma(\Hl_{\tau^{-n} \sigma}^n f)} - h_\sigma \right\|_\infty \le \frac{M^{2 n_3}}{\ds^{2 n_3}} \frac{\sup f}{\inf f} + (1 + a) \le (1+a) \left( \frac{M^{2 n_3}}{\ds^{2 n_3}} + 1 \right) \le \Const \nu^n,
\]
where $M > 1$ and $0 < \ds < 1$ are defined in Remark \ref{rem:tech_facts}.
For $n \ge n_3$ and $f \in C_{\inf}^a$, since $\inf h_\sigma \ge \exalpha$ and $\inf \frac{\Hl_{\tau^{-n} \sigma}^n f}{\Lambda_\sigma(\Hl_{\tau^{-n} \sigma}^n f)} \ge \exalpha$ by Lemma \ref{lem:inf_bounded}, we have
\[
\left\| \frac{\Lambda_\sigma(\Hl_{\tau^{-n} \sigma}^n f)}{\Hl_{\tau^{-n} \sigma}^n f} - h_\sigma^{-1} \right\|_\infty \le \exalpha^{-2} \left\| \frac{\Hl_{\tau^{-n} \sigma}^n f}{\Lambda_\sigma(\Hl_{\tau^{-n} \sigma}^n f)} - h_\sigma \right\|_\infty \le \Const \nu^n.
\]
We handle the case $n < n_3$ as previously, since $\| h_\sigma \|_\infty \le 1+a$.
\end{proof}

\begin{lem}  \label{lem:eigenvalue}
For all $\sigma  \in \Sigma^{\bZ}$, there exists $\lambda_\sigma \ge \rho_\sigma$ such that $\Hl_\sigma^n h_\sigma = \lambda_\sigma \cdots \lambda_{\tau^{n-1} \sigma} h_{\tau^n \sigma}$ for all $n \ge 1$.
\end{lem}

\begin{proof} Applying Lemma \ref{lem:squeeze} with $h = \Hl_{\tau^{-n} \sigma}^n \mathds{1}$, we have by definition of $\rho_\sigma = \Lambda_{\tau \sigma}(\Hl_\sigma^1 \mathds{1})$:
\[
\rho_\sigma \le \frac{\Lambda_{\tau \sigma}(\Hl_{\tau^{-n}\sigma}^{n+1} \mathds{1})}{\Lambda_\sigma(\Hl_{\tau^{-n } \sigma}^n \mathds{1})} \le B \rho_\sigma.
\]
Consequently, there exist a subsequence $\{ n_j \}$ and $\lambda_{\sigma} \in [\rho_\sigma, B \rho_\sigma]$ such that 
\[
\lambda_{\sigma} = \lim_j \frac{\Lambda_{\tau \sigma}(\Hl_{\tau^{-n_j}\sigma}^{n_j+1} \mathds{1})}{\Lambda_\sigma(\Hl_{\tau^{-n_j } \sigma}^{n_j} \mathds{1})}.
\]
We can now compute
\[
\begin{aligned}
\Hl_\sigma^1 h_\sigma = \lim_j \Hl_\sigma^1 \frac{\Hl_{\tau^{- n_j} \sigma}^{n_j} \mathds{1}}{\Lambda_\sigma(\Hl_{\tau^{- n_j} \sigma}^{n_j} \mathds{1})} & = \lim_j \frac{\Hl_{\tau^{-n_j} \sigma}^{n_j +1} \mathds{1}}{\Lambda_{\tau \sigma}(\Hl_{\tau^{-n_j} \sigma}^{n_j +1} \mathds{1})} \frac{\Lambda_{\tau \sigma}(\Hl_{\tau^{-n_j} \sigma}^{n_j +1} \mathds{1})}{\Lambda_\sigma(\Hl_{\tau^{- n_j} \sigma}^{n_j} \mathds{1})} \\
& = \lim_j \frac{\Hl_{\tau^{-(n_j+1)} \tau \sigma}^{n_j +1} \mathds{1}}{\Lambda_{\tau \sigma}(\Hl_{\tau^{-(n_j+1)} \tau \sigma}^{n_j +1} \mathds{1})} \frac{\Lambda_{\tau \sigma}(\Hl_{\tau^{-n_j} \sigma}^{n_j +1} \mathds{1})}{\Lambda_\sigma(\Hl_{\tau^{- n_j} \sigma}^{n_j}\mathds{1})} \\
& = \lambda_{\sigma} h_{\tau \sigma}.
\end{aligned}
\]\
The general case $n \ge 1$ is obtained by a simple induction.
\end{proof}

\begin{lem} \label{lem:lambda_linear}
For all $\sigma \in \Sigma^{\mathbb{N}}$, the functional $\Lambda_\sigma$ (restricted to ${\rm BV}$) is linear, positive, and enjoys the property 
$\Lambda_{\tau^n \sigma}(\Hl_\sigma^n f) = \rho_\sigma \cdots \rho_{\tau^{n-1} \sigma} \Lambda_\sigma(f)$ for all $f \in {\rm BV}$ and $n \ge 1$. Moreover, $\lambda_\sigma = \rho_\sigma$ and $\left\| \frac{\Hl_\sigma^n f}{\Hl_\sigma^n \mathds{1}} - \Lambda_\sigma(f) \right\|_\infty  \le \Const \nu^n \| f \|_{\rm BV}$ for all $f \in {\rm BV}$.
\end{lem}

\begin{proof}
For $f \in C_{\inf}^a$, we can write
\[
\frac{\Hl_\sigma^n f}{\Hl_\sigma^n \mathds{1}} = \frac{\Hl_\sigma^n f}{\Lambda_{\tau^n \sigma}(\Hl_\sigma^n f)} \frac{\Lambda_{\tau^n \sigma}(\Hl_\sigma^n f)}{\Lambda_{\tau^n \sigma}(\Hl_\sigma^n \mathds{1})} \frac{\Lambda_{\tau^n \sigma}(\Hl_\sigma^n \mathds{1})}{\Hl_\sigma^n \mathds{1}}.
\]
So
\[
\begin{aligned}
&\left\| \frac{\Hl_\sigma^n f}{\Hl_\sigma^n \mathds{1}} - \frac{\Lambda_{\tau^n \sigma}(\Hl_\sigma^n f)}{\Lambda_{\tau^n \sigma}(\Hl_\sigma^n \mathds{1})} \right\|_\infty = \frac{\Lambda_{\tau^n \sigma}(\Hl_\sigma^n f)}{\Lambda_{\tau^n \sigma}(\Hl_\sigma^n \mathds{1})} \left\| \frac{\Hl_\sigma^n f}{\Lambda_{\tau^n \sigma}(\Hl_\sigma^n f)} \frac{\Lambda_{\tau^n \sigma}(\Hl_\sigma^n \mathds{1})}{\Hl_\sigma^n \mathds{1}} - 1 \right\|_\infty \\
 & \le \|f\|_\infty \left( \left\|  \frac{\Hl_\sigma^n f}{\Lambda_{\tau^n \sigma}(\Hl_\sigma^n f)} - h_\sigma \right\|_\infty \left\| \frac{\Lambda_{\tau^n \sigma}(\Hl_\sigma^n \mathds{1})}{\Hl_\sigma^n \mathds{1}} \right\|_\infty + \|h_\sigma\|_\infty \left\| \frac{\Lambda_{\tau^n \sigma}(\Hl_\sigma^n \mathds{1})}{\Hl_\sigma^n \mathds{1}} - h_\sigma^{-1} \right\|_\infty \right).
\end{aligned}
\]
Since $\left\| \frac{\Lambda_{\tau^n \sigma}(\Hl_\sigma^n \mathds{1})}{\Hl_\sigma^n \mathds{1}} \right\|_\infty \le \exalpha^{-1}$ for $n \ge n_3$ by Lemma \ref{lem:inf_bounded}, we get, using Lemma \ref{lem:quasiinv_density}, for all $f \in C_{\inf}^a$ and $n \ge n_3$:
\begin{equation} \label{eqn:limit_lambda}
\left\| \frac{\Hl_\sigma^n f}{\Hl_\sigma^n \mathds{1}} - \frac{\Lambda_{\tau^n \sigma}(\Hl_\sigma^n f)}{\Lambda_{\tau^n \sigma}(\Hl_\sigma^n \mathds{1})} \right\|_\infty \le \Const \nu^n \|f \|_\infty.
\end{equation}
But, $\Lambda_\sigma(f) = \lim_{n\to\infty} \inf \frac{\Hl_\sigma^n f}{\Hl_\sigma^n \mathds{1}}$ by definition, and, since $\frac{\Lambda_{\tau^n \sigma}(\Hl_\sigma^n f)}{\Lambda_{\tau^n \sigma}(\Hl_\sigma^n \mathds{1})}$ are constants, we deduce that $\lim_{n\to\infty} \frac{\Lambda_{\tau^n \sigma}(\Hl_\sigma^n f)}{\Lambda_{\tau^n \sigma}(\Hl_\sigma^n \mathds{1})} = \Lambda_\sigma(f)$ and 
\[
\lim_{n \to \infty} \left\| \frac{\Hl_\sigma^n f}{\Hl_\sigma^n \mathds{1}} - \Lambda_\sigma(f) \right\|_\infty = 0.
\]
Now, if $f \in {\rm BV}$, we have $f + c \in C_{\inf}^a$ for $c = (1+a^{-1}) \|f \|_{\rm BV}$, so we get that $\Lambda_\sigma(f) = \lim_{n\to\infty} \frac{\Hl_\sigma^n f}{\Hl_\sigma^n \mathds{1}}$ in $L^\infty$ for all $f \in {\rm BV}$, since $\Lambda_\sigma(f + c) = \Lambda_\sigma(f) + c$. The linearity of $\Lambda$ follows from the linearity of the limit.

Next, as $\Hl_\sigma^1 f \in {\rm BV}$, we know that
\[
\Lambda_{\tau \sigma} (\Hl_\sigma^1 f) = \lim_{n \to \infty} \frac{\Hl_\sigma^{n+1} f}{\Hl_{\tau \sigma}^n \mathds{1}} = \lim_{n \to \infty} \frac{\Hl_\sigma^{n+1} f}{\Hl_\sigma^{n+1} \mathds{1}} \frac{\Hl_\sigma^{n+1} \mathds{1}}{\Hl_{\tau \sigma}^n \mathds{1}} = \Lambda_\sigma(f) \Lambda_{\tau \sigma}(\Hl_\sigma^1 \mathds{1}) = \rho_\sigma \Lambda_\sigma(f).
\]
But then $\rho_\sigma = \lambda_\sigma$ is obtained by taking $f = h_\sigma$, since $\Lambda_\sigma(h_\sigma) = 1$.
In particular, we have $\Lambda_{\tau^n \sigma} ( \Hl_\sigma^n f) = \rho_\sigma \cdots \rho_{\tau^{n-1} \sigma} \Lambda_\sigma(f)$ for all $f \in {\rm BV}$, so $\frac{\Lambda_{\tau^n \sigma}(\Hl_\sigma^n f)}{\Lambda_{\tau^n \sigma}(\Hl_\sigma^n \mathds{1})} = \Lambda_\sigma(f)$. But if we look back to \eqref{eqn:limit_lambda}, the above implies that for all $f \in C_{\inf}^a$ and $n \ge n_3$:
\[
\left\| \frac{\Hl_\sigma^n f}{\Hl_\sigma^n \mathds{1}} - \Lambda_\sigma(f) \right\|_\infty \le \Const \nu^n \| f \|_\infty.
\]
This can be easily extended to all $n \ge 0$, since $\left\| \frac{\Hl_\sigma^n f}{\Hl_\sigma^n \mathds{1}} - \Lambda_\sigma(f) \right\|_\infty \le 2 \|f \|_\infty$. We can again cover the general case $f \in {\rm BV}$ using the fact that $f + c \in C_{\inf}^a$ for $c = (1 + a^{-1}) \|f \|_{\rm BV}$, which finally implies
\[
\left\| \frac{\Hl_\sigma^n f}{\Hl_\sigma^n \mathds{1}} - \Lambda_\sigma(f) \right\|_\infty \le \Const \nu^n \| f + c \|_\infty \le \Const \nu^n \|f \|_{\rm BV}.
\]
\end{proof}

\begin{rem} Following closely the ideas of \cite{liverani2003lasota}, it is possible to prove that $\Lambda_\sigma$ can be interpreted as a non-atomic measure $\mu_\sigma$, i.e. $\Lambda_\sigma(f) = \int f \, d \mu_\sigma$ for all $f \in {\rm BV}$, and that the measure $\nu_\sigma$ defined by $d \nu_\sigma = h_\sigma d \mu_\sigma$ satisfies $(T_\sigma^1)_\star \nu_\sigma = \nu_{\tau \sigma}$. Since we will not make use of these facts, we leave their proofs to the interested reader.
\end{rem}

The main properties of $\Lambda_\sigma$ being proved, we can now improve Lemma \ref{lem:quasiinv_density} by extending it to general functions in ${\rm BV}$ and deduce Proposition \ref{prop:decay}:
\begin{proof}[{\bf\emph{Proof of Proposition \ref{prop:decay}}}]
By \eqref{eqn:loss_memory} with $f' = h_\sigma$, for any $f \in C_\sigma^a$, we get using Lemma \ref{lem:lambda_linear}
\[
\left\| \frac{\Hl_\sigma^n f}{\rho_\sigma \cdots \rho_{\tau^{n-1} \sigma}} - \Lambda_{\sigma}(f) h_{\tau^n \sigma} \right\|_\infty \le \Const \nu^n \Lambda_{\sigma} (f) \le \Const \nu^n \|f \|_\infty.
\]
Now, if $f \in {\rm BV}$, we have $f + c h_\sigma \in C_{\sigma}^a$ for all $\sigma \in \Sigma^{\bZ}$ with $c = 2(1+a^{-1}) \|f \|_{\rm BV}$. Indeed, since $\bigvee h_\sigma \le \frac{a}{2} \Lambda_\sigma(h_\sigma) = \frac{a}{2}$ by Lemma \ref{lem:quasiinv_density}, we have
\[
\bigvee (f + c h_\sigma) \le \bigvee f + c \bigvee h_\sigma \le \bigvee f + \frac{ac}{2},
\]
and
\[
\Lambda_\sigma(f + c h_\sigma) = \Lambda_\sigma(f) + c \ge \inf f + c.
\]
So $f + c h_\sigma \in C_\sigma^a$ if $c \ge 2( a^{-1} \bigvee f - \inf f)$, which is the case for our particular choice of $c$.
Consequently, we have
\[
\left\| \frac{\Hl_\sigma^n (f+c h_\sigma)}{\rho_\sigma \cdots \rho_{\tau^{n-1} \sigma}} - \Lambda_{\sigma}(f+c h_\sigma) h_{\tau^n \sigma} \right\|_\infty \le \Const \nu^n \| f + c h_\sigma \|_\infty \le \Const \nu^n \| f \|_{\rm BV},
\]
which leads to \eqref{eqn:exp_convergence} after simplifications, since 
\[
\Hl_\sigma^n(f+c h_\sigma) = \Hl_\sigma^nf + c \rho_\sigma \cdots \rho_{\tau^{n-1} \sigma} h_{\tau^n \sigma}
\]
and $\Lambda_{\sigma}(f + c h_\sigma) = \Lambda_{\sigma} (f) + c$.
\end{proof}

\begin{proof}[{\bf\emph{Proof of Lemma \ref{lem:property_inv_density}}}] By Remark \ref{rem:tech_facts}, we have $\rho_\sigma \le \| \Hl_\sigma^1 \mathds{1} \|_\infty \le M$, and so $\sup_\sigma \rho_\sigma \le M < \infty$.
The remaining statements are immediate consequences of Lemmata \ref{lem:quasiinv_density}, \ref {lem:eigenvalue} and \ref{lem:lambda_linear}.
\end{proof}

\begin{proof}[{\bf\emph{Proof of Lemma \ref{lem:tech_lemma}}}]
This follows immediately from Lemma \ref{lem:squeeze} and Remark \ref{rem:inf_bounded1}.
\end{proof}

\bibliographystyle{abbrv} \bibliography{rwre}

\end{document}